\documentclass[11pt]{article}

\usepackage{fullpage}
\usepackage{makeidx}
\usepackage{amstext}
\usepackage{amssymb}
\usepackage{amsmath}
\usepackage{amsthm}
\usepackage[cp866]{inputenc}

\usepackage[dvips]{graphicx}        

\newtheorem{prepos}{Proposition}

\theoremstyle{plain}
\newtheorem{theorem}{Theorem}[section]
\newtheorem{lemma}[theorem]{Lemma}
\newtheorem{corol}[theorem]{Corollary}
\newtheorem{propos}[theorem]{Proposition}

\theoremstyle{definition}
\newtheorem{definition}[theorem]{Definition}
\newtheorem{note}[theorem]{Remark}

\newtheorem{example}[theorem]{Example}
\newtheorem{notation}[theorem]{Notation}

\renewcommand{\Im}{\operatorname{Im}}

\def\sgn{\mathop{\rm sign}\nolimits}

\begin{document}
\title{On the number of real critical points of logarithmic
derivatives and the Hawaii conjecture}

\author{Mikhail Tyaglov\thanks{The work was supported by the
Sofja Kovalevskaja Research Prize of Alexander von Humboldt
Foundation.
Email: {\tt tyaglov@math.tu-berlin.de}}  \\
\small Institut f\"ur Mathematik,  Technische Universit\"at Berlin
}

\maketitle

\vspace{4mm}

\begin{flushright}
\textit{Dedicated with gratitude to Thomas Craven,\\
George Csordas, and Wayne Smith.}
\end{flushright}

\vspace{2mm}

\begin{abstract}
For a given real entire function $\varphi$ with finitely many
nonreal zeroes in the class $U_{2n}^*$, $n\geqslant0$, we
establish a connection between the number of real zeroes of the
functions $Q[\varphi]=(\varphi'/\varphi)'$ and
$Q_1[\varphi]=(\varphi''/\varphi')'$. This connection leads to a
proof of the Hawaii conjecture~[T.\,Craven, G.\,Csordas, and
W.\,Smith\/, The zeroes of derivatives of entire functions and the
P\'olya-Wiman conjecture, Ann. of Math.~(2) 125 (1987), 405--431]
stating that the number of real zeroes of $Q[\varphi]$ does not
exceed the number of nonreal zeroes of $\varphi$ if $\varphi$ is a
real polynomial.
\end{abstract}

%

\section*{Introduction}

In this paper, we investigate the real critical points of
logarithmic derivatives of real entire function, $\varphi(z)$, in
the class $\mathcal{L-P^*}$ (the class of functions of the form
$f(z)= e^{az^2}g(z)$ where $a>0$ and $g$ is a real entire function
of genus at most $1$, see also
Definition~\ref{prop.diff.U_2n*.function}). For a given real
entire function~$\varphi$, the derivative of its logarithmic
derivative is as follows
\begin{equation*}
Q[\varphi](z)\stackrel{def}{=}\dfrac{d}{dz}\left(\dfrac{\varphi'(z)}{\varphi(z)}\right)=
\dfrac{\varphi(z)\varphi''(z)-\left(\varphi'(z)\right)^2}{\left(\varphi(z)\right)^2},\quad\text{\textrm{where}}\quad
\varphi'(z)=\dfrac{d\varphi(z)}{dz}~.
\end{equation*}
So we are interested in the distribution of real zeroes of this
function.

In~\cite{CravenCsordasSmith}, G.\,Csordas, T.\,Craven and
W.\,Smith, via J.\,v.\,Sz.\,Nagy~\cite{Nagy}, attributed to Gauss
the enquiry\footnote{For the history of Gauss' enquiry,
see~\cite{EdwardsHinkkanen}.} about finding a relationship between
the number of nonreal zeroes of $\varphi$ and the number of real
zeroes of the function $Q[\varphi]$ in the case when $\varphi$ is
a real polynomial. They proved the following result for entire
functions in the class
$\mathcal{L-P^*}$~\cite{CravenCsordasSmith}.

\vspace{0.2cm}

\noindent\textbf{Theorem A
(\cite[Theorem~1]{CravenCsordasSmith}).} \textit{Let
$\varphi\in\mathcal{L-P^*}$. Suppose that the order of $\varphi$
is less than $2$ and that $\varphi$ has exactly $2m$, $m>0$,
nonreal zeroes. Let $\sigma\in\mathbb{R}$. Then the following
statements are equivalent.}
\begin{itemize}
\item[(a)] $\left(\dfrac{d}{dz}+\sigma\right)\varphi(z)\in\mathcal{L-P}$\textit{;}
\item[(b)] \textit{(i) $Z_{\mathbb{R}}(Q)=2m$, and (ii) if $\zeta_1\leqslant\zeta_2\leqslant\ldots
\leqslant\zeta_{2m}$ are the real zeroes of $Q[\varphi]$, then
$\varphi'/\varphi(\zeta_{2j-1})\leqslant-\sigma$ and
$\varphi'/\varphi(\zeta_{2j})\geqslant-\sigma$, $1\leqslant
j\leqslant m$, and $\varphi(z)\neq0$ for $\zeta_{2j-1}\leqslant
z\leqslant\zeta_{2j}$, $1\leqslant j\leqslant m$;}
\end{itemize}
\textit{where $Z_{\mathbb{R}}(Q)$ denotes the number of real
zeroes of $Q[\varphi]$, counting multiplicities.}

\vspace{0.2cm}

Thus, if for some $\sigma\in\mathbb{R}$, the polynomial $\sigma
p(z)+p'(z)$ has only real zeroes, then by~Theorem~A, the function
$Q[p](z)$ has exactly $2m$ real zeroes.

\vspace{0.2cm}

In light of Theorem~1.1, G.\,Csordas, T.\,Craven and
W.\,Smith~\cite{CravenCsordasSmith} stated the following
conjecture, which was recently nicknamed by T.\,Sheil-Small the
Hawaii conjecture~\cite{Sheil-Small} (see also~\cite{Csordas}).

\vspace{0.2cm}

\noindent\textbf{The Hawaii conjecture.} \textit{If a real
polynomial $p$ has precisely $2m$ nonreal zeroes, then}
\begin{equation*}
Z_{\mathbb{R}}(Q)\leqslant2m,
\end{equation*}
\textit{where $Z_{\mathbb{R}}(Q)$ denotes the number of real
zeroes of $Q[p]$, counting multiplicities.}

\vspace{0.2cm}

In~\cite{Csordas}, Csordas writes that this conjecture could also
have been stated for entire functions in the class
$\mathcal{L-P^*}$.

Thus, the Hawaii conjecture, if it is true, gives an exhaustive
answer on the question of a relationship between the number of
nonreal zeroes of the polynomial $p$ and the number of real zeroes
of the function $Q[p]$.

\vspace{0.2cm}

In~\cite[Chapter~9]{Sheil-Small}, T.\,Sheil-Small suggests several
appealing ideas concerning this conjecture. He writes in his
preface: "As this conjecture relates closely to the topological
structure formed by the level curves on which the logarithmic
derivative is real, it is a problem of fundamental interest in
understanding the structure of real polynomials". In particular,
Sheil-Small showed that the conjecture holds when
\begin{itemize}
\item the polynomial $p$ has exactly $2$ non-real zeroes;
\item the level set $\Im\left(\dfrac{p'(z)}{p(z)}\right)<0$, in
the upper half-plane, is connected;
\item the polynomial $p$ has purely imaginary zeroes and the degree
of $p$ is $2,4,6,8$ or $10$.
\end{itemize}

In~\cite{DilcherStolarsky}, K.\,Dilcher and K.\,Stolarsky
investigated the relationship between the distribution of zeroes
of a polynomial, $p(z)$, and those of the "Wronskian of the
polynomial", $Wp(z)$, where
$Wp(z)\stackrel{def}{=}p(z)p''(z)-(p'(z))^2$. K.\,Dilcher and
K.\,Stolarsky established several general properties of the
polynomial $Wp(z)$ and its zeroes. For example, they showed that
if $d$ is the minimum distance between two consecutive real zeroes
of $p(z)$, then the imaginary part of the zeroes of $Wp(z)$ cannot
be less than $d\sqrt{3}/4$~\cite[Lemma~2.8]{DilcherStolarsky}.

K.\,Dilcher~\cite{Dilcher} studied the geometry of the zeroes of
$Wp(z)$ and proved the Hawaii conjecture for polynomials whose
zeroes are sufficiently well spaced~\cite[Theorem~2.4]{Dilcher}.
Also he showed that any \textit{real} zero of $Wp(z)$, which is
not a zero of the polynomial $p(z)$, must lie on or inside the
Jensen circle of some pair of complex zeroes of $p(z)$.

Finally, we also mention that recently J.\,Borcea and
B.\,Shapiro~\cite{BorceaShapiro} developed a general theory of
level sets, which may imply the validity of the Hawaii conjecture
as a special case. But this approach, so far, has not led to a
resolution of the Hawaii conjecture. We should note that recently
a conjecture of J.\,Borcea and B.\,Shapiro made
in~\cite{BorceaShapiro} was disproved by S.\,Edwards and
A.\,Hinkkanen~\cite{EdwardsHinkkanen}.

\vspace{0.2cm}

We remark that while the upper bound of the number of real zeroes
of $Q$ was only conjectured recently, the lower bound was known a
long time ago, at least in the following special case:

\vspace{0.2cm}

\noindent\textbf{Problem 133 (\cite{GunterKuzmin}).} \textit{Let
the real polynomial $f(x)$ have only real zeroes and suppose that
the polynomial $f(x)+a$, where $a\in\mathbb{R}\backslash\{0\}$,
has $2m$ nonreal zeroes. Prove that the equation
$$\left(f'(x)\right)^2-f(x)f''(x)-af''(x)=0$$ has at~least~$2m$ real
roots.}

\vspace{0.2cm}

If $p(z)=f(z)+a$, then it follows that $p$ has only simple zeroes
and exactly $2m$ nonreal zeroes. Moreover, $p'=f'$ has only real
zeroes. By Problem~133, $Q$ associated with the polynomial
$p(z)=f(z)+a$ has at least $2m$ real zeroes. In fact, $Q$ has
exactly $2m$ real zeroes~\cite[Theorem~1]{CravenCsordasSmith}. The
hint provided for Problem~$133$ in~\cite{GunterKuzmin} suggests
using Rolle's theorem. This approach ultimately leads us to a more
precise result, namely, the following proposition.

\begin{prepos}\label{proposition.1}
Suppose that the polynomial $p$ has~$2m$ nonreal zeroes and its
derivative~$p'$ has~$2m_1$ nonreal zeroes, then $Q$ has at least
$2m-2m_1$ real zeroes.
\end{prepos}

Thus, the lower bound of the number of real zeroes of~$Q$ can be
easily determined. Moreover, as one can see, this bound depends
not only on the number of nonreal zeroes of the polynomial $p$ but
also on the number of nonreal zeroes of~$p'$. However, this simple
fact, Proposition~\ref{proposition.1}, was not well known.

\vspace{0.3cm}

In this paper, we establish the lower bound on the number of real
zeroes of the function $Q$ not only for real polynomials but also
for all functions in $\mathcal{L-P^*}$
(Theorem~\ref{theorem.the.Hawaii.conjecture}), so, \textit{ipso
facto}, we prove Proposition~\ref{proposition.1} for functions in
$\mathcal{L-P^*}$. But the main goal of the present work is to
estimate the number of real critical points of logarithmic
derivatives of real entire functions from above. Despite
geometrical importance of this problem, the basic and nearly
unique instrument we use in our investigation is Rolle's theorem.
Virtually, all results received in this paper are nontrivial
consequences of that theorem.

In Section~\ref{section.main}, we introduce some definitions and
auxiliary (old and new) facts, which we use throughout the paper.
We also introduce a specific property, \textit{property~A} (see
Definition~\ref{def.cond.A.axis}), of entire functions. Although
not every function in $\mathcal{L-P^*}$ possesses
\textit{property~A}, for a given
function~$\varphi\in\mathcal{L-P^*}$, one can always find another
function~$\psi_*$ in~$\mathcal{L-P^*}$ with \textit{property~A},
which has the same zeroes and the same associated
function~$Q[\varphi]=Q[\psi^*]$ (Theorem~\ref{Th.analog.of.PW}).

For a given entire function $\varphi$ we consider the function
$Q_1\stackrel{def}{=}Q_1[\varphi]\stackrel{def}{=}Q[\varphi']$
together with the function $Q[\varphi]$. In
Section~\ref{section:finite.intervals}, we study a relationship
between the number of real zeroes of the functions $Q$ and $Q_1$
on finite intervals.

Section~\ref{section:infinite.intervals} contains a comprehensive
account on the distribution of real critical points of logarithmic
derivatives of entire functions in $\mathcal{L-P^*}$. In
Sections~\ref{subsection:half-infinite.intervals}
and~\ref{subsection:functions.at most.one.real.zero}, we obtain
further particular results on the relationship between the number
of real zeroes of the functions $Q[\varphi]$ and $Q_1[\varphi]$.
Namely, we investigate this relationship on half-infinite
intervals free of poles of these functions and on the entire real
axis when the first derivative of $\varphi$ has no real zeroes.

In Section~\ref{section:general.theorems}, we obtain our main
result for functions in the class $\mathcal{L-P^*}$ with
\textit{property~A}. We prove the following inequalities, which
provide a connection between the number of nonreal zeroes of the
functions $\varphi$ and $\varphi'$ and the number of real zeroes
of the functions $Q[\varphi]$ and $Q_1[\varphi]$:
\begin{equation*}
2m-2m_1\leqslant
Z_{\mathbb{R}}(Q)\leqslant2m-2m_1+Z_{\mathbb{R}}(Q_1).
\end{equation*}
Here $2m$ and $2m_1$ are the number of nonreal zeroes of $\varphi$
and $\varphi'$, respectively. These inequalities together with one
technical result mentioned above, Theorem~\ref{Th.analog.of.PW},
allow us to prove the Hawaii conjecture,
Theorem~\ref{theorem.the.Hawaii.conjecture}.

In Section~\ref{section:main.results.for.U_2n*}, we extend results
of Section~\ref{section:general.theorems} to additional classes of
real entire functions of finite order with finitely many nonreal
zeroes.

\setcounter{equation}{0}


\section{The classes $U_{2n}^*$ and $\mathcal{L-P^*}$. Definitions and basic properties.}\label{section.main}

In this section, we give some necessary properties of functions in
the classes $U_{2n}^*$, $n\geqslant0$, and in particular, in the
class $\mathcal{L-P^*}$ and describe one important subclass of
$\mathcal{L-P^*}$.

\begin{definition}[\cite{Alander}, see also~\cite{EdwardsHellerstein,{HellersteinWilliamson1977}}]\label{def.class.V_2n}
The function $f$ is said to be in the class $V_{2n}$, $f\in
V_{2n}$, if
\begin{equation*}\label{def.V_2n.function}
f(z)=cz^de^{-\gamma
z^{2n+2}+q(z)}\prod^{\omega}_{j=1}\left(1-\dfrac{z}{\alpha_j}\right)e^{\tfrac{z}{\alpha_j}+\tfrac{z^2}{2\alpha_j^2}+\ldots+\tfrac{z^{l}}{l\alpha_j^{l}}},\quad
0\leqslant\omega\leqslant\infty.
\end{equation*}
where $\gamma\geqslant0$, $q$ is a real polynomial with $\deg
q\leqslant2n+1$, $l\leqslant2n+1$,
$\alpha_j\in\mathbb{R}\,\,\,\forall j$,
$\sum_j|\alpha_j|^{-l-1}<\infty$, $d$ is a nonnegative integer,
$c\in\mathbb{R}$, $c\neq0$.
\end{definition}

\begin{definition}\label{def.U_2n.function}
The class $U_{2n}$ is defined as follows
\begin{equation*}\label{def.U_0}
U_0\stackrel{def}{=}V_0
\end{equation*}
\begin{equation*}\label{def.U_2n}
U_{2n}\stackrel{def}{=}V_{2n}\backslash V_{2n-2}
\end{equation*}
for $n\geqslant0$.
\end{definition}

\begin{definition}[\cite{Laguerre,{Polya}}]\label{def.class.LP}
The class $U_0$ is called the Laguerre-P\'olya class
$\mathcal{L-P}$.
\end{definition}

By the classical results of Laguerre~\cite{Laguerre} and
P\'olya~\cite{Polya} (see also~\cite{Lindwart_Polya}
and~\cite{Levin}) $\varphi\in\mathcal{L-P}$ if and only if
$\varphi$ can be uniformly approximated on disks about the origin
by a sequence of polynomials with only real zeroes. Thus, this
fact it follows from this result that the class $\mathcal{L-P}$ is
closed with respect to differentiation.

\begin{definition}\label{def.U_2n*.function}
The function $\varphi$ is in the class $U_{2n}^{*}$ if
$\varphi=pf$ where $f\in U_{2n}$ and $p$ is a real polynomial with
no real zeroes. The class $U_0^{*}$ is often denoted by
$\mathcal{L-P^*}$ (see, for
instance,~\cite{CravenCsordasSmith,{Csordas}}).
\end{definition}

The class $U_0^{*}=\mathcal{L-P^*}$ plays a central role in our
investigation. All real polynomials belong to this class. Also
note that each class $U_{2n}^*$ is closed under differentiation
(see~\cite[Corollary~2.12]{EdwardsHellerstein}).
\begin{propos}\label{prop.diff.U_2n*.function}
If $\varphi\in U_{2n}^*$, then $\varphi'\in U_{2n}^*$.
\end{propos}

\begin{notation}\label{notation.number.of.zeroes}
For $\varphi\in U_{2n}^*$, by $Z_\mathbb{C}(\varphi)$ we denote
the number of nonreal zeroes of $\varphi$, counting
multiplicities. If $f$ is a real meromorphic function having only
a finite number of real zeroes, then $Z_\mathbb{R}(f)$ will denote
the number of real zeroes of $f$, counting multiplicities. In the
sequel, we also denote the number of zeroes of the function $f$ in
an~interval~$(a,b)$ and at~a~point~$\alpha\in\mathbb{R}$ by
$Z_{(a,b)}(f)$ and $Z_{\{\alpha\}}(f)$, respectively, thus
$Z_\mathbb{R}(f)=Z_{(-\infty,+\infty)}(f)$. Generally, the number
of zeroes of $f$ on a set $X$ where $X$ is a subset of $\mathbb{R}$ will be denoted by $Z_X(f)$.
\end{notation}

\subsection{Extra zeroes of functions in $U_{2n}^*$ }\label{susection:extra.zeroes}

Let $\varphi\in U_{2n}^*$. Between any two consecutive real
zeroes, say $a$ and $b$, $a<b$, of~$\varphi$,~$\varphi'$ has an
\textit{odd} number of real zeroes (and \textit{a fortiori} at
least one) by Rolle's theorem. Counting all zeroes with
multiplicities, suppose that $\varphi'$ has $2r+1$ zeroes between
$a$ and $b$. Then we will say that $\varphi'$ has $2r$
\textit{extra zeroes} between $a$ and $b$. If $\varphi$ has the
largest zero $a_L$ (or the smallest zero $a_S$), then any real
zero $\varphi'$ in $(a_L,\infty)$ (and in $(-\infty, a_S)$) is
also called an \textit{extra zero} of $\varphi'$.

\begin{notation}\label{notaion.extra.zeroes}
The total number of extra zeroes of $\varphi'$ on the entire real
axis, counting multiplicities, will be denoted by $E(\varphi')$.
\end{notation}
\begin{note}\label{remark.2.1.a}
The multiple real zeroes of~$\varphi$ are not counted as real
extra zeroes of~$\varphi'$.
\end{note}

It turns out that if a function $\varphi\in U_{2n}^*$ has at most
finitely many zeroes, then its number of extra zeroes can be
calculated exactly.

\begin{theorem}\label{Theorem.number.of.extra.zeroes.type.I}
Let $p$ be a real polynomial and let $\varphi$ be the function
defined as follows.
\begin{equation}\label{finite.zeroes.function.type.I}
\varphi(z)\stackrel{def}{=}e^{az^{2n+1}+q(z)}p(z),\quad
a\neq0\in\mathbb{R},\,\,\,n\in\mathbb{N}\cup\{0\},
\end{equation}
where $q$ and $p$ are real polynomials and $\deg q\leqslant2n$.
Then
\begin{equation*}\label{Number.of.extra.zeroes.of.function.type.I}
E(\varphi')=\begin{cases}
&Z_\mathbb{C}(\varphi)-Z_\mathbb{C}(\varphi')+2n\ \ \qquad\text{if}\ \deg p=Z_\mathbb{C}(p),\\
&Z_\mathbb{C}(\varphi)-Z_\mathbb{C}(\varphi')+2n+1\ \ \ \text{if}\
\deg p>Z_\mathbb{C}(p).
\end{cases}
\end{equation*}
\end{theorem}
\begin{proof}
From~\eqref{finite.zeroes.function.type.I}, it is easy to see that
$\varphi'$ has exactly $\deg p+2n$ zeroes, counting
multiplicities.

If $\varphi(z)\neq0$ for $z\in\mathbb{R}$, then $\deg
p=Z_\mathbb{C}(p)=Z_\mathbb{C}(\varphi)$ and all real zeroes of
$\varphi'$ are extra zeroes. Therefore,
$E(\varphi')=Z_\mathbb{C}(\varphi)+2n-Z_\mathbb{C}(\varphi')$.

If $\varphi$ has at least one real zero, then $\deg
p=Z_\mathbb{C}(\varphi)+r$, where $r\,(>0)$ is the number of real
zeroes of $\varphi$, counting multiplicities. Therefore,
$\varphi'$ has exactly
$Z_\mathbb{C}(\varphi)+2n+r-Z_\mathbb{C}(\varphi')$ real zeroes,
counting multiplicities, $r-1$ of which are guaranteed by Rolle's
theorem. Thus,
$E(\varphi')=Z_\mathbb{C}(\varphi)+2n+1-Z_\mathbb{C}(\varphi')$.
\end{proof}

In the same way, the following two theorems can be proved.

\begin{theorem}\label{Theorem.number.of.extra.zeroes.type.II}
Let $p$ be a real polynomial and let $\varphi$ be the function
defined as follows.
\begin{equation*}
\varphi(z)\stackrel{def}{=}e^{az^{2n}+q(z)}p(z),\quad a>0,\,\,\,
n\in\mathbb{N}\cup\{0\},
\end{equation*}
where $q$ is a real polynomial of degree at most\footnote{If
$n=0$, then $q(z)\equiv0$.} $2n-1$ and $p$ is a real polynomial.
Then
\begin{equation}\label{Number.of.extra.zeroes.of.function.type.II}
E(\varphi')=\begin{cases}
&Z_\mathbb{C}(\varphi)-Z_\mathbb{C}(\varphi')+2n-1\ \ \ \text{if}\ \deg p=Z_\mathbb{C}(p),\\
&Z_\mathbb{C}(\varphi)-Z_\mathbb{C}(\varphi')+2n\ \
\qquad\text{if}\ \deg p>Z_\mathbb{C}(p).
\end{cases}
\end{equation}
\end{theorem}

\begin{theorem}\label{Theorem.number.of.extra.zeroes.type.III}
Let $p$ be a real polynomial and let $\varphi$ be the function
defined as follows.
\begin{equation}\label{finite.zeroes.function.type.III}
\displaystyle\varphi(z)\stackrel{def}{=}e^{-az^{2n+2}+q(z)}p(z),\quad
a>0,\,\,\,n\in\mathbb{N}\cup\{0\},
\end{equation}
where $q$ and $p$ are real polynomials and $\deg q\leqslant2n+1$.
Then
\begin{equation*}\label{Number.of.extra.zeroes.of.function.type.III}
E(\varphi')=\begin{cases}
&Z_\mathbb{C}(\varphi)-Z_\mathbb{C}(\varphi')+2n+1\ \ \ \text{if}\ \deg p=Z_\mathbb{C}(p),\\
&Z_\mathbb{C}(\varphi)-Z_\mathbb{C}(\varphi')+2n+2\ \ \ \text{if}\
\deg p>Z_\mathbb{C}(p).
\end{cases}
\end{equation*}
\end{theorem}

We should note that
Theorems~\ref{Theorem.number.of.extra.zeroes.type.I},~\ref{Theorem.number.of.extra.zeroes.type.II}
and~\ref{Theorem.number.of.extra.zeroes.type.III} are parts of
Lemma~2.8 of the work~\cite{EdwardsHellerstein}. We state these
parts in terms of our definition of extra zeroes. Another part of
Lemma~2.8 in~\cite{EdwardsHellerstein} concerns functions with
infinitely many zeroes.

\begin{theorem}[\cite{EdwardsHellerstein}]\label{Theorem.number.of.extra.zeroes.type.IV}
Let $\varphi$ be in $U_{2n}^*$.
\begin{itemize}
\item If $\varphi$ has infinitely many positive and negative
zeroes, then
\begin{equation*}\label{Number.of.extra.zeroes.of.function.type.IV.1}
E(\varphi')= Z_\mathbb{C}(\varphi)-Z_\mathbb{C}(\varphi')+2n.
\end{equation*}
\item If $\varphi$ has infinitely many zeroes but only finitely many positive or negative
zeroes, then
\begin{equation*}\label{Number.of.extra.zeroes.of.function.type.IV.2}
Z_\mathbb{C}(\varphi)-Z_\mathbb{C}(\varphi')+2n\leqslant
E(\varphi')\leqslant
Z_\mathbb{C}(\varphi)-Z_\mathbb{C}(\varphi')+2n+1.
\end{equation*}

\end{itemize}

\end{theorem}

For functions in the class $U_0^*=\mathcal{L-P^*}$, this theorem
was established by T.\,Craven, G.\,Csordas, and W.\,Smith
in~\cite[p.~325]{CravenCsordasSmith2}.

\subsection{Logarithmic derivatives of entire functions in $U_{2n}$}\label{section:log.deriv}

Given a function $\varphi$, the following function
\begin{equation*}\label{logarithmic.derivative}
L(z)\stackrel{def}{=}\dfrac{d\ln\varphi(z)}{dz}=\dfrac{\varphi'(z)}{\varphi(z)}
\end{equation*}
is called the logarithmic derivative of $\varphi$.

The following important fact was proved
in~\cite[Lemma~4]{BergweilerEremenko}.
\begin{theorem}\label{Theorem.logarithmic.derivative.representation}
Let $\varphi$ be a function of class $U_{2n}$. Then the
logarithmic derivative of $\varphi$ has a representation
\begin{equation}\label{logarithmic.derivative.representation.1}
\dfrac{\varphi'(z)}{\varphi(z)}=h(z)\mu(z),
\end{equation}
where $h$ is a real polynomial, $\deg h=2n$, the leading
coefficient of $h$ is negative, and $\mu\not\equiv0$ is a function
with nonnegative imaginary part in the upper half-plane of the
complex plane.
\end{theorem}

As was shown in~\cite{BergweilerEremenko}, if $\varphi\in U_{2n}$
has no zeroes, then there can be only three situations.
\begin{itemize}
\item
$\varphi(z)=e^{-az^{2n+2}+q(z)}$, where $a>0$ and $q$ is a real
polynomial, $\deg q\leqslant2n+1$. Then
$\varphi'(z)/\varphi(z)=-a(2n+2)z^{2n+1}+q'(z)$ is a real
polynomial of an odd degree and, therefore, it has a real zero
$\beta$. So one can put
$$
h(z)=\dfrac{-a(2n+2)z^{2n+1}+q'(z)}{z-\beta}\quad\text{and}\quad\mu(z)=z-\beta.
$$
\item
$\varphi(z)=e^{az^{2n+1}+q(z)}$, where $a\neq0\in\mathbb{R}$ and
$q$ is a real polynomial, $\deg q\leqslant2n$. Then one can set
$h(z)=-|a|(2n+1)z^{2n}+q'(z)$ and $\mu(z)=-\sgn a$.
\item $\varphi(z)=e^{az^{2n}+q(z)}$, where $a>0$ and $q$ is a real
polynomial, $\deg q\leqslant2n-1$. Then we set
$h(z)=-2naz^{2n}-zq'(z)$ and $\mu(z)=-1/z$.
\end{itemize}

From now on, assume that $\varphi\in U_{2n}$ has at least one real
zero~$\alpha_0$. If $\varphi$ has only finitely many negative
zeroes and $\varphi(z)\to0$ as $z\to-\infty$, then we
follow~\cite{BergweilerEremenko} to consider $-\infty$ as a zero
of $\varphi$. Similarly, we consider $+\infty$ as a zero of
$\varphi$ if $\varphi$ has only finitely many positive zeroes and
$\varphi(z)\to0$ as $z\to+\infty$. We arrange the zeroes into an
increasing sequence $\{\alpha_j\}$, where each zero occurs once,
disregarding multiplicity. The range of the subscript $j$ will be
$M<j<N$, where $-\infty\leqslant M<0\leqslant N\leqslant+\infty$,
with $\alpha_{M+1}=-\infty$ and $\alpha_{N-1}=+\infty$ in the
cases described above.

By Rolle's theorem, each open interval $(\alpha_j,\alpha_{j+1})$ contains a zero $\beta_j$ of
$\varphi'$. To make a definite choice, we take for $\beta_j$ the
largest zero in this interval. Each $\beta_j$ occurs in this
sequence only once, and we disregard multiplicity. As was shown
in~\cite{BergweilerEremenko}, in this case, the function $\mu$ in
Theorem~\ref{Theorem.logarithmic.derivative.representation} has
the following form
\begin{equation*}\label{logarithmic.derivative.representation.mu.1}
\displaystyle\mu(z)=\dfrac1{z-\alpha_{N-1}}\prod_{M<j<N-1}\dfrac{1-z/\beta_j}{1-z/\alpha_j},
\end{equation*}
where the factor $z-\alpha_{N-1}$ is omitted if
$\alpha_{N-1}=+\infty$ or $N=+\infty$, and the factor
$1-z/\alpha_{M+1}$ is omitted if $\alpha_{M+1}=-\infty$. If for
some $j\in(M,N-1)$ we have $\alpha_j=0$ or $\beta_j=0$, then
the~$j^{\mathrm{th}}$ factor has to be replaced by
$(z-\beta_j)/(z-\alpha_j)$.

As was shown by Chebotarev~\cite[p.~310]{Levin}, the function
$\mu$ as a function mapping the upper half-plane onto itself may
be represented in the form
\begin{equation*}\label{logarithmic.derivative.representation.mu.2}
\displaystyle\mu(z)=az+b+\sum_{j=M}^{N}A_j\left(\dfrac1{\alpha_j-z}-\dfrac1{\alpha_j}\right),
\quad-\infty\leqslant M<N\leqslant+\infty,
\end{equation*}
where $a\geqslant0$, $b\in\mathbb{R}$, $A_j\geqslant0$ where the
series
%
$\displaystyle\sum_{j=M}^{N}\dfrac{A_j}{\alpha_j^2}$
%
converges.

From this representation it follows that $\mu'(z)>0$ for every
real $z\neq\alpha_j$. In particular, if $\varphi\in
U_0=\mathcal{L-P}$, then $h(z)\equiv c<0$
in~\eqref{logarithmic.derivative.representation.1}. Therefore,
\begin{equation}\label{deriv.of.log.deriv.in.LP}
\left(\dfrac{\varphi'(z)}{\varphi(z)}\right)'=
c\mu'(z)<0\quad\text{for}\quad z\in\mathbb{R},\ \ \ z\neq\alpha_j.
\end{equation}
Thus, we proved the well-known fact (see, for
example,~\cite{CravenCsordasSmith,{CravenCsordasSmith2}}) that the
logarithmic derivative of a function in the Laguerre-P\'olya class
is a decreasing function on the intervals where it has no poles.

We use these facts in the next section and in
Section~\ref{section:main.results.for.U_2n*}.

\subsection{Derivatives of logarithmic derivatives. Entire functions with \textit{property}~A}\label{subsection.property.A}

Let $\varphi\in U_2n^*$ and let the function $Q=Q[\varphi]$
associated with $\varphi$ be defined as
\begin{equation}\label{main.function.2}
Q(z)\stackrel{def}{=}Q[\varphi](z)\stackrel{def}{=}\dfrac{d}{dz}\left(\dfrac{\varphi'(z)}{\varphi(z)}\right)=
\dfrac{\varphi(z)\varphi''(z)-\left(\varphi'(z)\right)^2}{\left(\varphi(z)\right)^2}.
\end{equation}
We note that if $\varphi(z)=Ce^{\beta z}$, where
$C,\beta\in\mathbb{R}$, then $Q(z)\equiv0$. At the same time, all
functions of the form $Ce^{\beta z}$ belong to the
class~$U_0^*=\mathcal{L-P^*}$. Hence, we adopt the following
convention throughout this paper.

\vspace{0.2cm}

\noindent\textbf{Convention.} \textit{If
$\varphi\in\mathcal{L-P^*}$, then $\varphi$ is assumed not to be
of the form $\varphi(z)=Ce^{\beta z}$, $C,\beta\in\mathbb{R}$.}

\vspace{0.2cm}

Analogously to~\eqref{main.function.2}, we introduce the related
function
\begin{equation}\label{all.main.functions}
Q_1(z)\stackrel{def}{=}Q[\varphi'](z)\stackrel{def}{=}\dfrac{d}{dz}\left(\dfrac{\varphi''(z)}{\varphi'(z)}\right)=
\dfrac{\varphi'(z)\varphi'''(z)-\left(\varphi''(z)\right)^2}{\left(\varphi'(z)\right)^2}.
\end{equation}

Our interest is concentrated only on the number of \textit{real}
zeroes of the function $Q$, $Z_{\mathbb{R}}(Q)$, and on bounding
this number. Obviously, $Q[\varphi]$ has a finite number of real
zeroes if $\varphi$ has a finitely many zeroes. But generally
speaking, $Q[\varphi]$ may have infinitely many real zeroes.
However, for the function $\varphi$ in the class
$\mathcal{L-P^*}$, the function $Q[\varphi]$ has also a finite
number of real zeroes even if $\varphi$ has infinitely many
zeroes. This fact is essentially known
from~\cite{CravenCsordasSmith}, but we still include the proof for
completeness.

\begin{theorem}\label{Theorem.finitude.number.of.zeroes.of.Q}
Let $\varphi\in\mathcal{L-P^*}$. Then the function $Q$ has
finitely many real zeroes.
\end{theorem}
\begin{proof}
By definition, $\varphi$ is a product $\varphi=p\psi$, where
$\psi\in\mathcal{L-P}$ and $p$ is a real polynomial with no real
zeroes. Let $\deg p=2m\geqslant0$.

\vspace{1mm}

If $m=0$, then $p(z)\equiv\mathit{const}$. In this case,
$\varphi\in\mathcal{L-P}$. But it is well known (see, for
example,~\cite{CravenCsordasSmith,{CravenCsordasSmith2}}) that the
logarithmic derivative of a function from the Laguerre-P\'olya
class is a decreasing function on the intervals where it has no
poles. Therefore, $Q(z)<0$ for any real $z$, which is not a pole
of this function. Thus, if $m=0$, then $Q$ has no real zeroes.

\vspace{1mm}

We assume now that $m>0$. Observe that
\begin{equation*}\label{finitude.number.of.zeroes.of.Q.proof.2}
Q=Q[\varphi]=Q[p]+Q[\psi].
\end{equation*}
From this equality it follows that all zeroes of $Q$ are roots of
the equation
\begin{equation}\label{finitude.number.of.zeroes.of.Q.proof.3}
Q[p](z)=-Q[\psi](z).
\end{equation}

It is easy to see that if $p(z)=a_0z^{2m}+\ldots$ $(a_0\neq0)$,
then
\begin{equation*}\label{finitude.number.of.zeroes.of.Q.proof.4}
Q[p](z)=\dfrac{-2ma_0^2z^{4m-2}+\ldots}{a_0^2z^{4m}+\ldots}.
\end{equation*}
This formula shows that $Q[p](z)\to0$ whenever $z\to\pm\infty$ and
$Q[p](z)<0$ when $z$ is real and $|z|$ is sufficiently large.
Consequently, there exist two real numbers $a_1$ and $a_2$
$(a_1<a_2)$ such that $Q[p](z)<0$ for
$z\in(-\infty,a_1]\cup[a_2,+\infty)$. But since
$\psi\in\mathcal{L-P}$, we have $Q[\psi](z)<0$ for
$z\in\mathbb{R}$ outside the zeroes of $\psi$ as we mentioned
above. Thus, the right hand side of the
equation~\eqref{finitude.number.of.zeroes.of.Q.proof.3} is
positive for all $z\in\mathbb{R}$ but its left hand side is
negative for all $z\in(-\infty,a_1]\cup[a_2,+\infty)$. Therefore,
all real roots of the
equation~\eqref{finitude.number.of.zeroes.of.Q.proof.3} and,
consequently, all real zeroes of the function~$Q$ belong to the
interval $(a_1,a_2)$, and $Q(z)<0$ outside this interval and
outside the zeroes of $\psi$. Since real zeroes of a meromorphic
function are isolated, $Q$ has only finitely many real zeroes, as
required.
\end{proof}

\noindent In fact, the number of real zeroes of the function $Q$
associated with functions in the class $\mathcal{L-P^*}$ is even
(see Corollary~2).

For functions in the classes $U_{2n}^*$ with $n\geqslant1$,
Theorem~\ref{Theorem.finitude.number.of.zeroes.of.Q} is not valid.
That is, if $\varphi\in U_{2n}^*$ with $n\geqslant1$ has
infinitely many zeroes, then its associated function $Q[\varphi]$
may have infinitely many zeroes.
\begin{example}
Consider the function $f(z)=e^{az^2}\sin z$ with $a>0$. It is
clear that $f\in U_2^*$ and $f$ has infinitely many real zeroes
$\alpha_k=\pi k, k\in\mathbb{Z}$. In this case, the function
$Q[f]$ is as follows:
\begin{equation*}\label{example.for.inifinitely.many.zeroes}
Q[f](z)=-\dfrac1{\sin^2z}+2a.
\end{equation*}
Its zeroes are the roots of the equation $\sin^2z=\dfrac1{2a}$.
For $a\geqslant\dfrac12$, this equation has infinitely many roots:
\begin{equation}\label{example.for.inifinitely.many.zeroes.Zeros}
\zeta_k=\pm\arcsin\dfrac1{\sqrt{2a}}+\pi k,\quad k\in\mathbb{Z}.
\end{equation}
Thus, for $a\geqslant\dfrac12$, the function $Q[f]$ associated
with $f(z)=e^{az^2}\sin z\in U_2^*$ has infinitely many real
zeroes given by~\eqref{example.for.inifinitely.many.zeroes.Zeros}.
\end{example}
\begin{notation}\label{notation.zeroes.of.functions}
For convenience, we use $\alpha_j$ to denote the zeroes of the
function $\varphi$, $\beta_j$ to denote the zeroes of $\varphi'$,
and $\gamma_j$ to denote the zeroes of~$\varphi''$.
\end{notation}

The following simple fact is very important for the sequel.

\begin{propos}\label{prop.Q.at.alpha}
Let $\alpha\in\mathbb{R}$ be a zero of $\varphi\in U_{2n}^*$. For
all sufficiently small $\varepsilon>0$, the following inequality
holds
\begin{equation}\label{ineq.1}
Q(\alpha\pm\varepsilon)<0.
\end{equation}
\end{propos}
\begin{proof}
If $\alpha$ is a zero of $\varphi$ of multiplicity $M\geqslant1$,
then $\varphi(z)=(z-\alpha)^{M}\psi(z)$, where
$\psi(\alpha)\neq0$. Thus, we have
\begin{equation*}\label{ineq.0.1}
Q(z)=-\dfrac{M}{(z-\alpha)^2}+\dfrac{\psi(z)\psi''(z)-(\psi'(z))^2}{(\psi(z))^2}.
\end{equation*}
Consequently, the function $Q$ is negative in a small punctured
neighbourhood of $\alpha$ as~required.
\end{proof}

Furthermore, as was shown in~\cite[p.~418]{CravenCsordasSmith}
(see also the proof of
Theorem~\ref{Theorem.finitude.number.of.zeroes.of.Q}),
the~function~$Q[\varphi](z)$ associated with
$\varphi\in\mathcal{L-P^*}$ is negative for sufficiently
large~$z$. This fact and Proposition~\ref{prop.Q.at.alpha} imply
the following lemma which concerns the parity of the number of
real zeroes of $Q$ on the half-interval $[\alpha_L,+\infty)$ where
$\alpha_L$ is the largest zero of the function
$\varphi\in\mathbb{R}$.

\begin{lemma}\label{lemma.2}
Let $\varphi\in\mathcal{L-P^*}$. If $\varphi$ has the largest zero
$\alpha_L$ \emph{(}or the smallest zero $\alpha_S$\emph{)},
then~$Q$ has an even number of real zeroes in $(\alpha_L,+\infty)$
\emph{(}or in $(-\infty,\alpha_S)$\emph{)}, counting
multiplicities.
\end{lemma}
\begin{proof}
The inequality~\eqref{ineq.1} holds for any real zero of
$\varphi$, consequently, $Q$ is negative for $z$ sufficiently
close to $\alpha_L$ (or to $\alpha_S$). But it was already proved
in Theorem~\ref{Theorem.finitude.number.of.zeroes.of.Q} (see also
(3.11) in~\cite[p.415]{CravenCsordasSmith} and subsequent remark
there) that $Q(z)<0$ for all sufficiently large real $z$.
Therefore, $Q$~has an even number of zeroes
in~$(\alpha_L,+\infty)$ (and in~$(-\infty,a_S)$ if $\varphi$ has
the smallest real zero~$\alpha_S$), counting multiplicities, since
$Q(z)$ is negative for all real $z$ sufficiently close to the ends
of the interval~$(\alpha_L,+\infty)$~(or of the interval
$(-\infty,\alpha_S)$).
\end{proof}

Using this lemma and Proposition~\ref{prop.Q.at.alpha}, it is easy
to establish the following fact concerning the parity of the
number~$Z_{\mathbb{R}}(Q)$ for the function $Q$ associated with a
function in~$\mathcal{L-P^*}$.
\begin{corol}[Craven--Csordas--Smith \cite{CravenCsordasSmith}, p. 415]\label{corol.3}
If $\varphi\in\mathcal{L-P^*}$, then the function $Q$ associated
with $\varphi$ has an even number of real zeroes, counting
multiplicity.
\end{corol}
\begin{proof}
In fact, if $\varphi$ has no real zeroes, then $Q$ has no real
poles, and the number $Z_\mathbb{R}(Q)$ is even, since $Q(z)<0$
for all sufficiently large real $z$.

If $\varphi$ has only one real zero $\alpha$, then, according to
Lemma~\ref{lemma.2}, $Q$ has an even number of zeroes in each of
the intervals~$(-\infty,\alpha)$ and~$(\alpha,+\infty)$. Thus,
$Z_\mathbb{R}(Q)$ is also even in this case.

Let $\varphi$ have at least two distinct real zeroes. If
$\alpha_j$ and $\alpha_{j+1}$ are two consecutive zeroes of
$\varphi$, then, according to~\eqref{ineq.1}, $Q$ has an even
number of zeroes, counting multiplicities, in the interval
$(\alpha_j,\alpha_{j+1})$. If $\varphi$ has the largest (or/and
the smallest) real zero, say $\alpha_L$ ($\alpha_S$), then, by
Lemma~\ref{lemma.2}, $Q$ has an~even number of real zeroes,
counting multiplicities, in $(a_L,+\infty)$ (and in
$(-\infty,\alpha_S)$). Therefore, the number $Z_\mathbb{R}(Q)$ is
even.
\end{proof}
\begin{note}\label{remark.2.3}
Analogously, the function $Q_1$ associated with a function in the
class $\mathcal{L-P^*}$ has an even number of real zeroes,
counting multiplicities, since the class $\mathcal{L-P^*}$ is
closed under differentiation by
Proposition~\ref{prop.diff.U_2n*.function}.
\end{note}

As we will see in Section~\ref{section:main.results.for.U_2n*},
the functions $Q[\varphi]$ associated with functions $\varphi$ in
the classes $U_{2n}^*$ with $n\geqslant1$ can be positive for all
sufficiently large real~$z$. Moreover, $Q[\varphi]$ can have an
odd number of zeroes. However, in some cases $Q[\varphi](z)$ can
also be negative for all sufficiently large real~$z$. In
Section~\ref{section:main.results.for.U_2n*}, we use the following
theorem.

\begin{theorem}\label{Theorem.Q.for.U_2n*.at.plus.infinity}
Let $\varphi$ be in $U_{2n}^*$ with $n\geqslant1$. If $\varphi$
has the largest zero $\alpha_L$ and $\varphi'$ has an~odd number
of zeroes, counting multiplicities, in the interval
$(\alpha_L,+\infty)$, then $Q[\varphi](z)<0$ for all sufficiently
large positive~$z$.
\end{theorem}
\begin{proof}
By definition, $\varphi$ belongs to $U_{2n}^*$ if $\varphi=pf$,
where $p$ is a real polynomial with no real zeroes and $f\in
U_{2n}$. According to
Theorem~\ref{Theorem.logarithmic.derivative.representation}
(see~\eqref{logarithmic.derivative.representation.1}), the
logarithmic derivative of $f$ can be represented in the form
$f'/f=h\mu$, where $h$ is a real polynomial of
degree~$2n$ whose leading coefficient is negative, and $\mu$ is
the meromorphic function described in
Section~\ref{section:log.deriv}. Thus, for the derivative of the
logarithmic derivative of the function $\varphi$, we have
\begin{equation}\label{Theorem.Q.for.U_2n*.at.plus.infinity.proof.1}
Q[\varphi](z)=\left(\dfrac{\varphi'(z)}{\varphi(z)}\right)'=h(z)\mu'(z)+h'(z)\mu(z)+Q[p](z).
\end{equation}
Next, by assumption, the~function $\varphi'$ has an odd number of
zeroes in the interval $(\alpha_L,+\infty)$. Therefore, by
construction (see Section~\ref{section:log.deriv}), the function
$\mu$ has a unique simple zero $\beta$ in the interval
$(\alpha_L,+\infty)$ that coincides with one of the zeroes of
$\varphi'$. Moreover, as we noted in
Section~\ref{section:log.deriv}, $\mu'(z)>0$ in any interval
between its poles, therefore, $\mu(z)$ must be negative in the
interval $(\alpha_L,\beta)$ and positive in the interval
$(\beta,+\infty)$. Furthermore, the function $Q[p](z)$ is negative
for all sufficiently large real~$z$ as was shown in the proof of
Theorem~\ref{Theorem.finitude.number.of.zeroes.of.Q}
(see~\eqref{finitude.number.of.zeroes.of.Q.proof.4}). At last, by
Theorem~\ref{Theorem.logarithmic.derivative.representation}, the
polynomial $h(z)$ and its derivative $h'(z)$ are negative for all
sufficiently large positive~$z$, since $\deg h=2n>0$ and the
leading coefficient of $h$ is negative.

Thus, we have shown that all the summands are
negative~at~$+\infty$
in~\eqref{Theorem.Q.for.U_2n*.at.plus.infinity.proof.1}, that is,
$Q[\varphi](z)<0$ for all sufficiently large positive~$z$, as
required.
\end{proof}

Theorem~\ref{Theorem.Q.for.U_2n*.at.plus.infinity} is valid with
respective modification in the case when $\varphi$ has the
smallest zero~$a_S$ and $\varphi'$ has an odd number of zeroes in
the interval $(-\infty,\alpha_S)$.
\begin{theorem}\label{Theorem.Q.for.U_2n*.at.minus.infinity}
Let $\varphi$ be in $U_{2n}^*$ with $n\geqslant1$. If $\varphi$
has the smallest zero $\alpha_S$ and $\varphi'$ has an~odd number
of zeroes, counting multiplicities, in the interval
$(-\infty,\alpha_S)$, then $Q[\varphi](z)<0$ for all sufficiently
large negative~$z$.
\end{theorem}

\vspace{3mm}

The following definition plays a crucial role in our investigation
and in the proof of the~Hawaii conjecture.
\begin{definition}\label{def.cond.A.point}
Let $\varphi\in\mathcal{L-P^*}$ and let $\alpha$ be a real zero of
$\varphi$. Suppose that $\beta_1$ and $\beta_2$,
$\beta_1<\alpha<\beta_2$, are real zeroes of $\varphi'$ such that
$\varphi'(z)\neq0$ for $z\in(\beta_1,\alpha)\cup(\alpha,\beta_2)$.
The function~$\varphi$ is said to possess \textbf{property A at
its real zero $\alpha$} if $Q$ has no real zeroes in at least one
of the intervals $(\beta_1,\alpha)$ and $(\alpha,\beta_2)$. If
$\alpha$ is the smallest zero of~$\varphi$, then set
$\beta_1=-\infty$, and if $\alpha$ is the largest zero
of~$\varphi$, then set $\beta_2=+\infty$.
\end{definition}
\begin{definition}\label{def.cond.A.axis}
A function $\varphi\in\mathcal{L-P^*}$ is said to possess
\textbf{property~A} if $\varphi$ possesses property~A at each of
its real zeroes. In particular, $\varphi$ without real zeroes
possesses property~A.
\end{definition}

To illustrate \textit{property~A}, we use a function
$\varphi\in\mathcal{L-P^*}$ of the form $\varphi(z)=e^{\lambda
z}p(z)$, where $\lambda>0$ and $p$ is a real polynomial. On
Figure~\ref{pic.1}, there is the graphic of the logarithmic
derivative of such a function $\varphi$ with \textit{property~A}.

\begin{figure}[ht]
\centering \includegraphics{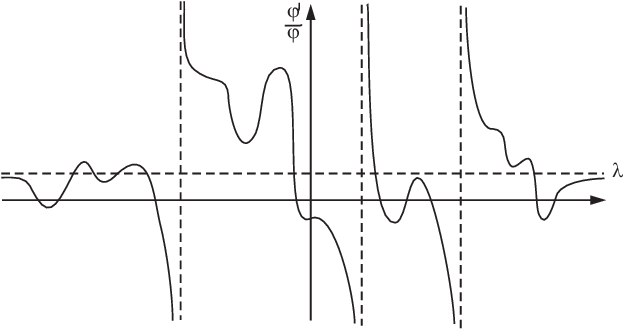} \caption{}
\label{pic.1}
\end{figure}

\textit{Property~A} is a very special property of entire
functions. It cannot be verified easily. However, for a given
function~$\varphi$ in~$\mathcal{L-P^*}$, we can always find
another function~$\psi_*$ in~$\mathcal{L-P^*}$ with
\textit{property~A}, which has the same zeroes and the same
associated function~$Q$. The~following theorem establishes the
existence of such a function.

\vspace{2mm}

\begin{theorem}\label{Th.analog.of.PW}
Let $\varphi\in\mathcal{L-P^*}$. Then there exist a real number
$\sigma_*$ such that $\psi_*(z)=e^{-\sigma_* z}\varphi(z)$
possesses \textit{property A}. Moreover, if
$Z_\mathbb{R}(Q)\neq0$, then
$Z_\mathbb{C}(\psi'_*)<Z_\mathbb{C}(\psi_*)$.
\end{theorem}

Before we start to prove the theorem, let us notice that the
function $\psi(z)=e^{-\sigma z}\varphi(z)$ has the same set of
zeroes as the function $\varphi(z)$, for any real $\sigma$. At the
same time, we have the following relations:
\begin{equation}\label{Th.analog.of.PW.proof.1}
\dfrac{\psi'(z)}{\psi(z)}=\dfrac{\varphi'(z)}{\varphi(z)}-\sigma,\qquad
Q[\psi]=Q[\varphi].
\end{equation}

Using these relations, one can easily give the main idea of the
proof of Theorem~\ref{Th.analog.of.PW}. In fact, let again, for
the sake of simplicity, $\varphi(z)=e^{\lambda z}p(z)$, where
$\lambda>0$ and $p$ is a real polynomial. Suppose that $\varphi$
does not possess \textit{property~A}. This means that its
logarithmic derivative $\varphi'/\varphi$ has critical points on
some intervals between a pole of $\varphi'/\varphi$, and there
exists a zero of $Q[\varphi]$, which is closer to a pole of
$\varphi'/\varphi$ than the closest to this pole zero of
$\varphi'/\varphi$ (see Figure~\ref{pic.2}).

From Figure~\ref{pic.2} one can see that the function
$\psi(z)=e^{-\sigma z}\varphi(z)$ possesses \textit{property~A},
for sufficiently large positive $\sigma$. Now we may reduce
$\sigma$ until the line $y=\sigma$ meets the first critical
point\footnote{That is, the critical point of $\varphi'/\varphi$
such that $\varphi'/\varphi$ has the maximal value among its
values at the points of local maximums of $Q[\varphi]$
(see~Figure~\ref{pic.2}).}~of~$\varphi'/\varphi$.

\begin{figure}[ht]
\centering \includegraphics{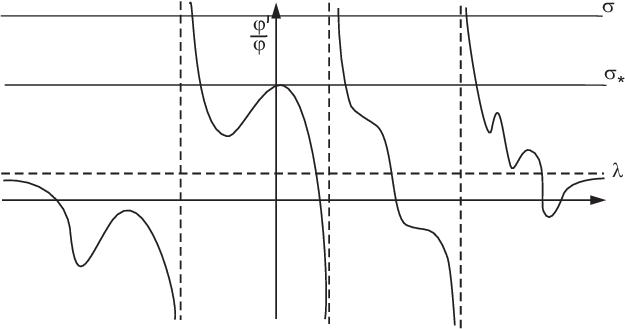} \caption{}
\label{pic.2}
\end{figure}

\begin{proof}
By Theorem~\ref{Theorem.finitude.number.of.zeroes.of.Q}, the
function $Q[\varphi]$ has finitely many real zeroes. If
$Q[\varphi]$ has no real zeroes, then, for any real $\sigma$,
$\psi(z)=e^{-\sigma z}\varphi(z)$ automatically possesses
\textit{property A}.

\vspace{3mm}

Let $Q[\varphi]$ have at least one real zero and let
$\zeta_1<\zeta_2<\ldots<\zeta_n$, $n\geqslant1$, be the
\textit{distinct} real zeroes~of~$Q[\varphi]$. We set
\begin{equation*}\label{Th.analog.of.PW.proof.1.1}
\displaystyle\sigma_*=\max_{1\leqslant i\leqslant n
}\left(\dfrac{\varphi'(\zeta_i)}{\varphi(\zeta_i)}\right)\quad\text{and}\quad
\psi_*(z)=e^{-\sigma_* z}\varphi(z).
\end{equation*}
Then from~\eqref{Th.analog.of.PW.proof.1} applied to $\psi_*$ it
follows that
\begin{equation}\label{Th.analog.of.PW.proof.1.3}
\dfrac{\psi'_*(\zeta_i)}{\psi_*(\zeta_i)}\leqslant0,\quad
i=1,\ldots,n.
\end{equation}
In this theorem, we denote by $\zeta_*$ any (fixed) zero of
$Q[\varphi]$ where the maximum
$\sigma_*=\displaystyle\max_{1\leqslant i\leqslant
n}(\varphi'(\zeta_i)/\varphi(\zeta_i))$ is attained. Thus, we have
\begin{equation}\label{Th.analog.of.PW.proof.1.4}
\dfrac{\psi'_*(\zeta_*)}{\psi_*(\zeta_*)}=0.
\end{equation}
We also denote~by~$\alpha_j$ $(j\in\mathbb{Z})$ the~\textit{real}
zeroes of $\varphi$ (and of $\psi_*$) and
$I_j=(\alpha_j,\alpha_{j+1})$. If $\varphi$ has the~largest
zero~$\alpha_L$ and the smallest zero~$\alpha_S$, then we also
consider the intervals
\begin{equation*}
I_{+\infty}=(\alpha_L,+\infty)\quad\text{and}\quad
I_{-\infty}=(-\infty,\alpha_S).
\end{equation*}

We now show that $\psi_*$ possesses \textit{property~A}. If
$\varphi$ has no real zeroes, then $\psi_*$ also has no real
zeroes and, therefore, it possesses \textit{property~A}.

Suppose that $\varphi$ has at least one real zero. Consider the
interval $I_j$ for a fixed finite $j$. Let $\beta$ be the~leftmost
zero of $\psi'_*/\psi_*$ in $I_j$. The existence of this zero is
guaranteed by Rolle's theorem. Since
$\psi'_*(z)/\psi_*(z)\to+\infty$ whenever~$z\searrow\alpha_j$
(see~\eqref{ineq.1}), we have $\psi'_*(z)/\psi_*(z)>0$ for
$z\in(\alpha_j,\beta)$. Now from~\eqref{Th.analog.of.PW.proof.1.3}
it follows that $Q[\psi_*]$ has no zeroes in $(\alpha_j,\beta)$.
Consequently, $\psi_*$ possesses \textit{property~A}
at~$\alpha_j$. In particular, if $\psi_*$ has the smallest zero
$\alpha_S=\alpha_1$, then considering the interval $I_1$, we
obtain that $\psi_*$ possesses \textit{property~A} at $\alpha_S$.
Let $\psi_*$ have the largest zero $\alpha_L$. If $\psi'_*/\psi_*$
has no zeroes in the interval $I_{+\infty}$, then
$\psi'_*(z)/\psi_*(z)>0$ for $z\in I_{+\infty}$, since
$\psi'_*(z)/\psi_*(z)\to+\infty$ whenever~$z\searrow\alpha_L$
(see~\eqref{ineq.1}). Therefore,
by~\eqref{Th.analog.of.PW.proof.1.3}, we have $Q[\psi_*](z)\neq0$
in the interval $(\alpha_L,+\infty)$, and $\psi_*$ also possesses
\textit{property~A} at $\alpha_L$. If $\psi'_*/\psi_*$ has at
least one zero in $I_{+\infty}$ and $\beta$ is the leftmost one,
then $\psi'_*(z)/\psi_*(z)>0$ for $z\in(\alpha_L,\beta)$. So,
$Q[\psi_*]$ can not have zeroes in $z\in(\alpha_L,\beta)$ by the
same reasoning as above. Thus, we obtain that $\psi_*$ possesses
\textit{property~A} at each of its real zeroes (if any).
Consequently, $\psi_*$ possesses \textit{property~A}.

\vspace{3mm}

We now prove that $Z_\mathbb{C}(\psi'_*)<Z_\mathbb{C}(\psi_*)$.
There can be only the following two cases.

\vspace{1mm}

\noindent Case I. Let the function $\psi_*$ have no real zeroes.
Then all real zeroes of $\psi'_*$ are extra~zeroes.

If $\zeta_*$ is a zero of $Q[\psi_*]$ of multiplicity
$M\geqslant1$, then $\zeta_*$ is a zero of $\psi'_*$ of
multiplicity $M+1$ according to~\eqref{Th.analog.of.PW.proof.1.4}.
Thus, we have $E(\psi'_*)\geqslant M+1$. Since by
Theorems~\ref{Theorem.number.of.extra.zeroes.type.I}--\ref{Theorem.number.of.extra.zeroes.type.III}
(with $n=0$)
\begin{equation}\label{extraroots.general.2}
Z_\mathbb{C}(\psi_*)-Z_\mathbb{C}(\psi'_*)-1\leqslant
E(\psi'_*)\leqslant Z_\mathbb{C}(\psi_*)-Z_\mathbb{C}(\psi'_*)+1,
\end{equation}
we have $1\leqslant M\leqslant
Z_\mathbb{C}(\psi_*)-Z_\mathbb{C}(\psi'_*)$, that is,
$Z_\mathbb{C}(\psi'_*)<Z_\mathbb{C}(\psi_*)$.

\vspace{4mm}

\noindent Case II. Let now $\psi_*$ have at least one real zero.
Recall that all zeroes of $\psi'_*$ in the every interval $I_j$
are extra zeroes of $\psi'_*$ except one, counting multiplicity.
At the same time, all the zeroes of $\psi'_*$ in the intervals
$I_{-\infty}$ and $I_{+\infty}$ are extra zeroes of $\psi'_*$.

\vspace{2mm}

Let $\psi_*$ have at least two distinct real zeroes and let
$\zeta_*\in I_j$ for some fixed finite $j$. If $\zeta_*$ is a zero of
$Q[\psi_*]$ of \textit{even} multiplicity, then $\zeta_*$ is a
zero of $\psi'_*$ of odd multiplicity (at least three). Since one
zero of $\psi_*$ in~$I_j$ is guaranteed by Rolle's theorem,  we
have
\begin{equation}\label{Th.analog.of.PW.proof.1.6}
E(\psi'_*)\geqslant2.
\end{equation}

If $\psi'_*$ is not of the
form~\eqref{finite.zeroes.function.type.III} (with $n=0$), then by
Theorems~\ref{Theorem.number.of.extra.zeroes.type.I},~\ref{Theorem.number.of.extra.zeroes.type.II}
and~\ref{Theorem.number.of.extra.zeroes.type.IV},
$\psi_*$~satisfies the inequalities
\begin{equation}\label{extraroots.2}
Z_\mathbb{C}(\psi_*)-Z_\mathbb{C}(\psi'_*)\leqslant
E(\psi'_*)\leqslant Z_\mathbb{C}(\psi_*)-Z_\mathbb{C}(\psi'_*)+1.
\end{equation}
These inequalities together with~\eqref{Th.analog.of.PW.proof.1.6}
imply $Z_\mathbb{C}(\psi'_*)<Z_\mathbb{C}(\psi_*)$.

If $\psi'_*$ has the form~\eqref{finite.zeroes.function.type.III}
(with $n=0$), then we have
\begin{equation}\label{Th.analog.of.PW.proof.1.2}
\dfrac{\psi'_*(z)}{\psi_*(z)}=\dfrac{p'(z)}{p(z)}-2az+b,\quad a>0,
\end{equation}
where $p$ is a real polynomial and $b\in\mathbb{R}$. Hence the
interval $I_{-\infty}$ exists, and $\psi'_*$ has an odd number of
extra zeroes (at least one) in~$I_{-\infty}$. In~fact,
$\psi'_*(z)/\psi_*(z)\to+\infty$ whenever $z\to-\infty$ and
$\psi'_*(z)/\psi_*(z)\to-\infty$ whenever $z\nearrow\alpha_S$
by~\eqref{ineq.1} and~\eqref{Th.analog.of.PW.proof.1.2}. Thus, in
this case the inequality~\eqref{Th.analog.of.PW.proof.1.6} can be
improved to the following one:
\begin{equation}\label{Th.analog.of.PW.proof.1.7}
E(\psi'_*)\geqslant3.
\end{equation}
But by Theorem~\ref{Theorem.number.of.extra.zeroes.type.III},
$E(\psi'_*)=Z_\mathbb{C}(\psi_*)-Z_\mathbb{C}(\psi'_*)+2$. Now
from~\eqref{Th.analog.of.PW.proof.1.7} we obtain
$Z_\mathbb{C}(\psi'_*)<Z_\mathbb{C}(\psi_*)$.

\vspace{2mm}

If $\zeta_*$ is a zero of $Q[\psi_*]$ of \textit{odd}
multiplicity, then $\zeta_*$ is a zero of $\psi'_*$ of even
multiplicity (at least two). But by Rolle's theorem, $\psi_*$ has
an odd number of zeroes in $I_j$ if $j$ is finite,
so~\eqref{Th.analog.of.PW.proof.1.6} is valid if $\psi_*$ is not
of the form~\eqref{finite.zeroes.function.type.III} (with $n=0$).
If $\psi_*$ is of the form~\eqref{finite.zeroes.function.type.III}
(with $n=0$), then in this case~\eqref{Th.analog.of.PW.proof.1.7}
can be proved by the same method as above. Thus, we also have
$Z_\mathbb{C}(\psi'_*)<Z_\mathbb{C}(\psi_*)$
by~\eqref{extraroots.2} and~\eqref{extraroots.general.2}.

\vspace{2mm}

Let $\psi_*$ have at least one real zero and let $\zeta_*\in
I_{-\infty}$ or $\zeta_*\in I_{+\infty}$. By the same reasoning as
above, one can show that the
inequality~\eqref{Th.analog.of.PW.proof.1.7} holds in this case.
So, we again obtain the inequality
$Z_\mathbb{C}(\psi'_*)<Z_\mathbb{C}(\psi_*)$ by
Theorems~\ref{Theorem.number.of.extra.zeroes.type.I}--\ref{Theorem.number.of.extra.zeroes.type.IV}.

\vspace{2mm}

Thus, we have shown that for a given $\varphi\in\mathcal{L-P^*}$,
there exists a real $\sigma_*$ such that the function
$\psi_*(z)=e^{-\sigma_*z}\varphi(z)$ possesses
\textit{property~A}. Additionally, if $Z_\mathbb{R}(Q)>0$, then
$Z_\mathbb{C}(\psi'_*)<Z_\mathbb{C}(\psi_*)$, as required.
\end{proof}
\begin{note}\label{remark.2.12}
For a given function $\varphi\in\mathcal{L-P^*}$ the number
$\sigma_*$ guaranteed by Theorem~\ref{Th.analog.of.PW} is not
unique. For example, one can find another number applying
Theorem~\ref{Th.analog.of.PW} to the function $\varphi(-z)$.
\end{note}
\begin{note}\label{remark.new}
We recall that, for $\varphi$ and $\psi_*$ of
Theorem~\ref{Th.analog.of.PW}, we have $Q[\psi_*]=Q[\varphi]$
(see~\eqref{Th.analog.of.PW.proof.1}).
\end{note}
\begin{note}\label{remark.Th.analog.of.PW.extension}
In the same way as in Theorem~\ref{Th.analog.of.PW}, one can prove
that if, for a function $\varphi\in U_{2n}^*$ with $n\geqslant1$,
its associated function $Q[\varphi]$ has only finitely many real
zeroes, then there exists a real number $\sigma_*$ such that
$\psi_*(z)=e^{-\sigma_* z}\varphi(z)$ possesses \textit{property
A}. This number is not unique. Generally speaking,
Theorem~\ref{Th.analog.of.PW} is not true for the classes
$U_{2n}^*$ with $n\geqslant1$.
\end{note}

We use the functions with \textit{property~A} and
Theorem~\ref{Th.analog.of.PW} in
Sections~\ref{subsection:finite.intervals.2}
and~\ref{section.Hawaii}

\setcounter{equation}{0}

\section{Bounds on the number of real critical points of the logarithmic derivative on finite
intervals}\label{section:finite.intervals}

Let $\varphi$ be in $U_{2n}^*$ and let the function $Q$ be defined
as in~\eqref{main.function.2}. In this section we establish bounds
on the number of zeroes of the function $Q[\varphi]$ on finite
intervals.

Section~\ref{subsection:finite.intervals.1} is devoted to
estimates of the number of zeroes of $Q$ on intervals free of
zeroes of the functions $\varphi$, $\varphi'$ and $\varphi''$. In
this section, we also prove a basic fact, Theorem~\ref{lemma.5},
establishing an interrelation between the numbers $Z_{(a,b)}(Q)$
and $Z_{(a,b)}(Q_1)$, where $(a,b)$ is an interval such that
$\varphi(z)\neq0$, $\varphi'(z)\neq0$ and $\varphi''(z)\neq0$ for
$z\in(a,b)$.

In Section~\ref{subsection:finite.intervals.2}, we establish
bounds on the number of zeroes of the function $Q$ on intervals
free of zeroes of $\varphi'$ and $\varphi$, on intervals free of
zeroes of $\varphi$ and on intervals with a unique zero of
$\varphi$.

\subsection{Intervals free of zeroes of $\varphi$, $\varphi'$ and $\varphi''$}\label{subsection:finite.intervals.1}

For a given function~$\varphi\in U_{2n}^*$, by $F$ and $F_1$ we
will denote the following functions
\begin{equation}\label{F.functions}
F(z)=\varphi(z)\varphi''(z)-\left(\varphi'(z)\right)^2,\quad
F_1(z)=\varphi'(z)\varphi'''(z)-\left(\varphi''(z)\right)^2.
\end{equation}

Our first result is about the number of zeroes of $Q$ on a finite
interval free of zeroes of the functions $\varphi$, $\varphi'$,
$\varphi''$ and $Q_1$.

\begin{lemma}\label{lemma.3}
Let $\varphi\in U_{2n}^*$ and let $a$ and $b$ be real and let
$\varphi(z)\neq0$, $\varphi'(z)\neq0$, $\varphi''(z)\neq0$,
$Q_{1}(z)\neq0$ in the~interval~$(a,b)$. Suppose additionally that
if $\varphi(b)\neq0$ then $\varphi'(b)\neq0$ as well.
\begin{itemize}
\item [\emph{I.}] If, for all sufficiently small $\delta>0$,
\begin{equation}\label{lemma.3.condition.1}
\varphi'(a+\delta)\varphi''(a+\delta)Q(a+\delta)Q_{1}(a+\delta)>0,
\end{equation}
then $Q$ has no zeroes in $(a,b]$.
\item [\emph{II.}] If, for all sufficiently small $\delta>0$,
\begin{equation}\label{lemma.3.condition.2}
\varphi'(a+\delta)\varphi''(a+\delta)Q(a+\delta)Q_{1}(a+\delta)<0,
\end{equation}
then $Q$ has at most one zero in $(a,b)$, counting multiplicities.
Moreover, if~$Q(\zeta)=0$ for some $\zeta\in(a,b)$, then
$Q(b)\neq0$ \emph{(}if $Q$ is finite at $b$\emph{)}.
\end{itemize}
\end{lemma}
\begin{proof} The condition $\varphi(z)\neq0$ for $z\in(a,b)$ means
that $Q$ is finite at every point of~$(a,b)$.

If $\zeta\in(a,b)$ and $Q(\zeta)=0$, then $F(\zeta)=0$
and~\eqref{F.functions} implies
\begin{equation}\label{main.equality.1}
\varphi'(\zeta)=\dfrac{\varphi(\zeta)\varphi''(\zeta)}{\varphi'(\zeta)}.
\end{equation}
Now we consider $F_{1}$. From~\eqref{F.functions}
and~\eqref{main.equality.1} it is easy to derive that
\begin{equation}\label{main.work.formula}
\begin{array}{c}
F_{1}(\zeta)=\varphi'(\zeta)\varphi'''(\zeta)-\left(\varphi''(\zeta)\right)^2=
\dfrac{\varphi(\zeta)\varphi''(\zeta)\varphi'''(\zeta)}{\varphi'(\zeta)}-\left(\varphi''(\zeta)\right)^2=\\
 \\
=\dfrac{\varphi''(\zeta)}{\varphi'(\zeta)}\left[\varphi(\zeta)\varphi'''(\zeta)-
\varphi'(\zeta)\varphi''(\zeta)\right]=\dfrac{\varphi''(\zeta)}{\varphi'(\zeta)}F'(\zeta).
\end{array}
\end{equation}
Since $\varphi'(z)\neq0,\varphi''(z)\neq0,Q_{1}(z)\neq0$ (and
therefore $F_{1}(z)\neq0$) in $(a,b)$ by assumption,
from~\eqref{main.work.formula} it follows that $\zeta$ is a~simple
zero of $Q$. That is, all zeroes of $Q$ in~$(a,b)$ are simple.

\vspace{2mm}

\noindent I. Let the inequality~\eqref{lemma.3.condition.1} hold.
Assume that, for all sufficiently small~$\delta>0$,
\begin{equation}\label{lemma.3.proof.1}
\varphi'(a+\delta)\varphi''(a+\delta)Q_{1}(a+\delta)>0,
\end{equation}
then $Q(a+\delta)>0$, that is, $F(a+\delta)>0$. Therefore, if
$\zeta$ is the leftmost zero of $Q$ in $(a,b)$, then
$F'(\zeta)<0$. This inequality
contradicts~\eqref{main.work.formula}, since
\begin{equation*}
\varphi'(z)\varphi''(z)Q_{1}(z)>0
\end{equation*}
for $z\in(a,b)$, which follows from~\eqref{lemma.3.proof.1} and
from the assumption of the lemma. Consequently, $Q$~cannot have
zeroes in the interval~$(a,b)$ if the
inequalities~\eqref{lemma.3.condition.1}
and~\eqref{lemma.3.proof.1} hold. In the same way, one can prove
that if $\varphi'(a+\delta)\varphi''(a+\delta)Q_{1}(a+\delta)<0$
for all sufficiently small~$\delta>0$ and if the
inequality~\eqref{lemma.3.condition.1} hold, then $Q(z)\neq0$
for~$z\in(a,b)$.

Thus, $Q$ has no zeroes in the interval~$(a,b)$ if the
inequality~\eqref{lemma.3.condition.1} holds. Moreover, it is easy
to show that $Q(b)\neq0$ as well. To do this, we first note that
from~\eqref{F.functions} it follows that
\begin{equation*}
\varphi'(z)=\dfrac{\varphi(z)\varphi''(z)}{\varphi'(z)}-\dfrac{F(z)}{\varphi'(z)},
\end{equation*}
since $\varphi'(z)\neq0$ for $z\in(a,b)$ by assumption. Substituting this expression into the formula~\eqref{F.functions},
we obtain
\begin{equation}\label{new}
\begin{array}{c}
F_{1}(z)=\dfrac{\varphi(z)\varphi''(z)\varphi'''(z)}{\varphi'(z)}-
\dfrac{\varphi'(z)\varphi''(z)\varphi''(z)}{\varphi'(z)}-\dfrac{F(z)\varphi'''(z)}{\varphi'(z)}=\\
=\dfrac{\varphi''(z)}{\varphi'(z)}\cdot
F'(z)-\dfrac{F(z)\varphi'''(z)}{\varphi'(z)}.
\end{array}
\end{equation}

Suppose now that $Q(b)=0$. Then $\varphi(b)\neq0$, therefore,
$\varphi'(b)\neq0$ by assumption. So we obtain $F(b)=0$ and, from~\eqref{F.functions},
$\varphi''(b)\neq0$. Thus, we have $(\varphi\varphi'\varphi'')(b)\neq0$. From~\eqref{new} it follows
that the functions $F'$ and $F_1$ have have a zero of the same order at $b$. In particular, $F'(b)\neq0$
if and only if $F_1(b)\neq0$. Furthermore, it is clear that the order of the zero of $F'$ (and $F_1$)
at $b$ is strictly smaller than the order of the zero of $\varphi'''F$ at $b$. Consequently, from~\eqref{new}
we obtain, for all sufficiently small~$\varepsilon>0$,
\begin{equation}\label{additional.work.formula.1}
\sgn\left(\dfrac{\varphi'(b-\varepsilon)}
{\varphi''(b-\varepsilon)}F_{1}(b-\varepsilon)\right)=\sgn(F'(b-\varepsilon)).
\end{equation}
But if the
inequality~\eqref{lemma.3.condition.1} holds, then
\begin{equation}\label{lemma.3.proof.2}
\sgn\left(\dfrac{\varphi'(b-\varepsilon)}
{\varphi''(b-\varepsilon)}F_{1}(b-\varepsilon)\right)=\sgn(F(b-\varepsilon))
\end{equation}
for all sufficiently small $\varepsilon>0$, since
$\varphi'(z)\neq0$, $\varphi''(z)\neq0$, $Q_{1}(z)\neq0$ in the
interval~$(a,b)$ by assumption and since $Q(z)\neq0$ in~$(a,b)$,
which was proved above. So, if the
inequality~\eqref{lemma.3.condition.1} holds and if $F(b)=0$, then
from~\eqref{additional.work.formula.1} and~\eqref{lemma.3.proof.2}
we obtain that
\begin{equation*}
F(b-\varepsilon)F'(b-\varepsilon)>0
\end{equation*}
for all sufficiently small $\varepsilon>0$. This inequality
contradicts the analyticity\footnote{If a real function $f$ is
analytic at some neighbourhood of a real point $a$ and equals zero
at this point, then, for all sufficiently small $\varepsilon>0$,
\begin{equation*}
f(a-\varepsilon)f'(a-\varepsilon)<0.
\end{equation*}
} of the function~$F$. Since $Q(b)$ exists, $Q_1(b)$ is finite by
assumption  of the lemma, so everything established above is true
for the functions~$Q$ and $Q_1$. Therefore, if the
inequality~\eqref{lemma.3.condition.1} holds and if $Q$ is finite
at~the~point~$b$, then $Q(b)\neq0$. Thus, the first part of the
lemma is proved.

\vspace{2mm}

\noindent II. Let the inequality~\eqref{lemma.3.condition.2}
hold, then $Q$ can have zeroes in $(a,b)$. But it cannot have
more than one zero, counting multiplicity. In fact, if $\zeta$ is
the leftmost zero of $Q$ in $(a,b)$, then this zero is simple as
we proved above. Therefore, the following inequality holds for all
sufficiently small $\varepsilon>0$
\begin{equation*}
\varphi'(\zeta+\varepsilon)\varphi''(\zeta+\varepsilon)Q(\zeta+\varepsilon)Q_{1}(\zeta+\varepsilon)>0.
\end{equation*}
Consequently, $Q$ has no zeroes in $(\zeta,b]$ according to Case~I
of the lemma.
\end{proof}
\begin{note}\label{remark.half.inf.1}
Lemma~\ref{lemma.3} is also true if $(a,b)$ is a half-infinite
interval, i.e., $(a,+\infty)$ or $(-\infty,b)$.
\end{note}

Thus, we have found out that $Q$ has at most one real zero,
counting multiplicity, in an interval where the functions
$\varphi$, $\varphi'$, $\varphi''$ and $Q_1$ have no real zeroes.
Now we study multiple zeroes of $Q$ and its zeroes common with one
of the above-mentioned functions. From~\eqref{main.function.2} it
follows that all zeroes of $\varphi'$ of multiplicity at least $2$
that are not zeroes of $\varphi$ are also zeroes of $Q$ and all
zeroes of $\varphi'$ of multiplicity at least $3$ that are not
zeroes of $\varphi$ are multiple zeroes of $Q$. The following
lemma provides information about common zeroes of $Q$ and $Q_1$.

\begin{lemma}\label{lemma.4}
Let $\varphi\in U_{2n}^*$ and let $a$ and $b$ be real and let
$\varphi(z)\neq0$, $\varphi'(z)\neq0$, $\varphi''(z)\neq0$ in the
interval~$(a,b)$. Suppose that $Q_1$ has a unique zero
$\xi\in(a,b)$ of multiplicity  $M$ in~$(a,b)$, and suppose
additionally that $\varphi'(b)\neq0$ if $\varphi(b)\neq0$.

If~$Q(\xi)=0$, then $\xi$ is a zero of $Q$ of multiplicity $M+1$,
and $Q(z)\neq0$ for $z\in(a,\xi)\cup(\xi,b]$.
\end{lemma}
\begin{proof} The condition $\varphi(z)\neq0$ for $z\in(a,b)$ means
that $Q$ is finite at every point of~$(a,b)$.

By assumption, $\xi$ is a zero of $F_{1}$ of multiplicity $M$ and
$F(\xi)=0$. First, we prove that $\xi$ is a~zero of $F$
of~multiplicity~$M+1$.

Note that the expression~\eqref{new} can be rewritten in the form
\begin{equation*}
\dfrac{\varphi'(z)}{\left[\varphi''(z)\right]^2}F_{1}(z)=\left(\dfrac{F(z)}{\varphi''(z)}\right)',
\end{equation*}
since $\varphi''(z)\neq0$ for $z\in(a,b)$ by assumption.
Differentiating this equality $j$ times with respect~to~$z$, we
get
\begin{equation}\label{derivatives}
\left(\dfrac{\varphi'(z)}{\left[\varphi''(z)\right]^2}~F_{1}(z)\right)^{(j)}=\left(\dfrac{F(z)}{\varphi''(z)}\right)^{(j+1)}.
\end{equation}
From~\eqref{derivatives} it follows that  $F^{(j+1)}(\xi)=0$ if
$\varphi'(\xi)\neq0$, $\varphi''(\xi)\neq0$, $F^{(i)}_{1}(\xi)=0$
and $F^{(i)}(\xi)=0$, $i=0,1,\ldots,j$. Consequently, $\xi$ is a
zero of $F$ of multiplicity at least $M+1$. But by assumptions,
\eqref{derivatives} implies the following formula
\begin{equation*}
0\neq\varphi'(\xi)F^{(M)}_{1}(\xi)=\varphi''(\xi)F^{(M+1)}(\xi).
\end{equation*}
Hence, $\xi$ is a zero of $F$ of multiplicity exactly $M+1$. But
$\varphi(\xi)\neq0$ by assumption, therefore,~$\xi$~is a~zero of
$Q$ of multiplicity $M+1$.

It remains to prove that $Q$ has no zeroes in $(a,b]$ except
$\xi$. In fact, consider the interval~$(a,\xi)$. According to
Lemma~\ref{lemma.3}, $Q$ can have a zero at $\xi$ only if the
inequality~\eqref{lemma.3.condition.2} holds and $Q(z)\neq0$ for
$z\in(a,\xi)$. Furthermore, the function $\varphi'\varphi''$ does
not change its sign at $\xi$ but the function~$QQ_{1}$ does, since
$\xi$ is a zero of~$QQ_{1}$ of multiplicity~$2M+1$. Thus, for all
sufficiently small $\delta>0$,
\begin{equation}\label{additional.work.formula.2}
\varphi'(\xi+\delta)\varphi''(\xi+\delta)Q(\xi+\delta)Q_{1}(\xi+\delta)>0,
\end{equation}
since the inequality~\eqref{lemma.3.condition.2} must hold in the
interval $(a,\xi)$ by Lemma~\ref{lemma.3}.
From~\eqref{additional.work.formula.2} it follows that Case~I of
Lemma~\ref{lemma.3} holds in the interval~$(\xi,b)$, so
$Q(z)\neq0$ for $z\in(\xi,b]$.
\end{proof}
\begin{note}\label{remark.2.4.new}
Lemma~\ref{lemma.4} remains valid if $(a,b)$ is a half-infinite
interval, that is, $(a,+\infty)$ or $(-\infty,b)$.
\end{note}

Now combining the two last lemmata, we provide a general bound on
the number of real zeroes of $Q$ in terms of the number of real
zeroes of $Q_1$ in a given interval.
\begin{lemma}\label{lemma.5}
Let $\varphi\in U_{2n}^*$ and let $a$ and $b$ be real. If
$\varphi(z)\neq0$, $\varphi'(z)\neq0$ and $\varphi''(z)\neq0$ for
$z\in(a,b)$, then
\begin{equation}\label{main.work.formula.1}
Z_{(a,b)}(Q)\leqslant1+Z_{(a,b)}(Q_{1}).
\end{equation}
\end{lemma}
\begin{proof}
If $\varphi(z)\varphi''(z)<0$ in $(a,b)$, then $Q(z)<0$
for~$z\in(a,b)$ by~\eqref{main.function.2}, that
is,~$Z_{(a,b)}(Q)=0$. Therefore, the
inequality~\eqref{main.work.formula.1} holds automatically in this
case.

Let now  $\varphi(z)\varphi''(z)>0$ for $z\in(a,b)$.
If~$Q_{1}(z)\neq0$ in~$(a,b)$, that is, $Z_{(a,b)}(Q_{1})=0$, then
by Lemma~\ref{lemma.3}, $Q$ has at most one real zero, counting
multiplicity, in~$(a,b)$. Therefore,~\eqref{main.work.formula.1}
also holds in this case.

If $Q_1$ has a unique zero $\xi$ in~$(a,b)$ and $Q(\xi)\neq0$,
then by Lemma~\ref{lemma.3}, $Q$ has at most one real zero,
counting multiplicity, in each interval $(a,\xi)$ and $(\xi,b)$:
\begin{equation}\label{lemma.5.proof.1}
Z_{(a,\xi)}(Q)\leqslant1+Z_{(a,\xi)}(Q_{1}),
\end{equation}
where $Z_{(a,\xi)}(Q_{1})=0$, and
\begin{equation}\label{lemma.5.proof.2}
Z_{(\xi,b)}(Q)\leqslant1+Z_{(\xi,b)}(Q_{1}),
\end{equation}
where $Z_{(\xi,b)}(Q_{1})=0$. Since $Q(\xi)\neq0$ and
$Q_1(\xi)=0$, we have
\begin{equation}\label{lemma.5.proof.3}
0=Z_{\{\xi\}}(Q)\leqslant-1+Z_{\{\xi\}}(Q_{1}).
\end{equation}
Thus, summing the
inequalities~\eqref{lemma.5.proof.1}--\eqref{lemma.5.proof.3}, we
obtain~\eqref{main.work.formula.1}.

\vspace{2mm}

If $Q_1$ has a unique zero $\xi$ in~$(a,b)$ and $Q(\xi)=0$, then,
by Lemma~\ref{lemma.4}, we have
\begin{equation*}\label{lemma.5.proof.4}
Z_{\{\xi\}}(Q)=1+Z_{\{\xi\}}(Q_{1}),
\end{equation*}
and $Q(z)\neq0$ for $z\in(a,\xi)\cup(\xi,b)$. Therefore, the
inequality~\eqref{main.work.formula.1} is also true in this case.

\vspace{3mm}

Now, let $Q_{1}$ have exactly $r\geqslant2$ \textit{distinct} real
zeroes, say $\xi_1<\xi_2<\ldots<\xi_r$, in the interval~$(a,b)$.
These zeroes divide $(a,b)$ into $r+1$ subintervals. If, for some
number $i$, $1\leqslant i\leqslant r$, $Q(\xi_i)\neq0$,
then by Lemma~\ref{lemma.3}, $Q$ has \textit{at~most} one real
zero, counting multiplicity, in~$(\xi_{i-1},\xi_i]$
($\xi_0\stackrel{def}=a$). But $Q_{1}$ has \textit{at~least} one
real zero in $(\xi_{i-1},\xi_i]$, counting multiplicities (at the
point~$\xi_i$). Consequently,
\begin{equation}\label{intermediate.1.lemma.5}
Z_{(\xi_{i-1},\xi_i]}(Q)\leqslant Z_{(\xi_{i-1},\xi_i]}(Q_{1})
\end{equation}
If, for some number $i$, $1\leqslant i\leqslant r-1$, $Q(\xi_i)=0$
and $\xi_i$ is a zero of $Q_{1}$ of multiplicity~$M$, then by
Lemma~\ref{lemma.4}, $Q$ has only one zero $\xi_{i}$ of
multiplicity~$M+1$ in~$(\xi_{i-1},\xi_{i+1}]$. But in
the~interval~$(\xi_{i-1},\xi_{i+1}]$, $Q_{1}$ has \textit{at
least} $M+1$ real zeroes, counting multiplicities (namely, $\xi_i$
which is a zero of multiplicity~$M$, and $\xi_{i+1}$). Therefore,
in this case, the following inequality holds
\begin{equation}\label{intermediate.2.lemma.5}
Z_{(\xi_{i-1},\xi_{i+1}]}(Q)\leqslant
Z_{(\xi_{i-1},\xi_{i+1}]}(Q_{1})
\end{equation}
Thus, if $Q(\xi_r)\neq0$, then
from~\eqref{intermediate.1.lemma.5}--\eqref{intermediate.2.lemma.5}
it follows that
\begin{equation}\label{intermediate.3.lemma.5}
Z_{(a,\xi_r]}(Q)\leqslant Z_{(a,\xi_r]}(Q_{1}).
\end{equation}
But by Lemma~\ref{lemma.3}, $Q$ has \textit{at~most} one real
zero, counting multiplicity, in the interval $(\xi_r,b)$.
Consequently, if $Q(\xi_r)\neq0$, then the
inequality~\eqref{main.work.formula.1} is valid.

If $Q(\xi_r)=0$, then by Lemma~\ref{lemma.4}, $Q(\xi_{r-1})\neq0$
(otherwise, $\xi_r$ cannot be a zero of $Q$) and
from~\eqref{intermediate.1.lemma.5}--\eqref{intermediate.2.lemma.5}
it follows that
\begin{equation}\label{intermediate.4.lemma.5}
Z_{(a,\xi_{r-1}]}(Q)\leqslant Z_{(a,\xi_{r-1}]}(Q_{1}).
\end{equation}
Now Lemma~\ref{lemma.4} implies
\begin{equation}\label{intermediate.5.lemma.5}
Z_{(\xi_{r-1},b)}(Q)=1+Z_{(\xi_{r-1},b)}(Q_{1}),
\end{equation}
therefore, the inequality~\eqref{main.work.formula.1} follows
from~\eqref{intermediate.4.lemma.5}--\eqref{intermediate.5.lemma.5}.
\end{proof}
\begin{note}\label{remark.2.4}
Lemma~\ref{lemma.5} is also true if $(a,b)$ is a half-infinite
interval, that is, $(a,+\infty)$ or $(-\infty,b)$ (see
Remarks~\ref{remark.half.inf.1} and~\ref{remark.2.4.new}).
\end{note}

The inequality~\eqref{main.work.formula.1} plays the main role in
the sequel. \textit{Per se}, all subsequent inequalities are
consequences of this inequality and Rolle's theorem.

\subsection{Intervals between zeroes of $\varphi$, $\varphi'$ and $\varphi''$}\label{subsection:finite.intervals.2}

At first, we estimate the parities of the numbers of real zeroes
of $Q$ in certain intervals.

\begin{lemma}\label{lemma.1}
Let $\varphi\in U_{2n}^*$ and let $\beta_1$ and $\beta_2$ be two
real zeroes of $\varphi'$.
\begin{itemize}
\item [\emph{I.}] If $\beta_1$ and
$\beta_2$ are consecutive real zeroes of $\varphi'$, and
$\varphi(z)\neq0$ for $z\in[\beta_1,\beta_2]$, then $Q$ has an odd
number of real zeroes in $(\beta_1,\beta_2)$, counting
multiplicities.
\item[\emph{II.}] If $\beta_1$ and
$\beta_2$ are two real zeroes of $\varphi'$ such that $\varphi$
has a unique real zero $\alpha$ in $(\beta_1,\beta_2)$ and
$\varphi'(z)\neq0$ for $z\in(\beta_1,\alpha)\cup(\alpha,\beta_2)$,
then $Q$ has an even number of real zeroes, counting
multiplicities, in each of the intervals $(\beta_1,\alpha)$
and~$(\alpha,\beta_2)$.
\end{itemize}
\end{lemma}
\begin{proof}
\noindent I. In fact, since $\varphi'/\varphi$ equals zero at the
points $\beta_1$ and $\beta_2$, its derivative, the function $Q$,
has an odd number of zeroes in $(\beta_1,\beta_2)$ by Rolle's
theorem.

\vspace{2mm}

\noindent II. According to Proposition~\ref{prop.Q.at.alpha}
(see~\eqref{ineq.1}), $\varphi'/\varphi$ is decreasing in a small
left-sided vicinity of $\alpha$, and therefore,
$\varphi'(z)/\varphi(z)\to-\infty$ whenever $z\nearrow\alpha$.
Since $\varphi'/\varphi$ equals zero at $\beta_1$, its derivative,
the function $Q$, must have an even number of zeroes in
$(\beta_1,\alpha)$ by Rolle's theorem.

By the same argumentation, $Q$ has an even number of zeroes in
$(\alpha,\beta_2)$.
\end{proof}

We now embark on a more detailed analysis of the zeroes of
$\varphi$, $\varphi'$, and $\varphi''$. In the following lemma, we
consider an arbitrary pair of zeroes of $\varphi''$. According to
Notation~\ref{notation.zeroes.of.functions}, we denote them
by~$\gamma^{(1)}$,~$\gamma^{(2)}$.

\begin{lemma}\label{lemma.6.5}
Let $\varphi\in U_{2n}^*$ and let $\gamma^{(1)}$ and
$\gamma^{(2)}$, $\gamma^{(1)}<\gamma^{(2)}$, be real zeroes of
$\varphi''$ such that $\varphi(z)\neq0$ and $\varphi'(z)\neq0$ for
$z\in[\gamma^{(1)},\gamma^{(2)}]$ and suppose that $\varphi''$ has
exactly $q\geqslant2$ zeroes, counting multiplicities, in the
interval~$[\gamma^{(1)},\gamma^{(2)}]$. Then $Q$ has an even
number of zeroes in $[\gamma^{(1)},\gamma^{(2)}]$ and one of the
following holds:
\begin{itemize}
\item[\emph{I.}] If $q$ is an odd number, then
\begin{equation}\label{intermediate.11.lemma.6}
0\leqslant Z_{[\gamma^{(1)},\gamma^{(2)}]}(Q)\leqslant
Z_{[\gamma^{(1)},\gamma^{(2)}]}(Q_{1}).
\end{equation}
\item[\emph{II.}] If $q$ is an even number and
$\varphi(\gamma^{(1)}-\varepsilon)\varphi''(\gamma^{(1)}-\varepsilon)>0$
for all sufficiently small~$\varepsilon>0$, then
\begin{equation}\label{corol.4.intermediate.10}
0\leqslant Z_{[\gamma^{(1)},\gamma^{(2)}]}(Q)\leqslant-1+
Z_{[\gamma^{(1)},\gamma^{(2)}]}(Q_{1}).
\end{equation}
\item[\emph{III.}] If $q$ is an even number and
$\varphi(\gamma^{(1)}-\varepsilon)\varphi''(\gamma^{(1)}-\varepsilon)<0$
for all sufficiently small~$\varepsilon>0$, then
\begin{equation}\label{lemma6.5.main.1}
0\leqslant Z_{[\gamma^{(1)},\gamma^{(2)}]}(Q)\leqslant1+
Z_{[\gamma^{(1)},\gamma^{(2)}]}(Q_{1}).
\end{equation}
\end{itemize}
\end{lemma}
\begin{proof} We assume that $\varphi''$
has exactly~$r\leqslant q$ \textit{distinct} zeroes in the
interval~$[\gamma^{(1)},\gamma^{(2)}]$, say
$\gamma^{(1)}=\gamma_1<\gamma_2<\ldots<\gamma_{r-1}<\gamma_r=\gamma^{(2)}$.
Below in the proof, we denote the
interval~$[\gamma^{(1)},\gamma^{(2)}]$ by~$[\gamma_1,\gamma_r]$.

From~\eqref{main.function.2} it follows that
\begin{equation}\label{intermediate.10.lemma.6}
Q(\gamma_i)=-\dfrac{\left(\varphi'(\gamma_i)\right)^2}{\left(\varphi(\gamma_i)\right)^2}<0,
\end{equation}
since $\varphi(z)\neq0$ and $\varphi'(z)\neq0$
in~$[\gamma_1,\gamma_r]$ by assumption. Thus, $Q$ has an even
number of zeroes in the interval~$[\gamma_1,\gamma_r]$.

\vspace{2mm}

\noindent I. Let $q$ be an odd number.

\vspace{2mm}

If $\varphi(z)\varphi''(z)\leqslant0$ for
$z\in[\gamma_1,\gamma_r]$, then by~\eqref{main.function.2}, we
have the inequalities
\begin{equation*}
0=Z_{[\gamma_{1},\gamma_{r}]}(Q)\leqslant
Z_{[\gamma_{1},\gamma_{r}]}(Q_{1}),
\end{equation*}
which are exactly~\eqref{intermediate.11.lemma.6}.

\vspace{2mm}

Let now the function $\varphi(z)\varphi''(z)$ be positive on one
of the intervals $(\gamma_{i-1},\gamma_{i})$. Suppose that, for
some $i$, $2\leqslant i\leqslant r$, $\varphi(z)\varphi''(z)>0$ in
the interval~$(\gamma_{i-1},\gamma_{i})$, then
by~Lemma~\ref{lemma.5}, we have
\begin{equation}\label{intermediate.1.lemma.6}
Z_{(\gamma_{i-1},\gamma_i)}(Q)\leqslant1+Z_{(\gamma_{i-1},\gamma_i)}(Q_{1}).
\end{equation}
Moreover, there can be only two possibilities:
\begin{itemize}
\item[I.1.] If the number $\gamma_i$ is a zero of $\varphi''$ of
\textit{even} multiplicity (at least two), then according
to~\eqref{all.main.functions}, $\gamma_i$ is a zero of $Q_{1}$ of
multiplicity \textit{at least} one, but $Q(\gamma_i)\neq0$
by~\eqref{intermediate.10.lemma.6}. From these facts and from the
inequality~\eqref{intermediate.1.lemma.6} we obtain
\begin{equation}\label{intermediate.2.lemma.6}
Z_{(\gamma_{i-1},\gamma_i]}(Q)\leqslant
Z_{(\gamma_{i-1},\gamma_i]}(Q_{1}).
\end{equation}
For the sequel, we notice that, in this case, $\varphi''$ has an
\textit{even} number of zeroes, counting multiplicities, in
the~interval~$(\gamma_{i-1},\gamma_i]$.
\item[I.2.] If the number $\gamma_i$, $i<r$, is a zero of $\varphi''$ of
\textit{odd} multiplicity (at least one), then
$\varphi(z)\varphi''(z)$ is nonpositive
in~$[\gamma_{i},\gamma_{i+1}]$, so $Q(z)\neq0$ in this interval
by~\eqref{main.function.2}. But $Q_{1}$ has an odd number (at
least one) of zeroes on the interval~$(\gamma_i,\gamma_{i+1})$
according to Lemma~\ref{lemma.1} (Case~I) applied to $Q_1$,
therefore,
\begin{equation*}
0=Z_{[\gamma_{i},\gamma_{i+1}]}(Q)\leqslant
Z_{[\gamma_{i},\gamma_{i+1}]}(Q_{1})-1.
\end{equation*}
This inequality and the inequality~\eqref{intermediate.1.lemma.6}
imply
\begin{equation}\label{intermediate.3.lemma.6}
Z_{(\gamma_{i-1},\gamma_{i+1}]}(Q)\leqslant
Z_{(\gamma_{i-1},\gamma_{i+1}]}(Q_{1}).
\end{equation}
Suppose that, for some number $j$, $i+1\leqslant j\leqslant r$,
the zeroes~$\gamma_{i+1},\ldots,\gamma_{j-1}$ of $\varphi''$ are
all of even multiplicities, and $\gamma_j$ is a zero of
$\varphi''$ of an odd multiplicity. 
Then we have $\varphi(z)\varphi''(z)\leqslant0$ and, therefore,
$Q(z)<0$ for $z\in(\gamma_{i+1},\gamma_j]$ according
to~\eqref{main.function.2}. But by Lemma~\ref{lemma.1} applied to
$Q_1$, the function $Q_1$ has \textit{at least} one zero on each
interval
$(\gamma_{i+1},\gamma_{i+2}),\ldots,(\gamma_{j-1},\gamma_j)$.
This fact and the~inequality~\eqref{intermediate.3.lemma.6} imply
\begin{equation}\label{intermediate.4.lemma.6}
Z_{(\gamma_{i-1},\gamma_{j}]}(Q)\leqslant
Z_{(\gamma_{i-1},\gamma_{j}]}(Q_{1}).
\end{equation}
As well as in Case I.1, we notice that~$\varphi''$ has an
\textit{even} number of zeroes, counting multiplicities, in the
interval~$(\gamma_{i-1},\gamma_j]$. It follows from the fact that
$\gamma_i$ and $\gamma_j$ are zeroes of~$\varphi''$ of odd
multiplicities and~$\gamma_{i+1},\ldots,\gamma_{j-1}$ are zeroes
of~$\varphi''$ of even multiplicities.
\end{itemize}

Let now $l$, $1\leqslant l\leqslant r-1$, be an integer such that
$\varphi(z)\varphi''(z)$ is nonpositive in the interval
$[\gamma_1,\gamma_l]$ and is positive in
the~interval~$(\gamma_l,\gamma_{l+1})$, thus $\gamma_l$ is a zero
of $\varphi''$ of an odd multiplicity if $l\geqslant2$. Then
by~\eqref{main.function.2}, $Q(z)\neq0$ for
$z\in[\gamma_1,\gamma_l]$, so we have the following inequality
\begin{equation}\label{intermediate.5.lemma.6}
0=Z_{[\gamma_{1},\gamma_{l}]}(Q)\leqslant
Z_{[\gamma_{1},\gamma_{l}]}(Q_{1}),
\end{equation}
since~$Q_1$ has at least~$l-1\,(\geqslant0)$ zeroes
in~$[\gamma_{1},\gamma_{l}]$ by Lemma~\ref{lemma.1} applied to
$Q_1$ on the intervals
$(\gamma_1,\gamma_2),\ldots,(\gamma_{l-1},\gamma_{l})$.

\vspace{2mm}

Now we consider various possibilities of behaviour of the function
$\varphi$ on the interval~$(\gamma_{r-1},\gamma_r)$ in order to
establish the inequality~\eqref{intermediate.11.lemma.6}. There
can be only three possibilities:
\begin{itemize}
\item[a)] Let either $\gamma_r$ be a zero of $\varphi''$ of even multiplicity or
$\varphi(z)\varphi''(z)<0$ in~$(\gamma_{r-1},\gamma_r)$. Then the
interval~$(\gamma_{l},\gamma_r]$ consists only of the subintervals
described in Cases I.1 and I.2. Consequently, from the
inequalities~\eqref{intermediate.2.lemma.6},~\eqref{intermediate.4.lemma.6},~\eqref{intermediate.5.lemma.6}
we obtain
\begin{equation}\label{intermediate.6.lemma.6}
Z_{[\gamma_{1},\gamma_{r}]}(Q)\leqslant
Z_{[\gamma_{1},\gamma_{r}]}(Q_{1}).
\end{equation}
\item[b)] Let $\gamma_r$ be a zero of $\varphi''$ of odd multiplicity and
let $\varphi(z)\varphi''(z)>0$ in~$(\gamma_{r-1},\gamma_r)$ and
let $l\geqslant2$. Then the interval~$(\gamma_{l},\gamma_{r-1}]$
consists only of subintervals described in Cases I.1 and I.2 and,
therefore,
\begin{equation}\label{intermediate.7.lemma.6}
Z_{(\gamma_{l},\gamma_{r-1}]}(Q)\leqslant
Z_{(\gamma_{l},\gamma_{r-1}]}(Q_{1}).
\end{equation}
Moreover, since $l\geqslant2$ by assumption, we have $Q(z)\neq0$
for $z\in[\gamma_1,\gamma_l]$, but $Q_{1}$ has \textit{at~least}
one zero in $[\gamma_1,\gamma_l]$ as we mentioned above.
Consequently, in this case, we can improve the
inequality~\eqref{intermediate.5.lemma.6} to the following one:
\begin{equation}\label{lemma.6.proof.10}
0=Z_{[\gamma_{1},\gamma_{l}]}(Q)\leqslant
Z_{[\gamma_{1},\gamma_{l}]}(Q_{1})-1.
\end{equation}
Furthermore, we have
\begin{equation}\label{lemma.6.proof.11}
Z_{(\gamma_{r-1},\gamma_r]}(Q)\leqslant1+Z_{(\gamma_{r-1},\gamma_r]}(Q_{1})
\end{equation}
by Lemma~\ref{lemma.5} and by the fact that $Q(\gamma_r)\neq0$
(see~\eqref{intermediate.10.lemma.6}). Summing the
inequalities~\eqref{intermediate.7.lemma.6}--\eqref{lemma.6.proof.11},
we again obtain~\eqref{intermediate.6.lemma.6}.
\item[c)] Let, finally, $\gamma_r$ be a zero of $\varphi''$ of odd
multiplicity and let $\varphi(z)\varphi''(z)>0$
in~$(\gamma_{r-1},\gamma_r)$, but let $l=1$. In this case, by the
same reasoning as above, the
inequalities~\eqref{intermediate.7.lemma.6} (for $l=1$)
and~\eqref{lemma.6.proof.11} hold. Recalling the final remarks in
Cases I.1 and I.2, we conclude that $\varphi''$ has an
\textit{even} number of zeroes in the interval
$(\gamma_{1},\gamma_{r-1}]$. Since $\varphi''$ has an odd
number~$q$ of zeroes in $[\gamma_1,\gamma_r]$ and $\gamma_r$ is a
zero of~$\varphi''$ of odd multiplicity by assumption, $\gamma_1$
must be a zero of~$\varphi''$ of even multiplicity.
But~\eqref{all.main.functions} shows that every zero of
$\varphi''$ of even multiplicity is a zero of $Q_{1}$ of odd
multiplicity. Consequently, $\gamma_1$ is a zero of $Q_{1}$ of
multiplicity \textit{at least} one, and we have
\begin{equation}\label{intermediate.9.lemma.6}
0=Z_{\{\gamma_{1}\}}(Q)\leqslant Z_{\{\gamma_{1}\}}(Q_{1})-1.
\end{equation}
Summing the
inequalities~\eqref{intermediate.7.lemma.6},~\eqref{intermediate.7.lemma.6}
and~\eqref{intermediate.9.lemma.6}, we also obtain the
inequality~\eqref{intermediate.6.lemma.6}.
\end{itemize}

Thus, the number of zeroes of $Q$ in
the~interval~$[\gamma_{1},\gamma_{r}]$ does not exceed the number
of zeroes of $Q_{1}$ in this interval. This implies the validity
of the inequalities~\eqref{intermediate.11.lemma.6} where the
lower bound cannot be improved by parity considerations, since the
number of zeroes of $Q$ in the interval~$[\gamma_{1},\gamma_{r}]$
is even as we showed above (see~\eqref{intermediate.10.lemma.6}).

\vspace{2mm}

\noindent II. Now let $q=2M$ and
$\varphi(\gamma_{1}-\varepsilon)\varphi''(\gamma_{1}-\varepsilon)>0$
for all sufficiently small~$\varepsilon>0$.

At first, we assume that one of the $\gamma_i$, $i=1,\ldots,r$, is
a zero of $\varphi''$ of odd multiplicity. Let $\gamma_j$,
$1\leqslant j\leqslant r-1$, be\footnote{Since $\varphi''$ has an
\textit{even} number of real zeroes in~$[\gamma_1,\gamma_{r}]$ by
assumption, there is always \textit{even} number of
zeroes~of~$\varphi''$ of odd multiplicities
in~$[\gamma_1,\gamma_{r}]$. Thus, $\gamma_{r}$ cannot be the
closest to $\gamma_1$ (and thereby the only) zero of~$\varphi''$
of~odd~multiplicity.} the closest to $\gamma_1$ (possibly
$\gamma_1$ itself) zero of~$\varphi''$ of~odd multiplicity.
Since~$\varphi''$ has an odd number of zeroes in the
interval~$[\gamma_1,\gamma_j]$, $\varphi''$ has different signs in
the interval~$(\gamma_{j},\gamma_{j+1})$ and in the left-sided
neighbourhood of $\gamma_1$, where $\varphi(z)\varphi''(z)$ is
positive by assumption. Consequently, $\varphi(z)\varphi''(z)$ is
negative in~$(\gamma_{j},\gamma_{j+1})$, because $\varphi(z)\neq0$
for $z\in[\gamma_1,\gamma_r]$ by assumption. Hence, according
to~\eqref{main.function.2}, $Q(z)\neq0$ in the
interval~$(\gamma_{j},\gamma_{j+1})$, and we have
\begin{equation}\label{corol.4.intermediate.7}
0=Z_{(\gamma_{j},\gamma_{j+1})}(Q)\leqslant
-1+Z_{(\gamma_{j},\gamma_{j+1})}(Q_{1}),
\end{equation}
since $Q_{1}$ has an odd number (at least one) of zeroes, counting
multiplicities, in $(\gamma_{j},\gamma_{j+1})$ by
Lemma~\ref{lemma.1} applied to $Q_1$ on the interval
$(\gamma_{j},\gamma_{j+1})$. On the other hand, inasmuch
as~$\varphi''$ has an odd number of zeroes, counting
multiplicities, in both the intervals~$[\gamma_1,\gamma_{j}]$
and~$[\gamma_{j+1},\gamma_{r}]$, it follows from Case I
that\footnote{If $\gamma_j=\gamma_1$ and $\gamma_1$ is a zero
of~$\varphi''$ of multiplicity $K>1$, then $\gamma_1$ is a zero
of~$Q_{1}$ of multiplicity $K-1$ according
to~\eqref{all.main.functions}. If $\gamma_1$ is a simple zero
of~$\varphi''$, then $\gamma_1$ is not a zero of~$Q_{1}$. In both
cases,~\eqref{corol.4.intermediate.8} reduces to
the~following~inequality
\begin{equation*}
0=Z_{\{\gamma_1\}}(Q)\leqslant Z_{\{\gamma_1\}}(Q_{1}).
\end{equation*}
}
\begin{equation}\label{corol.4.intermediate.8}
0\leqslant Z_{[\gamma_1,\gamma_{j}]}(Q)\leqslant
Z_{[\gamma_1,\gamma_{j}]}(Q_{1}),
\end{equation}
\begin{equation}\label{corol.4.intermediate.9}
0\leqslant Z_{[\gamma_{j+1},\gamma_{r}]}(Q)\leqslant
Z_{[\gamma_{j+1},\gamma_{r}]}(Q_{1}).
\end{equation}
The
inequalities~\eqref{corol.4.intermediate.7}--\eqref{corol.4.intermediate.9}
together give the inequality~\eqref{corol.4.intermediate.10}.

\vspace{2mm}

Now we assume that all $\gamma_i$, $i=1,\ldots,r$, are zeroes
of~$\varphi''$ of even multiplicities. Therefore,
\begin{equation}\label{corol.4.intermediate.9999}
1\leqslant r\leqslant M
\end{equation}
and $\varphi''$ has equal signs in all intervals
$(\gamma_{i},\gamma_{i+1})$, $i=1,\ldots,r-1$. Since
$\varphi(z)\varphi''(z)$ is positive in a small left-sided
neighbourhood of $\gamma_1$ by assumption, it is positive in each
of the intervals~$(\gamma_{i},\gamma_{i+1})$, $i=1,\ldots,r-1$.
Thus, we have
\begin{equation}\label{corol.4.intermediate.11}
0\leqslant Z_{X_1}(Q)\leqslant r-1+Z_{X_1}(Q_{1})\leqslant
M-1+Z_{X_1}(Q_{1}),
\end{equation}
where $X_1=\bigcup_{i=1}^{r-1}(\gamma_{i},\gamma_{i+1})$. Indeed,
since $Q(\gamma_i)<0$ for $i=1,\ldots,r$
(see~\eqref{intermediate.10.lemma.6}), $Q$ has an even number of
zeroes in each of the interval~$(\gamma_i,\gamma_{i+1})$. Thus, we
cannot improve the lower estimate
in~\eqref{corol.4.intermediate.11} by parity considerations. The
upper bound follows from Lemma~\ref{lemma.5} applied to the
intervals $(\gamma_{i},\gamma_{i+1})$ and
from~\eqref{corol.4.intermediate.9999}. Further,~$Q_{1}$ has
exactly $2M-r$ zeroes, counting multiplicities, among the points
$X_2=\bigcup_{i=1}^{r}\{\gamma_i\}$. Consequently,
by~\eqref{corol.4.intermediate.9999}, we have
\begin{equation}\label{corol.4.intermediate.12}
Z_{X_2}(Q_{1})=2M-r\geqslant M.
\end{equation}
Now we note that $[\gamma_1,\gamma_r]=X_1\bigcup X_2$ and
$Z_{X_2}(Q)=0$. Therefore,
from~\eqref{corol.4.intermediate.11}--\eqref{corol.4.intermediate.12}
it follows that
\begin{equation*}
\begin{array}{c}
0\,\leqslant\,Z_{[\gamma_1,\gamma_r]}(Q)\,=\,Z_{X_1}(Q)\,\leqslant\,
M-1+Z_{X_1}(Q_{1})\,=\\
 \\
=\,M-1+Z_{[\gamma_1,\gamma_r]}(Q_{1})-Z_{X_2}(Q_{1})\,\leqslant\,-1+Z_{[\gamma_1,\gamma_r]}(Q_{1}),
\end{array}
\end{equation*}
where the lower bound cannot be improved by parity considerations,
since $Q$ has an even number of zeroes in the
interval~$[\gamma_{1},\gamma_{r}]$ as we showed above
(see~\eqref{intermediate.10.lemma.6}). So, the
inequalities~\eqref{corol.4.intermediate.10} are also valid in
this case.

\vspace{2mm}

\noindent III. At last, let $q=2M$ and let
\begin{equation}\label{new.lemma.6}
\varphi(\gamma_{1}-\varepsilon)\varphi''(\gamma_{1}-\varepsilon)<0
\end{equation}
for sufficiently small $\varepsilon>0$. If $\varphi''$ has zeroes
of odd multiplicities in the interval $[\gamma_1,\gamma_r]$, then
this case differs from Case II only by the sign of the function
$\varphi\varphi''$ in the interval $(\gamma_j,\gamma_{j+1})$,
where~$\gamma_j$ is the closest to $\gamma_1$ (possibly~$\gamma_1$
itself) zero of~$\varphi''$ of odd multiplicity. Thus, now we have
$\varphi(z)\varphi''(z)>0$ in~$(\gamma_j,\gamma_{j+1})$, since
$\varphi''$ has an odd number of zeroes in~$[\gamma_1,\gamma_j]$
(see~\eqref{new.lemma.6}). Therefore, instead of the
inequalities~\eqref{corol.4.intermediate.7} of Case II, we have
\begin{equation}\label{lemma6.5.intermediate.1}
0\leqslant Z_{(\gamma_{j},\gamma_{j+1})}(Q)\leqslant
1+Z_{(\gamma_{j},\gamma_{j+1})}(Q_{1})
\end{equation}
by Lemma~\ref{lemma.5}. As in Cases I and II, the lower bound in~\eqref{lemma6.5.intermediate.1}
cannot be improved by parity considerations.
The~inequalities~\eqref{corol.4.intermediate.8},
\eqref{corol.4.intermediate.9} and~\eqref{lemma6.5.intermediate.1}
imply~\eqref{lemma6.5.main.1}.

\vspace{2mm}

If all $\gamma_i$, $i=1,\ldots,r$, are zeroes of~$\varphi''$ of
even multiplicities, then~\eqref{new.lemma.6} implies
$\varphi(z)\varphi''(z)\leqslant0$ in $[\gamma_1,\gamma_r]$, so we
have
\begin{equation}\label{new.2.lemma6}
0=Z_{[\gamma_1,\gamma_r]}(Q)\leqslant
Z_{[\gamma_1,\gamma_r]}(Q_{1}),
\end{equation}
which also gives~\eqref{lemma6.5.main.1}.
\end{proof}

Further, we derive a very useful bound on the number of real
zeroes of $Q$ in terms of the number of real zeroes of $Q_1$
between two consecutive zeroes of $\varphi'$.
\begin{theorem}\label{Theorem.conseq.deriv.zeroes}
Let $\varphi\in U_{2n}^*$. If $\beta_1$ and $\beta_2$,
$\beta_1<\beta_2$, are consecutive real zeroes of~$\varphi'$ and
$\varphi(z)\neq0$ for $z\in[\beta_1,\beta_2]$, then
\begin{equation}\label{main.work.formula.2}
1\leqslant
Z_{(\beta_1,\beta_2)}(Q)\leqslant1+Z_{(\beta_1,\beta_2)}(Q_{1}).
\end{equation}
\end{theorem}
\begin{proof} By Lemma~\ref{lemma.1}, $Q$ has an odd number of
zeroes in the interval $(\beta_1,\beta_2)$, therefore,
\begin{equation}\label{Theorem.conseq.deriv.zeroes.proof.1}
1\leqslant Z_{(\beta_1,\beta_2)}(Q).
\end{equation}
According to Rolle's theorem, $\varphi''$ has an odd number of
real zeroes, counting multiplicities, in the interval
$(\beta_1,\beta_2)$. Let $\gamma_S$ and $\gamma_L$ be the smallest
and the largest zeroes of~$\varphi''$ in~$(\beta_1,\beta_2)$. If
$\gamma_S<\gamma_L$, then from Case~I of Lemma~\ref{lemma.6.5} it
follows that
\begin{equation}\label{intermediate.11.lemma.6.1}
0\leqslant Z_{[\gamma_S,\gamma_L]}(Q)\leqslant
Z_{[\gamma_S,\gamma_L]}(Q_{1}).
\end{equation}
If $\gamma_S=\gamma_L$, then (see~\eqref{intermediate.10.lemma.6})
\begin{equation}\label{Theorem.conseq.deriv.zeroes.proof.2}
0=Z_{\{\gamma_S\}}(Q)\leqslant Z_{\{\gamma_S\}}(Q_{1}).
\end{equation}

Without loss of generality, we may assume that the function
$\varphi(z)\varphi''(z)$ is positive in the
interval~$(\beta_1,\gamma_S)$. Then $\varphi(z)\varphi''(z)<0$ for
$z\in(\gamma_L,\beta_2)$, since~$\varphi''$ has an~odd number of
zeroes in~$[\gamma_S,\gamma_L]$ and $\varphi(z)\neq0$ in
$(\beta_1,\beta_2)$ by assumption. Thus,
by~\eqref{main.function.2}, we have
\begin{equation}\label{intermediate.13.lemma.6}
0=Z_{(\gamma_L,\beta_{2})}(Q)\leqslant
Z_{(\gamma_L,\beta_{2})}(Q_{1}).
\end{equation}
We cannot improve this inequality by parity considerations, since
$Q_1$ has an even number of zeroes in $(\gamma_L,\beta_{2})$ by
Lemma~\ref{lemma.1} applied to $Q_1$ on the
interval~$(\gamma_L,\beta_{2})$. Lemma~\ref{lemma.5} now gives
\begin{equation}\label{intermediate.12.lemma.6}
Z_{(\beta_1,\gamma_S)}(Q)\leqslant1+
Z_{(\beta_1,\gamma_S)}(Q_{1}).
\end{equation}
The
inequalities~\eqref{Theorem.conseq.deriv.zeroes.proof.1}--\eqref{intermediate.12.lemma.6}
imply~\eqref{main.work.formula.2}.
\end{proof}

The following corollary is a very useful generalization of
Theorem~\ref{Theorem.conseq.deriv.zeroes}.
\begin{corol}\label{biggest.technical.theorem}
Let $\varphi\in U_{2n}^*$ and let $\beta^{(1)}$ and $\beta^{(2)}$,
$\beta^{(1)}\leqslant\beta^{(2)}$, be zeroes of $\varphi'$ and let
$\varphi(z)\neq0$ for $z\in[\beta^{(1)},\beta^{(2)}]$. If
$\varphi'$ has exactly $q\geqslant2$ real zeroes, counting
multiplicities, in the interval $[\beta^{(1)},\beta^{(2)}]$, then
\begin{equation}\label{main.work.formula.5}
q-1\leqslant Z_{[\beta^{(1)},\beta^{(2)}]}(Q)\leqslant
q-1+Z_{[\beta^{(1)},\beta^{(2)}]}(Q_{1}).
\end{equation}
\end{corol}
\begin{proof}
In fact, if some real number $\beta$ is a~zero of $\varphi'$ of
multiplicity $M$, then $\beta$ is a~zero of $Q$ of
multiplicity~$M-1$ according to~\eqref{main.function.2}.
Therefore, the following inequalities hold:
\begin{equation}\label{biggest.technical.theorem.proof.1}
M-1\leqslant Z_{\{\beta\}}(Q)\leqslant M-1+Z_{\{\beta\}}(Q_{1}).
\end{equation}
Thus, if the function $\varphi'$ has exactly $l$, $1\leqslant
l\leqslant q$, \textit{distinct} real zeroes, say
$\beta^{(1)}=\beta_1<\beta_2<\ldots<\beta_{l}=\beta^{(2)}$, in
the~interval~$[\beta^{(1)},\beta^{(2)}]$, then $Q$ has exactly
$q-l$ real zeroes at the points~$\beta_i$, $i=1,2,\ldots,l$, so
from~\eqref{biggest.technical.theorem.proof.1} it follows
\begin{equation}\label{biggest.technical.theorem.proof.4}
q-l\leqslant Z_{X_2}(Q)\leqslant q-l+Z_{X_2}(Q_{1}),
\end{equation}
where $X_2=\bigcup_{i=1}^{l}\{\beta_i\}$. Note that
$Z_{X_2}(Q_{1})=0$ indeed, because $X_2$ is the set of all poles
of the function $Q_1$ on the interval $[\beta^{(1)},\beta^{(2)}]$.

\vspace{2mm}

If $l=1$, then $\beta^{(1)}=\beta^{(2)}$ is a zero of $\varphi'$
of multiplicity $q$. By~\eqref{biggest.technical.theorem.proof.4},
we have the following inequalities
\begin{equation}\label{biggest.technical.theorem.proof.2}
q-1\leqslant Z_{\{\beta^{(1)}\}}(Q)\leqslant
q-1+Z_{\{\beta^{(1)}\}}(Q_{1}),
\end{equation}
which are equivalent to~\eqref{main.work.formula.5}, since
$Z_{\{\beta^{(1)}\}}(Q)=Z_{[\beta^{(1)},\beta^{(2)}]}(Q)$ and
$Z_{\{\beta^{(1)}\}}(Q_1)=Z_{[\beta^{(1)},\beta^{(2)}]}(Q_1)$ in
this case.

\vspace{2mm}

Now let $l>1$. Then the points $\beta_i$, $i=1,2,\ldots,l$, divide
the interval $(\beta_1,\beta_l)=(\beta^{(1)},\beta^{(2)})$ into
$l-1$ subintervals $(\beta_i,\beta_{i+1})$, $i=1,2,\ldots,l-1$.
Since $\varphi(z)\neq0$ on $[\beta_i,\beta_{i+1}]$,
$i=1,\ldots,l-1$, by assumption, we can apply
Theorem~\ref{Theorem.conseq.deriv.zeroes} to each of those
subintervals to obtain
\begin{equation}\label{biggest.technical.theorem.proof.3}
l-1\leqslant Z_{X_1}(Q)\leqslant l-1+Z_{X_1}(Q_{1}),
\end{equation}
where $X_1=\bigcup_{i=1}^{l-1}(\beta_i,\beta_{i+1})$.

\vspace{2mm}

Since $X_1\cup X_2=[\beta^{(1)},\beta^{(2)}]$, we
get~\eqref{main.work.formula.5} by summing the
inequalities~\eqref{biggest.technical.theorem.proof.4}
and~\eqref{biggest.technical.theorem.proof.3} or, if
$\beta^{(1)}=\beta^{(2)}$, directly
from~\eqref{biggest.technical.theorem.proof.2}.
\end{proof}

We now analyze a relation between the number of real zeroes of $Q$
and the number of real zeroes of $Q_1$ in an interval adjacent to
a zero of $\varphi$. This analysis justifies the introduction of
\textit{property~A} (Definition~\ref{def.cond.A.point}).
\begin{lemma}\label{lemma.7}
Let $\varphi\in U_{2n}^*$ and let $\beta$ be a real zero of
$\varphi'$ and let $\alpha>\beta$ be a real zero of $\varphi$ such
that $\varphi'(z)\neq0$ and $\varphi(z)\neq0$ for
$z\in(\beta,\alpha)$.
\begin{itemize}
\item [\emph{I.}] If $Q_{1}$ has an even number of zeroes in
the interval $(\beta,\alpha)$, then
\begin{equation}\label{main.work.formula.3}
0\leqslant Z_{(\beta,\alpha)}(Q)\leqslant
Z_{(\beta,\alpha)}(Q_{1}).
\end{equation}
\item[\emph{II.}] If $Q_{1}$ has an odd number of zeroes in
the interval $(\beta,\alpha)$, then
\begin{equation}\label{main.work.formula.3.5}
0\leqslant Z_{(\beta,\alpha)}(Q)\leqslant1+
Z_{(\beta,\alpha)}(Q_{1}).
\end{equation}
\end{itemize}
\end{lemma}
\begin{proof}
According to Lemma~\ref{lemma.1}, $Q$ has an even number of zeroes
in the interval $(\beta,\alpha)$. In particular, $Q$ can have no
zeroes in this interval. Thus, the lower bounds for
$Z_{(\beta,\alpha)}(Q)$ in~\eqref{main.work.formula.3} and
in~\eqref{main.work.formula.3.5} cannot be improved by parity
considerations.

\vspace{2mm}

\noindent I. Let $Q_{1}$ have an even number of zeroes
in~$(\beta_1,\alpha)$ (or have no zeroes at all in this interval).

\begin{itemize}

\item [I.1.] At first, we assume that $\varphi''$ has an even
number of zeroes, counting multiplicities, in~$(\beta,\alpha)$,
say $\beta<\gamma_1\leqslant\ldots\leqslant\gamma_{2M}<\alpha$,
where $M\geqslant0$. It is easy to see that, for sufficiently
small~$\varepsilon>0$, we have the following inequality
\begin{equation}\label{lemma.8.new.1}
\varphi(\beta+\varepsilon)\varphi''(\beta+\varepsilon)<0.
\end{equation}
In fact, if $\varphi(z)$ is positive (negative) in
$(\beta,\alpha)$, then it is a decreasing (increasing) function,
since $\varphi(\alpha)=0$. Consequently, $\varphi(z)\varphi'(z)<0$
in the interval $(\beta,\alpha)$. By the same reasoning, we have
$\varphi'(\beta+\varepsilon)\varphi''(\beta+\varepsilon)>0$ for
all sufficiently small $\varepsilon>0$. Thus, the
inequality~\eqref{lemma.8.new.1} is true.

Let $\varphi''$ have zeroes in the interval $(\beta,\alpha)$, that
is, let $M>0$, then from~\eqref{lemma.8.new.1} it follows that
$Q(z)\neq0$ for $z\in(\beta,\gamma_1)$
(see~\eqref{main.function.2}), so we have
\begin{equation}\label{lemma.7.intermediate.1}
0=Z_{(\beta,\gamma_1)}(Q)\leqslant Z_{(\beta,\gamma_1)}(Q_{1}).
\end{equation}

If $\gamma_1<\gamma_{2M}$, then by~\eqref{lemma.8.new.1} and
Case~III of Lemma~\ref{lemma.6.5}, we obtain
\begin{equation}\label{lemma.7.intermediate.6}
0\leqslant Z_{[\gamma_1,\gamma_{2M}]}(Q)\leqslant1+
Z_{[\gamma_1,\gamma_{2M}]}(Q_{1}).
\end{equation}

Moreover, from~\eqref{lemma.8.new.1} it also follows that
$\varphi(z)\varphi''(z)<0$ in~$(\gamma_{2M},\alpha)$, since
$\varphi''$ has equal signs in the~intervals~$(\beta,\gamma_1)$
and~$(\gamma_{2M},\alpha)$. Therefore,
\begin{equation}\label{lemma.7.intermediate.7}
0=Z_{(\gamma_{2M},\alpha)}(Q)\leqslant
Z_{(\gamma_{2M},\alpha)}(Q_{1}).
\end{equation}
Thus,
from~\eqref{lemma.7.intermediate.1}--\eqref{lemma.7.intermediate.7}
we obtain
\begin{equation}\label{lemma.7.intermediate.7.5}
0\leqslant Z_{(\beta,\alpha)}(Q)\leqslant1+
Z_{(\beta,\alpha)}(Q_{1}),
\end{equation}
which is equivalent to~\eqref{main.work.formula.3}, since the
number $Z_{(\beta,\alpha)}(Q)$ is even by Lemma~\ref{lemma.1} and
$Z_{(\beta_,\alpha)}(Q_{1})$ is even by assumption.

Let now $\gamma_1=\gamma_{2M}$, that is, let $\gamma_1$ be a
unique zero of $\varphi''$ of multiplicity~$2M$
in~$(\beta,\alpha)$. Since $Z_{\{\gamma_1\}}(Q)=0$
by~\eqref{intermediate.10.lemma.6} and $Q_1$ has a zero of
multiplicity $2M-1$ at the point $\gamma_1$
(see~\eqref{all.main.functions}), we obtain the following
inequality
\begin{equation}\label{lemma.7.intermediate.7.2}
0=Z_{\{\gamma_1\}}(Q)\leqslant1+Z_{\{\gamma_1\}}(Q_{1})=2M.
\end{equation}

Likewise, the inequality~\eqref{lemma.7.intermediate.7} holds for
the same reasoning as in the case $\gamma_1<\gamma_{2M}$.

Thus, the
inequalities~\eqref{lemma.7.intermediate.1},~\eqref{lemma.7.intermediate.7},~\eqref{lemma.7.intermediate.7.2}
imply~\eqref{lemma.7.intermediate.7.5}, which is equivalent
to~\eqref{main.work.formula.3}, since the number
$Z_{(\beta,\alpha)}(Q)$ is even by Lemma~\ref{lemma.1} and
$Z_{(\beta_,\alpha)}(Q_{1})$ is even by assumption.

\vspace{2mm}

At last, let $\varphi''(z)\neq0$ for $z\in(\beta,\alpha)$. Then
from~\eqref{lemma.8.new.1} and~\eqref{main.function.2} it follows
that
\begin{equation}\label{lemma.7.intermediate.3}
0=Z_{(\beta,\alpha)}(Q)\leqslant Z_{(\beta,\alpha)}(Q_{1}),
\end{equation}
which is exactly~\eqref{main.work.formula.3}.

Note that the upper bounds in~\eqref{lemma.7.intermediate.1}
and~\eqref{lemma.7.intermediate.3} cannot be improved by parity
considerations, since $Q_{1}$ has an even number of zeroes on both
intervals~$(\beta,\gamma_1)$ and $(\beta,\alpha)$.
\item [I.2.] Now we assume that $\varphi''$ has an odd number of
zeroes, counting multiplicities, in the interval~$(\beta,\alpha)$,
say $\beta<\gamma_1\leqslant\ldots\leqslant\gamma_{2M+1}<\alpha$,
where $M\geqslant0$.  As in Case I.1, the
inequalities~\eqref{lemma.8.new.1}
and~\eqref{lemma.7.intermediate.1} hold.

If $\gamma_1<\gamma_{2M+1}$, then from Case~I of
Lemma~\ref{lemma.6.5} it follows that
\begin{equation}\label{lemma.7.intermediate.16}
0\leqslant Z_{[\gamma_1,\gamma_{2M+1}]}(Q)\leqslant
Z_{[\gamma_1,\gamma_{2M+1}]}(Q_{1}),
\end{equation}

If $\gamma_1=\gamma_{2M+1}$, that is, $\gamma_1$ is a unique zero
of $\varphi''$ of multiplicity~$2M+1$ in the
interval~$(\beta,\alpha)$, then $Z_{\{\gamma_1\}}(Q)=0$
(see~\eqref{intermediate.10.lemma.6}) and $Q_1$ has a zero of
multiplicity $2M$ at the point $\gamma_1$
(see~\eqref{all.main.functions}). Therefore, we have
\begin{equation}\label{lemma.7.intermediate.17}
0=Z_{\{\gamma_1\}}(Q)\leqslant Z_{\{\gamma_1\}}(Q_{1})=2M.
\end{equation}

We note that $\varphi(z)\varphi''(z)>0$
in~$(\gamma_{2M+1},\alpha)$, since $\varphi''$ has different signs
in intervals~$(\beta,\gamma_1)$ and~$(\gamma_{2M+1},\alpha)$ and
the inequality~\eqref{lemma.8.new.1} holds. Therefore, by
Lemma~\ref{lemma.5}, we have
\begin{equation}\label{lemma.7.intermediate.18}
0\leqslant
Z_{(\gamma_{2M+1},\alpha)}(Q)\leqslant1+Z_{(\gamma_{2M+1},\alpha)}(Q_{1}).
\end{equation}
Note that the lower bound in~\eqref{lemma.7.intermediate.18}
cannot be improved by parity considerations, since $Q$ has an even
number of zeroes on the interval $(\gamma_{2M+1},\alpha)$ as it
follows from the inequality~\eqref{lemma.8.new.1} and from
Lemmata~\ref{lemma.1} and~\ref{lemma.6.5}.

Thus,
from~\eqref{lemma.7.intermediate.1},~\eqref{lemma.7.intermediate.16}--\eqref{lemma.7.intermediate.18}
we obtain the inequalities~\eqref{lemma.7.intermediate.7.5} again.
As above, these inequalities are equivalent
to~\eqref{main.work.formula.3}, since the number
$Z_{(\beta,\alpha)}(Q)$ is even by Lemma~\ref{lemma.1} and
$Z_{(\beta,\alpha)}(Q_{1})$ is an even number by assumption.
\end{itemize}

\vspace{2mm}

\noindent II. If $Q_{1}$ has an odd number of zeroes
in~$(\beta,\alpha)$, then, by the same argument as in Case I, we
obtain the inequalities~\eqref{lemma.7.intermediate.7.5}, which
coincide with the inequalities~\eqref{main.work.formula.3.5}. But
unlike in Case I, these inequalities cannot be improved by parity
consideration, since $Z_{(\beta,\alpha)}(Q_1)$ is odd by
assumption.
\end{proof}

\begin{note}\label{remark.2.5}
Lemma~\ref{lemma.7} is true if $\beta=-\infty$ and
$\varphi(z)\neq0$, $\varphi'(z)\neq0$ for $z\in(-\infty,\alpha)$
(see~Remark~\ref{remark.2.4}). Lemma~\ref{lemma.7} is also valid
in the case when $\beta>\alpha$ or in the interval
$(\alpha,+\infty)$.
\end{note}

Note that the upper bound in~\eqref{main.work.formula.3.5} cannot
be improved. This is confirmed by the following example
constructed recently by S.\,Edwards~\cite{Edwards}.
\begin{example}\label{example.Edwards}
The polynomial

\vspace{-2mm}
\begin{equation}\label{poly.example.Edwards}
p(z)=z(z^2-1/4)(z^2+1)^{25}
\end{equation}
\vspace{-2mm}
has $3$ real zeroes: $\pm\dfrac 12$ and $0$. Its derivative
$p'(z)$ has only two real zeroes: $\pm\beta$, where

\vspace{-1mm}
\begin{equation*}
\displaystyle\beta=\dfrac{\sqrt{4134+106\sqrt{2369}}}{212}\approx0.4547246059.
\end{equation*}
\vspace{-3mm}

In the interval $(-\beta,\beta)$, the polynomial $p$ has only one
simple zero $\alpha=0$. Straightforward computation shows that, in
this case, we have
$Z_{(-\beta,\alpha)}(Q)=Z_{(\alpha,\beta)}(Q)=2$ and
$Z_{(-\beta,\alpha)}(Q_1)=Z_{(\alpha,\beta)}(Q_1)=1$, where
$Q[p](z)=\left(p'(z)/p(z)\right)'$ and
$Q_1[p](z)=\left(p''(z)/p'(z)\right)'$. Therefore, for the
polynomial $p$, the inequalities~\eqref{main.work.formula.3.5}
have the form
\begin{equation*}
Z_{(-\beta,\alpha)}(Q)=Z_{(-\beta,\alpha)}(Q_1)+1,
\end{equation*}
\begin{equation*}
Z_{(\alpha,\beta)}(Q_1)=Z_{(\alpha,\beta)}(Q_1)+1.
\end{equation*}
%
\end{example}

For functions with \textit{property A}, Lemma~\ref{lemma.7} has
the following form.
\begin{theorem}\label{corol.new}
Let $\varphi\in U_{2n}^*$ and let $\beta_1$ and $\beta_2$,
$\beta_1<\beta_2$, be real zeroes of $\varphi'$ and let $\varphi$
have a unique real zero $\alpha$ in the interval
$(\beta_1,\beta_2)$ such that $\varphi'(z)\neq0$ for all
$z\in(\beta_1,\alpha)\cup(\alpha,\beta_2)$. If~$\varphi$ possesses
property A at its zero $\alpha$, then
\begin{equation}\label{main.work.formula.new}
0\leqslant Z_{(\beta_1,\beta_2)}(Q)\leqslant
Z_{(\beta_1,\beta_2)}(Q_{1}).
\end{equation}
\end{theorem}
\begin{proof} At first, we note that $\varphi(\beta_1)\varphi(\beta_2)\neq0$ by
Rolle's theorem. Consider now two situations: the number $\alpha$
is a zero of $Q_1$ of even multiplicity (possibly not a zero of
$Q_1$) and $\alpha$ is a zero of $Q_1$ of an odd multiplicity.

\vspace{2mm}

\noindent Case I. Let the number $\alpha$ be a zero of $Q_1$ of an
even multiplicity, including the situation $Q_1(\alpha)\neq0$. By
the same method as in the proof of Corollary~\ref{corol.3}, one
can show that~$Q_{1}$ has an even number of zeroes between two
consecutive zeroes of $\varphi'$, therefore, $Q_1$ has an even
number of zeroes in the interval $(\beta_1,\beta_2)$. Thus, there
can be only two possibilities:

\begin{itemize}
\item [I.1.] The function $Q_1$ has an even number of zeroes in
each of the intervals $(\beta_1,\alpha)$ and $(\alpha,\beta_2)$.
Then the inequalities~\eqref{main.work.formula.new}
follow\footnote{In this case, we do not use the fact that
$\varphi$ possesses property~A at its zero~$\alpha$.} from the
fact that $Z_{\{\alpha\}}(Q)\neq0$ and from
the~inequalities~\eqref{main.work.formula.3}, which hold in this
case according to Lemma~\ref{lemma.7} (see also
Remark~\ref{remark.2.5}).
\item [I.2.]  The function $Q_1$ has an odd number of zeroes in
each of the intervals $(\beta_1,\alpha)$ and $(\alpha,\beta_2)$.
Then Lemma~\ref{lemma.7} and Remark~\ref{remark.2.5} provide the
validity of the inequalities~\eqref{main.work.formula.3.5} in each
of the intervals~$(\beta_1,\alpha)$ and~$(\alpha,\beta_2)$:
\begin{equation}\label{corol.3.proof.1}
0\leqslant Z_{(\beta_1,\alpha)}(Q)\leqslant1+
Z_{(\beta_1,\alpha)}(Q_{1}),
\end{equation}
\begin{equation}\label{corol.3.proof.2}
0\leqslant Z_{(\alpha,\beta_2)}(Q)\leqslant1+
Z_{(\alpha,\beta_2)}(Q_{1}).
\end{equation}

Since $\varphi$ possesses property A at $\alpha$, we have
$Q(z)\neq0$ in at least one of the intervals $(\beta_1,\alpha)$
and~$(\alpha,\beta_2)$.

If $Q(z)\neq0$ for $z\in(\beta_1,\alpha)$, then
$Z_{(\beta_1,\alpha)}(Q)=0$. Since $Q_1$ has an odd number of
zeroes in the interval~$(\beta_1,\alpha)$ (hence at least one) by
assumption, the inequality~\eqref{corol.3.proof.1} can be improved
to yield
\begin{equation}\label{corol.new.prof.1}
0=Z_{(\beta_1,\alpha)}(Q)\leqslant Z_{(\beta_1,\alpha)}(Q_{1})-1.
\end{equation}
Together with~\eqref{corol.3.proof.2} and the fact that
$Z_{\{\alpha\}}(Q)\neq0$, this inequality
gives~\eqref{main.work.formula.new}.

If $Q(z)\neq0$ in the interval~$(\alpha,\beta_2)$, then
\eqref{main.work.formula.new} can be proved analogously.
\end{itemize}

\vspace{2mm}

\noindent Case II. Let the number $\alpha$ be a zero of $Q_1$ of
odd multiplicity (hence at least one).

By assumption, $\varphi$ possesses \textit{property~A} at
$\alpha$. At first, we assume that $Q(z)\neq0$ in the
interval~$(\beta_1,\alpha)$. As we mentioned in Case~I, $Q_1$ has
an even number of zeroes in the interval $(\beta_1,\beta_2)$, so
there can be only the following two situations:

\begin{itemize}
\item [II.1.] The function $Q_1$ has an odd number of zeroes in the
interval $(\beta_1,\alpha)$ and it has an~even number of zeroes in
$(\alpha,\beta_2)$. Then as above, we have the
inequality~\eqref{corol.new.prof.1} in the interval
$(\beta_1,\alpha)$ and the inequality
\begin{equation*}\label{corol.3.proof.3}
0\leqslant Z_{(\alpha,\beta_2)}(Q)\leqslant
Z_{(\alpha,\beta_2)}(Q_{1}),
\end{equation*}
which follows from Lemma~\ref{lemma.7} and
Remark~\ref{remark.2.5}. Together with~\eqref{corol.new.prof.1}
and the fact that $Z_{\{\alpha\}}(Q)\neq0$, this inequality
gives~\eqref{main.work.formula.new}.

\item [II.2.] The function $Q_1$ has an even number of zeroes in the
interval $(\beta_1,\alpha)$ and $Q_1$ has an~odd number of zeroes
in $(\alpha,\beta_2)$. Then the number $Z_{(\beta_1,\alpha)}(Q_1)$
can also be zero and we have the following inequality
\begin{equation}\label{corol.3.proof.4}
0=Z_{(\beta_1,\alpha)}(Q)\leqslant Z_{(\beta_1,\alpha)}(Q_{1}).
\end{equation}
By Lemma~\ref{lemma.7} and Remark~\ref{remark.2.5}, we obtain
the~inequality~\eqref{corol.3.proof.2}. But since $Q_1(\alpha)=0$
by assumption and $Q(\alpha)\neq0$ by~\eqref{main.function.2}, we
have the inequality
\begin{equation*}\label{corol.3.proof.5}
0=Z_{\{\alpha\}}(Q)\leqslant-1+Z_{\{\alpha\}}(Q_1),
\end{equation*}
which, together with the~inequality~\eqref{corol.3.proof.2},
implies
\begin{equation*}\label{corol.3.proof.6}
0\leqslant Z_{[\alpha,\beta_2)}(Q)\leqslant
Z_{[\alpha,\beta_2)}(Q_{1}).
\end{equation*}
From this inequality and the inequality~\eqref{corol.3.proof.4} we
again obtain~\eqref{main.work.formula.new}.
\end{itemize}

\noindent If $Q(z)\neq0$ in the interval~$(\alpha,\beta_2)$, then
the~inequality~\eqref{main.work.formula.new} can be proved
analogously.
\end{proof}

\begin{note}\label{remark.2.6}
Theorem~\ref{corol.new} remains valid if one of $\beta_1$ and
$\beta_2$ is infinite rather than a zero of $\varphi'$, that is,
if $(\beta_1,\beta_2)=(\beta_1,+\infty)$ or
$(\beta_1,\beta_2)=(-\infty,\beta_2)$ and the numbers
$Z_{(\beta_1,\beta_2)}(Q)$ and $Z_{(\beta_1,\beta_2)}(Q_1)$ are of
equal parities. More general, if, say, $\beta_2=+\infty$, then the
inequalities~\eqref{main.work.formula.new} must turn to the
following ones
\begin{equation}\label{main.work.formula.new.new}
0\leqslant Z_{(\beta_1,+\infty)}(Q)\leqslant1+
Z_{(\beta_1,+\infty)}(Q_{1}).
\end{equation}

However, for functions in $U_0^*=\mathcal{L-P^*}$ the
inequalities~\eqref{main.work.formula.new} are valid for
half-infinite intervals, since both functions $Q(z)$ and $Q_1(z)$
are negative for all sufficiently large real~$z$ (see the proof of
Theorem~\ref{Theorem.finitude.number.of.zeroes.of.Q} and
Proposition~\ref{prop.diff.U_2n*.function}), and therefore, the
numbers $Z_{(\beta_1,\beta_2)}(Q)$ and
$Z_{(\beta_1,\beta_2)}(Q_1)$ are of equal parities whether
$\beta_i$'s are finite or infinite (see also
Section~\ref{section:main.results.for.U_2n*}).
\end{note}
\begin{note}\label{remark.multiple.zeroes}
If $\alpha$ is a multiple zero of $\varphi\in U_{2n}^*$ and
$\beta_1$ and $\beta_2$, where $\beta_1<\alpha<\beta_2$, are
zeroes of~$\varphi'$ such that $\varphi(z)\neq0$ and
$\varphi'(z)\neq0$ for all
$z\in(\beta_1,\alpha)\cup(\alpha,\beta_2)$, then the
inequalities~\eqref{main.work.formula.new} hold, since we have
Case~I of Lemma~\ref{lemma.7} in both intervals $(\beta_1,\alpha)$
and $(\alpha,\beta_2)$ in this case.
\end{note}

\setcounter{equation}{0}

\section{Bounds on the number of real critical points of the logarithmic derivatives of functions in
$\mathcal{L-P^*}$}\label{section:infinite.intervals}

In this section, we study bounds on the number of real zeroes of
the function $Q[\varphi]$ associated with a function
$\varphi\in\mathcal{L-P^*}=U_0^*$ in detail.

In Section~\ref{subsection:half-infinite.intervals}, we
investigate bounds on the number of zeroes of the function $Q$ in
half-infinite intervals free of zeroes of the function $\varphi$.
Such intervals may appear if $\varphi$ has the smallest and/or the
largest zero.

Section~\ref{subsection:functions.at most.one.real.zero} is
devoted to the study of the number of real zeroes of $Q$
associated with functions in $\mathcal{L-P^*}$ whose derivatives
have no zeroes on the real axis. We establish some auxiliary
statements and then prove Theorems~\ref{corol.real.axis.4}
and~\ref{corol.real.exis.6}, which we use in
Section~\ref{section:general.theorems}.

\subsection{Half-infinite intervals free of zeroes of
$\varphi$}\label{subsection:half-infinite.intervals}

In Section~\ref{subsection.property.A}, we already established
that the function $Q[\varphi]$ has an even number of zeroes in the
interval $(\alpha_L,+\infty)$ where $\alpha_L$ is the largest zero
(if any) of the function $\varphi$ in $\mathcal{L-P^*}$. But if
additional information on the number of real zeroes of $\varphi'$
is available, then Lemma~\ref{lemma.1} and Lemma~\ref{lemma.2} can
be used to derive the following sharper result.
\begin{lemma}\label{corol.1}
Let $\varphi\in\mathcal{L-P^*}$ and suppose that $\varphi$ has the
largest zero $\alpha_L$ and $\varphi'$ has exactly $r\geqslant1$
extra zeroes, counting multiplicities, in the
interval~$(\alpha_L,+\infty)$. If $\beta_L$ is the largest zero of
$\varphi'$ in~$(\alpha_L,+\infty)$, then $Q$ has an odd
\emph{(}even\emph{)} number of real zeroes in $(\beta_L,+\infty)$
whenever~$r$~is an~even~\emph{(}odd\emph{)} number.
\end{lemma}
\begin{proof}
Let $\varphi'$ have $l\leqslant r$ \textit{distinct} zeroes, say
$\beta_1<\beta_2<\ldots<\beta_l=\beta_L$, in
the~interval~$(\alpha_L,+\infty)$. According to
Lemma~\ref{lemma.2}, $Q$ has an even number of real zeroes
in~$(\alpha_L,+\infty)$, counting multiplicities. But from
Lemma~\ref{lemma.1} it follows that $Q$ has an even number of real
zeroes in~$(\alpha_L,\beta_1)$, counting multiplicities, and an
odd number of real zeroes, say $2M_i+1$, in each of the intervals
$(\beta_i,\beta_{i+1})$ ($i=1,2,\ldots,l-1$). Hence, $Q$ has
exactly $\sum_{i=1}^{l-1}{(2M_i+1)}$ real zeroes, counting
multiplicities, in~$\bigcup_{i=1}^{l-1}{(\beta_i,\beta_{i+1})}$.

Moreover, from~\eqref{main.function.2} it follows that $\beta$ is
a zero of $Q$ of multiplicity~$M-1$ whenever~$\beta$ is a zero of
$\varphi'$ of multiplicity~$M$ and $\varphi(\beta)\neq0$.
Consequently, in our case, $Q$ has exactly $r-l$ real zeroes,
counting multiplicities, at the points $\beta_i$ that are multiple
zeroes of $\varphi'$. Thus, $Q$ has
$r-l+\sum_{i=1}^{l-1}{(2M_i+1)}=r-1+\sum_{i=1}^{l-1}{2M_i}$ real
zeroes, counting multiplicities, in the
interval~$[\beta_1,\beta_L]$. Therefore, if $r$ is an even (odd)
number, then $Q$ has an~odd~(even) number of real zeroes in
$(\alpha_L,\beta_L]$. Recall that $Q$ has an even number of zeroes
in~$(\alpha_L,\beta_1)$ by Lemma~\ref{lemma.1}. Consequently, $Q$
has an odd~(even) number of real zeroes in~$(\beta_L,+\infty)$,
since $Q$ has an even number of real zeroes in
$(\alpha_L,+\infty)$, according to~Lemma~\ref{lemma.2}.
\end{proof}

\begin{note}\label{remark.2.2} Lemma~\ref{corol.1} is valid with respective
modification in the case when $\varphi$ has the smallest
zero~$a_S$.
\end{note}

Lemmata~\ref{lemma.2} and~\ref{remark.2.2} together with results
of Section~\ref{section:finite.intervals} imply the following
theorem, which concerns half-infinite intervals adjacent to the
largest zero of $\varphi$.

\begin{theorem}\label{corol.4}
Let $\varphi\in\mathcal{L-P^*}$ and let $\alpha_L$ be the largest
zero~of~$\varphi$. If~$\varphi'$ has exactly $r\geqslant1$ zeroes
in the interval~$(\alpha_L,+\infty)$, counting multiplicities, and
$\beta_{S}$ is the minimal one, then
\begin{equation}\label{main.work.formula.6}
2\left[\dfrac{r}2\right]\leqslant
Z_{[\beta_{S},+\infty)}(Q)\leqslant
2\left[\dfrac{r}2\right]+Z_{[\beta_{S},+\infty)}(Q_{1}),
\end{equation}
where $\left[\dfrac r2\right]$ is the largest integer not
exceeding $\dfrac r2$.
\end{theorem}
\begin{proof}
Let
$\beta_S=\beta_1\leqslant\beta_2\leqslant\ldots\leqslant\beta_r$
be the zeroes of~$\varphi'$ in the interval~$(\alpha_L,+\infty)$.
Corollary~\ref{biggest.technical.theorem} gives
\begin{equation}\label{corol.4.intermediate.1}
r-1\leqslant Z_{[\beta_1,\beta_{r}]}(Q)\leqslant
r-1+Z_{[\beta_1,\beta_{r}]}(Q_{1}).
\end{equation}

We consider the following two cases.

\vspace{2mm}

\noindent Case I. The number $r$ is odd. There can be only the
following two situations:
\begin{itemize}
\item[I.1.] The function $\varphi''$ has an odd number of
zeroes, counting multiplicities, in the
interval~$(\beta_r,+\infty)$, say
$\gamma_1\leqslant\gamma_1\leqslant\gamma_2\leqslant\ldots\leqslant\gamma_{2M+1}$,
$M\geqslant0$.

From the analyticity of $\varphi$ and $\varphi'$ it follows that
\begin{equation}\label{corol.4.proof.1}
\varphi(\alpha_L+\varepsilon)\varphi'(\alpha_L+\varepsilon)>0
\end{equation}
for all sufficiently small $\varepsilon>0$, and
\begin{equation}\label{corol.4.proof.2}
\varphi'(\beta_r+\delta)\varphi''(\beta_r+\delta)>0
\end{equation}
for all sufficiently small $\delta>0$. Since $r$ is an odd number,
$\varphi'$ has different signs in the intervals
$(\alpha_L,\beta_1)$ and $(\beta_r,+\infty)$. Therefore,
from~\eqref{corol.4.proof.1} we obtain
\begin{equation}\label{corol.4.proof.3}
\varphi(\beta_r+\delta)\varphi'(\beta_r+\delta)<0
\end{equation}
for sufficiently small $\delta>0$, since $\varphi(z)\neq0$ for
$z\in(\alpha_L,+\infty)$. Hence, from~\eqref{corol.4.proof.2}
and~\eqref{corol.4.proof.3} it follows that
\begin{equation}\label{corol.4.proof.4}
\varphi(\beta_r+\varepsilon)\varphi''(\beta_r+\varepsilon)<0
\end{equation}
for sufficiently small~$\varepsilon>0$. This means that $Q(z)<0$
for $z\in(\beta_r,\gamma_1)$ by~\eqref{main.function.2}, so
\begin{equation}\label{corol.4.proof.4.5}
0=Z_{(\beta_r,\gamma_{1})}(Q)\leqslant
Z_{(\beta_r,\gamma_{1})}(Q_{1}).
\end{equation}
Here the upper bound cannot be improved by parity considerations, since $Q_1$ has an even
number of zeroes in the interval $(\beta_r,\gamma_{1})$ according
to Lemma~\ref{lemma.1} applied to $Q_1$. By Case~I of
Lemma~\ref{lemma.6.5}, we have
\begin{equation}\label{corol.4.intermediate.2}
0\leqslant Z_{[\gamma_{1},\gamma_{2M+1}]}(Q)\leqslant
Z_{[\gamma_{1},\gamma_{2M+1}]}(Q_{1}).
\end{equation}
But $\varphi''$ has different signs in the intervals
$(\beta_1,\gamma_1)$ and $(\gamma_{2M+1},+\infty)$, therefore,
by~\eqref{corol.4.proof.4},
\begin{equation*}
\varphi(\gamma_{2M+1}+\delta)\varphi''(\gamma_{2M+1}+\delta)>0
\end{equation*}
for sufficiently small $\delta>0$, so the
number~$Z_{(\gamma_{2M+1},+\infty)}(Q)$ can be positive according
to~\eqref{main.function.2}. Now we note that $Q$ has an even
number of zeroes in $(\beta_r,+\infty)$ and in
$[\gamma_{1},\gamma_{2M+1}]$ by Lemmata~\ref{corol.1}
and~\ref{lemma.6.5}. Since $Q$ has no zeroes in
$(\beta_r,\gamma_1)$, it has an even number of zeroes in the
interval~$(\gamma_{2M+1},+\infty)$. Lemma~\ref{lemma.5} and
Remark~\ref{remark.2.4} imply the validity of the
inequality~\eqref{main.work.formula.1} for
$(\gamma_{2M+1},+\infty)$:
\begin{equation}\label{corol.4.intermediate.3}
Z_{(\gamma_{2M+1},+\infty)}(Q)\leqslant
1+Z_{(\gamma_{2M+1},+\infty)}(Q_{1}).
\end{equation}
Applying Lemma~\ref{corol.1} to $Q_1$ on the interval
$(\gamma_{2M+1},+\infty)$, we obtain that $Q_{1}$ has an even
number of zeroes in~$(\gamma_{2M+1},+\infty)$ too. Consequently,
the inequality~\eqref{corol.4.intermediate.3} can be improved to
the following one
\begin{equation}\label{corol.4.intermediate.4}
0\leqslant Z_{(\gamma_{2M+1},+\infty)}(Q)\leqslant
Z_{(\gamma_{2M+1},+\infty)}(Q_{1}).
\end{equation}
Here the trivial lower bound cannot be improved by parity considerations, since
$Q$ has an even number of zeroes in~$(\gamma_{2M+1},+\infty)$. Now the
inequalities~\eqref{corol.4.intermediate.1},~\eqref{corol.4.proof.4.5},~\eqref{corol.4.intermediate.2}
and~\eqref{corol.4.intermediate.4} imply
\begin{equation}\label{corol.4.intermediate.5}
r-1\leqslant Z_{[\beta_S,+\infty)}(Q)\leqslant
r-1+Z_{[\beta_S,+\infty)}(Q_{1}),
\end{equation}
which is equivalent to~\eqref{main.work.formula.6}, since $r$ is
odd, so $2\left[\dfrac{r}2\right]=r-1$.

\item[I.2.] The function $\varphi''$ has an even number of
zeroes,  counting multiplicities, in the
interval~$(\beta_r,+\infty)$, say
$\gamma_1\leqslant\gamma_1\leqslant\gamma_2\leqslant\ldots\leqslant\gamma_{2M}$,
$M\geqslant0$.

\vspace{2mm}

Let $M=0$, that is, $\varphi''(z)\neq0$ for
$z\in(\beta_r,+\infty)$. Since $r$ is an odd number, the
inequality~\eqref{corol.4.proof.4} holds as we proved above.
Therefore,
\begin{equation*}\label{corol.4.proof.5}
0=Z_{(\beta_r,+\infty)}(Q)\leqslant Z_{(\beta_r,+\infty)}(Q_{1}).
\end{equation*}
Combined with~\eqref{corol.4.intermediate.1}, this inequality
implies~\eqref{corol.4.intermediate.5}, which is equivalent
to~\eqref{main.work.formula.6} as we explained above.

\vspace{2mm}

Let $M>0$, then~\eqref{corol.4.proof.4} and Case~III of
Lemma~\ref{lemma.6.5} imply
\begin{equation}\label{corol.4.proof.6}
0\leqslant
Z_{[\gamma_{1},\gamma_{2M}]}(Q)\leqslant1+Z_{[\gamma_{1},\gamma_{2M}]}(Q_{1}).
\end{equation}
As above, the inequality~\eqref{corol.4.proof.4}
implies~\eqref{corol.4.proof.4.5}. Moreover, from the fact that
$\varphi''$ has an even number of zeroes in $(\beta_r,+\infty)$ it
follows that $\varphi''$ has equal signs in the intervals
$(\beta_r,\gamma_1)$ and $(\gamma_{2M},+\infty)$. Then
by~\eqref{corol.4.proof.4}, we have the following inequality
\begin{equation}\label{corol.4.proof.7}
\varphi(\gamma_{2M}+\varepsilon)\varphi''(\gamma_{2M}+\varepsilon)<0
\end{equation}
for sufficiently small $\varepsilon$, so
$Z_{(\gamma_{2M},+\infty)}(Q)=0$ by~\eqref{main.function.2}.
Applying Lemma~\ref{corol.1} to $Q_1$, we obtain that $Q_{1}$ has
an odd number of zeroes in~$(\gamma_{2M},+\infty)$ (at least one).
Therefore, we have
\begin{equation}\label{corol.4.proof.7.8}
0=Z_{(\gamma_{2M},+\infty)}(Q)\leqslant
-1+Z_{(\gamma_{2M},+\infty)}(Q_{1}).
\end{equation}
Now
from~\eqref{corol.4.intermediate.1},~\eqref{corol.4.proof.4.5},~\eqref{corol.4.proof.6}
and~\eqref{corol.4.proof.7.8} we again obtain the
inequalities~\eqref{corol.4.intermediate.5}, which are equivalent
to~\eqref{main.work.formula.6} as above.
\end{itemize}

\vspace{2mm}

\noindent Case II. The number $r>0$ is even. There also can be
only two situations:
\begin{itemize}
\item[II.1.] The function~$\varphi''$ has an even number
of zeroes, counting multiplicities, in the
interval~$(\beta_r,+\infty)$, say
$\gamma_1\leqslant\gamma_1\leqslant\gamma_2\leqslant\ldots\leqslant\gamma_{2M}$,
$M\geqslant0$.

\vspace{2mm}

Let $M>0$. According to Lemma~\ref{corol.1}, $Q$ has an odd number
of zeroes in~$(\beta_r,+\infty)$. But $Q(z)$ is negative for
sufficiently large real $z$ as was shown in the proof of
Theorem~\ref{Theorem.finitude.number.of.zeroes.of.Q}, therefore,
\begin{equation}\label{corol.4.proof.8}
Q(\beta_r+\varepsilon)>0
\end{equation}
for sufficiently small $\varepsilon>0$.
From~\eqref{intermediate.10.lemma.6} it follows that
$Q(\gamma_1)<0$, consequently, $Q$ has an odd number of zeroes (at
least one) in the interval~$(\beta_r,\gamma_1)$. This fact and
Lemma~\ref{lemma.5} imply the following inequalities
\begin{equation}\label{corol.4.intermediate.6}
1\leqslant
Z_{(\beta_r,\gamma_1)}(Q)\leqslant1+Z_{(\beta_r,\gamma_1)}(Q_{1}).
\end{equation}
By Lemma~\ref{lemma.1} applied to $Q_{1}$, the latter has an even
number of zeroes in the interval~$(\beta_r,\gamma_1)$. Hence, the
inequalities~\eqref{corol.4.intermediate.6} cannot be improved by
parity considerations.

Since~$\varphi''$ has an even number of zeroes, counting
multiplicities, in the~interval~$(\beta_r,+\infty)$, $\varphi''$
has equal signs in the intervals~$(\beta_r,\gamma_1)$
and~$(\gamma_{2M},+\infty)$. But $Q$ has at least one zero
in~$(\beta_r,\gamma_1)$ (see~\eqref{corol.4.intermediate.6}),
therefore from~\eqref{main.function.2} it follows that
$\varphi(z)\varphi''(z)$ is positive in~$(\beta_r,\gamma_1)$.
Consequently, $\varphi(z)\varphi''(z)$ is also positive
in~$(\gamma_{2M},+\infty)$. Thus, by~\eqref{main.function.2}, the
number~$Z_{(\gamma_{2M},+\infty)}(Q)$ can be positive. Moreover,
since $Q(z)$ is negative for $z$ sufficiently close
to~$\gamma_{2M}$ (see~\eqref{intermediate.10.lemma.6}) and for
sufficiently large real $z$ (see the proof of
Theorem~\ref{Theorem.finitude.number.of.zeroes.of.Q}), $Q$~has an
even number of zeroes in~$(\gamma_{2M},+\infty)$.
Lemma~\ref{lemma.5} and Remark~\ref{remark.2.4} imply
\begin{equation}\label{corol.4.intermediate.6.5}
0\leqslant Z_{(\gamma_{2M},+\infty)}(Q)\leqslant
1+Z_{(\gamma_{2M},+\infty)}(Q_{1}).
\end{equation}
Here the upper bound cannot be improved by parity considerations, since $Q_{1}$ has an odd
number of zeroes in the interval $(\gamma_{2M},+\infty)$ by
Lemma~\ref{corol.1} applied to $Q_{1}$.

Since $\varphi(z)\varphi''(z)>0$ in~$(\beta_r,\gamma_1)$
(see~\eqref{corol.4.intermediate.6} and~\eqref{main.function.2}),
by Case~II of Lemma~\ref{lemma.6.5}, we have
\begin{equation}\label{corol.4.intermediate.20}
0\leqslant
Z_{[\gamma_{1},\gamma_{2M}]}(Q)\leqslant-1+Z_{[\gamma_{1},\gamma_{2M}]}(Q_{1}).
\end{equation}
Now from~\eqref{corol.4.intermediate.1}
and~\eqref{corol.4.intermediate.6}--\eqref{corol.4.intermediate.20}
we obtain the following inequalities
\begin{equation}\label{corol.4.intermediate.22}
r\leqslant Z_{(\beta_s,+\infty)}(Q)\leqslant
r+Z_{(\beta_s,+\infty)}(Q_{1}),
\end{equation}
which are equivalent to~\eqref{main.work.formula.6}, since $r$ is
even, so $2\left[\dfrac{r}2\right]=r$.

\vspace{2mm}

If $M=0$, that is, $\varphi''(z)\neq0$ in the interval
$(\beta_r,+\infty)$, then we have the inequalities
\begin{equation}\label{corol.4.intermediate.607}
1\leqslant
Z_{(\beta_r,+\infty)}(Q)\leqslant1+Z_{(\beta_r,+\infty)}(Q_{1}),
\end{equation}
where the upper bound follows from Lemma~\ref{lemma.5} and
Remark~\ref{remark.2.4}, and the lower bound follows
from~\eqref{corol.4.proof.8} and from the fact that $Q(z)<0$ for
sufficiently large real $z$.

By~\eqref{corol.4.intermediate.1}
and~\eqref{corol.4.intermediate.607}, we again
obtain~\eqref{corol.4.intermediate.22}, which is equivalent
to~\eqref{main.work.formula.6} as we showed above.

\item[II.2.] The function~$\varphi''$ has an odd number
of zeroes, counting multiplicities, in the
interval~$(\beta_r,+\infty)$, say
$\gamma_1\leqslant\gamma_1\leqslant\gamma_2\leqslant\ldots\leqslant\gamma_{2M+1}$,
$M\geqslant0$.

As in Case II.1, Lemma~\ref{corol.1} implies the
inequality~\eqref{corol.4.proof.8}, from which the
inequalities~\eqref{corol.4.intermediate.6} follow by
Lemma~\ref{lemma.5}. Furthermore, from Case~I of
Lemma~\ref{lemma.6.5} we obtain the
inequalities~\eqref{corol.4.intermediate.2}. Since $Q$ has at
least one zero in the interval $(\beta_r,\gamma_1)$
(see~\eqref{corol.4.intermediate.6}), $\varphi(z)\varphi''(z)>0$
in~$(\beta_r,\gamma_1)$ by~\eqref{main.function.2}. But
$\varphi''$ has an odd number of zeroes in $(\beta_r,+\infty)$ by
assumption, and $\varphi(z)\neq0$ in this interval, therefore,
$\varphi(z)\varphi''(z)<0$ in $(\gamma_{2M+1},+\infty)$ and
$Q$~has no zeroes in the interval $(\gamma_{2M+1},+\infty)$. This
fact implies the following inequality
\begin{equation*}\label{corol.4.proof.11}
0=Z_{(\gamma_{2M+1},+\infty)}(Q)\leqslant
Z_{(\gamma_{2M+1},+\infty)}(Q_{1}),
\end{equation*}
which, together with~\eqref{corol.4.intermediate.2},
and~\eqref{corol.4.intermediate.6}
implies~\eqref{corol.4.intermediate.607}. Now
by~\eqref{corol.4.intermediate.1}
and~\eqref{corol.4.intermediate.607}, we
obtain~\eqref{corol.4.intermediate.22}, which is equivalent
to~\eqref{main.work.formula.6} as we showed in Case~II.1.
\end{itemize}
\end{proof}
\begin{note}\label{remark.2.8}
Theorem~\ref{corol.4} remains valid with respective modification
in the case when $\varphi$ has the~smallest~zero~$a_S$.
\end{note}
\begin{note}\label{remark.2.9}
The number~$r$ in Theorem~\ref{corol.4} is the number of extra
zeroes of $\varphi'$ in~$(\alpha_L,+\infty)$.
\end{note}

\subsection{Functions in $\mathcal{L-P^*}$ whose first derivatives have no
real zeroes}\label{subsection:functions.at most.one.real.zero}

So far, we have considered finite and half-infinite intervals. Our
statements in this section address the entire~real~axis.

\begin{lemma}\label{corol.real.axis.1}
Let $\varphi\in\mathcal{L-P^*}$. If $\varphi'(z)\neq0$,
$\varphi''(z)\neq0$, $Q_{1}(z)\neq0$ for $z\in\mathbb{R}$, then
$Z_\mathbb{R}(Q)=0$, i.e. $Q$ has no real zeroes.
\end{lemma}
\begin{proof}
The function $\varphi$ may have real zeroes. But by Rolle's
theorem, $\varphi$ has at most one real zero, counting
multiplicity, since $\varphi'(z)\neq0$ for $z\in\mathbb{R}$ by
assumption.

If $\varphi(z)\neq0$ for $z\in\mathbb{R}$, then by
Lemma~\ref{lemma.3} applied on the real axis, $Q$ has at most one
real zero, counting multiplicity. But by Corollary~\ref{corol.3},
the number of real zeroes of $Q$ is even. Consequently,
$Z_\mathbb{R}(Q)=0$.

If $\varphi$ has one real zero, counting multiplicity, say
$\alpha$, then $\alpha$ is the unique real pole of $Q$. Moreover,
since $\varphi''(z)\neq0$ for $z\in\mathbb{R}$ by assumption, the
function $\varphi\varphi''$ changes its sign at~$\alpha$ and,
therefore, $\varphi(z)\varphi''(z)<0$ in one the intervals
$(-\infty,\alpha)$ and $(\alpha,+\infty)$. Consequently,
by~\eqref{main.function.2}, $Q(z)\neq0$ in one of the intervals
$(-\infty,\alpha)$ and $(\alpha,+\infty)$. Without loss of
generality, we may suppose that $Z_{(-\infty,\alpha]}(Q)=0$. Then
according to Lemma~\ref{lemma.3} and
Remark~\ref{remark.half.inf.1}, we obtain that $Q$ has at most one
zero, counting multiplicity, in the interval $(\alpha,+\infty)$.
But the number of real zeroes of $Q$ is even by
Corollary~\ref{corol.3}, and $Q$ has no zeroes in the
interval~$(-\infty,\alpha]$. Therefore, $Q$ cannot have zeroes in
the interval~$(\alpha,+\infty)$, so $Z_\mathbb{R}(Q)=0$, as
required.
\end{proof}

\begin{lemma}\label{corol.real.axis.2}
Let $\varphi\in\mathcal{L-P^*}$ and let $\varphi'(z)\neq0$,
$\varphi''(z)\neq0$ for $z\in\mathbb{R}$. If $Q_1$ has only one
real zero, then~$Q$ has no real zeroes, i.e. $Z_\mathbb{R}(Q)=0$.
\end{lemma}
\begin{proof}
As in Lemma~\ref{corol.real.axis.1}, we note that $\varphi$ has at
most one real zero, counting multiplicity, by Rolle's theorem.

\vspace{2mm}

\noindent Case I. Let $\varphi(z)\neq0$ for $z\in\mathbb{R}$ and
let $\xi$ be a unique real zero of $Q_1$. By
Corollary~\ref{corol.3} (see Remark~\ref{remark.2.3}), $\xi$ is a
zero of $Q_1$ of even multiplicity $2M$. In this case, the number
$\xi$ cannot be zero of $Q$. In fact, if $Q(\xi)=0$, then, by
Lemma~\ref{lemma.4} applied on the real axis, $\xi$ is a unique
real zero of $Q$ of multiplicity $2M+1$, that is,
$Z_{\{\xi\}}(Q)=Z_\mathbb{R}(Q)=2M+1$. This contradicts
Corollary~\ref{corol.3}, so~$Q(\xi)\neq0$.

Since $\xi$ is a unique real zero of $Q_1$ of even multiplicity,
$Q_1$ has equal signs in the intervals~$(-\infty,\xi)$
and~$(\xi,+\infty)$. By Lemma~\ref{lemma.3} applied to these
intervals, $Q$ can have at most one zero, counting multiplicity,
in each of the intervals~$(-\infty,\xi)$ and~$(\xi,+\infty)$. But
if $Q$ has a zero in the interval~$(-\infty,\xi)$, then
$Q(z)\neq0$ for~$z\in(\xi,+\infty)$. In fact, if
$\zeta\in(-\infty,\xi)$ is a zero of $Q$,
then~\eqref{lemma.3.condition.2} must hold on $(-\infty,\zeta)$,
and $\zeta$ is a~simple zero of $Q$ by Lemma~\ref{lemma.3}.
Moreover, $Q(z)\neq0$ for $z\in(\zeta,\xi]$, since (see the proof
of Lemma~\ref{lemma.3}), for all sufficiently small $\delta>0$, we
have
\begin{equation}\label{corol.real.axis.2.proof.1}
\varphi'(\zeta+\delta)\varphi''(\zeta+\delta)Q(\zeta+\delta)Q_{1}(\zeta+\delta)>0.
\end{equation}
But the functions $\varphi'$, $\varphi''$, $Q$ and $Q_1$ do not
change their signs at $\xi$, therefore, we have the
inequality~\eqref{corol.real.axis.2.proof.1} in a small
right-sided neighbourhood of $\xi$. Consequently, $Q(z)\neq0$
in~$(\xi,+\infty)$ by Lemma~\ref{lemma.3}.

Thus, if $Q$ has real zeroes, then it has at most one zero,
counting multiplicity, in one of the intervals~$(-\infty,\xi)$
and~$(\xi,+\infty)$, that is, $Z_\mathbb{R}(Q)\leqslant1$. Now
Corollary~\ref{corol.3} implies $Z_\mathbb{R}(Q)=0$.

\vspace{2mm}

\noindent Case II. If $\varphi$ has one real zero, counting
multiplicity, say $\alpha$, then $\alpha$ is the unique real pole of
$Q$. In this case, as in the proof of
Lemma~\ref{corol.real.axis.1}, one can show that $Q(z)\neq0$ in
one of the intervals $(-\infty,\alpha)$ and $(\alpha,+\infty)$.
Without loss of generality, we may assume that $Q(z)\neq0$ for
$z\in(-\infty,\alpha]$, that is, $Z_{(-\infty,\alpha]}(Q)=0$. Then
by Corollary~\ref{corol.3}, the number $Z_{(\alpha,+\infty)}(Q)$
is even.

Let $\xi$ be the unique real zero of $Q_1$ and
$\xi\in(-\infty,\alpha]$. Recall that
$Z_{\{\xi\}}(Q_1)=Z_{\mathbb{R}}(Q_1)=2M$, $M\geqslant1$,
according to Corollary~\ref{corol.3}. By Lemma~\ref{lemma.3} (see
Remark~\ref{remark.half.inf.1}), $Q$~has at most one zero,
counting multiplicity, in~$(\alpha,+\infty)$. Since
$Z_{(\alpha,+\infty)}(Q)$ is an even number, we have
$Z_\mathbb{R}(Q)=0$.

If $\xi\in(\alpha,+\infty)$, then by the same argument as in
Case~I, one can show that $Q(\xi)\neq0$. Furthermore, by
Lemma~\ref{lemma.3} applied to the intervals~$(\alpha,\xi)$
and~$(\xi,+\infty)$, $Q$ can have at most one zero, counting
multiplicity, in each of these intervals. In the same way as in
Case~I, one can show that $Q(z)\neq0$ for $z\in(\xi,+\infty)$ if
$Q$ has a zero on the interval $(\alpha,\xi)$. So
$Z_{(\alpha,+\infty)}(Q)\leqslant1$, therefore,
$Z_\mathbb{R}(Q)\leqslant1$, since $Z_{(-\infty,\alpha]}(Q)=0$ by
assumption. Now Corollary~\ref{corol.3} again implies
$Z_\mathbb{R}(Q)=0$.
\end{proof}

We are now in a position to establish a relation between the
number of real zeroes of $Q$ and $Q_1$ on the real line analogous
to~\eqref{main.work.formula.1}.

\begin{theorem}\label{corol.real.axis.3}
Let $\varphi\in\mathcal{L-P^*}$. If $\varphi'(z)\neq0$,
$\varphi''(z)\neq0$ for $z\in\mathbb{R}$, then
\begin{equation}\label{corol.real.axis.4.main.formula}
0\leqslant Z_\mathbb{R}(Q)\leqslant Z_\mathbb{R}(Q_{1}).
\end{equation}
\end{theorem}
\begin{proof}
As in Lemma~\ref{corol.real.axis.1}, we note that $\varphi$ has at
most one real zero, counting multiplicity, by Rolle's theorem.

\vspace{2mm}

\noindent Case I. Let $\varphi(z)\neq0$ for $z\in\mathbb{R}$. Then
we can apply Lemma~\ref{lemma.5} (see also
Lemmata~\ref{corol.real.axis.1} and~\ref{corol.real.axis.2}) on
the real axis to obtain
\begin{equation}\label{corol.real.axis.3.proof.1}
0\leqslant Z_\mathbb{R}(Q)\leqslant1+Z_\mathbb{R}(Q_{1}),
\end{equation}
that is equivalent to~\eqref{corol.real.axis.4.main.formula},
since the numbers $Z_\mathbb{R}(Q)$ and $Z_\mathbb{R}(Q_1)$ are
both even by Corollary~\ref{corol.3} and Remark~\ref{remark.2.3}.

\vspace{2mm}

\noindent Case II. If $\varphi$ has one real zero, counting
multiplicity, say $\alpha$, then $\alpha$ is a unique pole of $Q$.
In this case, as in the proof of Lemma~\ref{corol.real.axis.1},
one can show that $Q(z)\neq0$ in one of the intervals
$(-\infty,\alpha)$ and $(\alpha,+\infty)$. Without loss of
generality, we assume that $Q(z)\neq0$ for $z\in(-\infty,\alpha]$.
Then by Corollary~\ref{corol.3}, the number
$Z_{(\alpha,+\infty)}(Q)$ is even and the following inequality
holds
\begin{equation}\label{corol.real.axis.3.proof.2}
0=Z_{(-\infty,\alpha]}(Q)\leqslant Z_{(-\infty,\alpha]}(Q_{1}).
\end{equation}
By Lemma~\ref{lemma.5} (see
Remarks~\ref{remark.half.inf.1},~\ref{remark.2.4.new}
and~\ref{remark.2.4}), we obtain
\begin{equation}\label{corol.real.axis.3.proof.3}
0\leqslant
Z_{(\alpha,+\infty)}(Q)\leqslant1+Z_{(\alpha,+\infty)}(Q_{1}).
\end{equation}
The
inequalities~\eqref{corol.real.axis.3.proof.2}--\eqref{corol.real.axis.3.proof.3}
imply~\eqref{corol.real.axis.3.proof.1}, which is equivalent
to~\eqref{corol.real.axis.4.main.formula} as we showed above.
\end{proof}

We now address the case when $\varphi$ and $\varphi'$ have no real
zeroes.
\begin{theorem}\label{corol.real.axis.4}
Let $\varphi\in\mathcal{L-P^*}$ and let $\varphi(z)\neq0$,
$\varphi'(z)\neq0$ for $z\in\mathbb{R}$. Then the
inequalities~\eqref{corol.real.axis.4.main.formula} hold for
$\varphi$.
\end{theorem}
\begin{proof} At first, we note that $\varphi''$ cannot
have an odd number of real zeroes, counting multiplicities, if
$\varphi(z)\neq0$ for $z\in\mathbb{R}$. In fact, since
$\varphi\in\mathcal{L-P^*}$ has no real zeroes by assumption,
$\varphi(z)=e^{q(z)}p(z)$, where $q$ is a real polynomial of
degree at most two and $p$ is a real polynomial with no real
zeroes. Therefore, $\deg p$ is even. Since $\varphi''$ has exactly
$\deg p+2\deg q-2$ zeroes, $\varphi''$ has an even number of real
zeroes, counting multiplicities, or has no real zeroes.

\vspace{2mm}

If $\varphi''(z)\neq0$ for $z\in\mathbb{R}$, then the
inequalities~\eqref{corol.real.axis.4.main.formula} follow from
Theorem~\ref{corol.real.axis.3}.

\vspace{2mm}

\noindent Let $\varphi''$ have an even number of real zeroes,
counting multiplicities, say
$\gamma_1\leqslant\gamma_2\leqslant\ldots\leqslant\gamma_{2r}$,
$r\geqslant1$. If $\gamma_1=\gamma_{2r}$, then $\gamma_1$ is a
unique real zero of $\varphi''$ and its multiplicity is even.
Consequently, $\gamma_1$ is a zero of $Q_1$, according
to~\eqref{all.main.functions}. Therefore, the
inequality~\eqref{intermediate.10.lemma.6} implies
\begin{equation}\label{corol.real.axis.4.proof.5}
0=Z_{\{\gamma_1\}}(Q)\leqslant-1+Z_{\{\gamma_1\}}(Q_{1}).
\end{equation}
Since $Q(z)$ is negative when $z=\gamma_1$
(see~\eqref{intermediate.10.lemma.6}) and when $z\to\pm\infty$,
$Q$ has an even number~of~zeroes in each of the
intervals~$(-\infty,\gamma_1)$ and~$(\gamma_{1},+\infty)$. Then by
Lemma~\ref{lemma.5} and Remark~\ref{remark.2.4}, the following
inequalities hold
\begin{equation}\label{corol.real.axis.4.proof.6}
0\leqslant Z_{(-\infty,\gamma_1)}(Q)\leqslant1+
Z_{(-\infty,\gamma_1)}(Q_{1}),
\end{equation}
\begin{equation*}\label{corol.real.axis.4.proof.7}
0\leqslant Z_{(\gamma_1,+\infty)}(Q)\leqslant1+
Z_{(\gamma_1,+\infty)}(Q_{1}).
\end{equation*}
Together with~\eqref{corol.real.axis.4.proof.5}, these
inequalities imply~\eqref{corol.real.axis.3.proof.1}, which is
equivalent to~\eqref{corol.real.axis.4.main.formula}, since
the~numbers $Z_\mathbb{R}(Q)$ and $Z_\mathbb{R}(Q_1)$ are both
even by Corollary~\ref{corol.3} and Remark~\ref{remark.2.3}.

\vspace{2mm}

Let $\gamma_1<\gamma_{2r}$ and let $\varphi(z)\varphi''(z)>0$ for
$z\in(-\infty,\gamma_1)$. Then $\varphi(z)\varphi''(z)>0$ for
$z\in(\gamma_{2r},+\infty)$, since $\varphi''$ has equal signs in
the intervals~$(-\infty,\gamma_1)$ and~$(\gamma_{2r},+\infty)$ and
$\varphi(z)\neq0$ for $z\in\mathbb{R}$ by assumption. By
Lemma~\ref{lemma.5} and Remark~\ref{remark.2.4}, the
inequalities~\eqref{corol.real.axis.4.proof.6} hold and
\begin{equation}\label{corol.real.axis.4.proof.8}
0\leqslant Z_{(\gamma_{2r},+\infty)}(Q)\leqslant1+
Z_{(\gamma_{2r},+\infty)}(Q_{1}).
\end{equation}
Case~II of Lemma~\ref{lemma.6.5} implies
\begin{equation}\label{corol.real.axis.4.proof.9}
0\leqslant
Z_{[\gamma_1,\gamma_{2r}]}(Q)\leqslant-1+Z_{[\gamma_1,\gamma_{2r}]}(Q_{1}).
\end{equation}
Now
from~\eqref{corol.real.axis.4.proof.6}--\eqref{corol.real.axis.4.proof.9}
we obtain the inequalities~\eqref{corol.real.axis.3.proof.1},
which are equivalent to~\eqref{corol.real.axis.4.main.formula} as
we mentioned above.

\vspace{2mm}

If $\gamma_1<\gamma_{2r}$ and $\varphi(z)\varphi''(z)<0$ for
$z\in(-\infty,\gamma_1)$, then $\varphi(z)\varphi''(z)<0$ for
$z\in(\gamma_{2r},+\infty)$. Therefore,
by~\eqref{main.function.2}, the following inequalities hold
\begin{equation}\label{corol.real.axis.4.proof.10}
0=Z_{(-\infty,\gamma_1)}(Q)\leqslant
Z_{(-\infty,\gamma_1)}(Q_{1}),
\end{equation}
\begin{equation}\label{corol.real.axis.4.proof.11}
0=Z_{(\gamma_{2r},+\infty)}(Q)\leqslant
Z_{(\gamma_{2r},+\infty)}(Q_{1}).
\end{equation}
From Case~III of Lemma~\ref{lemma.6.5} it follows that the
inequalities~\eqref{lemma6.5.main.1} hold in the
interval~$[\gamma_1,\gamma_{2r}]$ and together
with~\eqref{corol.real.axis.4.proof.10}--\eqref{corol.real.axis.4.proof.11},
they imply~\eqref{corol.real.axis.3.proof.1}, which is equivalent
to~\eqref{corol.real.axis.4.main.formula}.
\end{proof}

At last, we examine the case when $\varphi$ has a unique real zero
$\alpha$ and possesses \textit{property~A} at $\alpha$.
\begin{theorem}\label{corol.real.exis.6}
Let $\varphi\in\mathcal{L-P^*}$ have a unique real zero $\alpha$
such that $\varphi'(z)\neq0$ for
$z\in\mathbb{R}\backslash\{\alpha\}$. If~$\varphi$ possesses
property A at $\alpha$, then the
inequalities~\eqref{corol.real.axis.4.main.formula} hold.
\end{theorem}
\begin{proof}
Since $\varphi$ possesses \textit{property~A}, $Q$ associated with
$\varphi$ has no zeroes in at least one of the intervals
$(-\infty,\alpha)$ and $(\alpha,+\infty)$. Without loss of
generality, we may assume that $Q(z)\neq0$ for
$z\in(-\infty,\alpha)$. Then we have
\begin{equation}\label{corol.real.exis.6.proof.1}
0=Z_{(-\infty,\alpha]}(Q)\leqslant Z_{(-\infty,\alpha]}(Q_{1}),
\end{equation}
since $Q(\alpha)\neq0$. From Lemma~\ref{lemma.7} and
Remark~\ref{remark.2.5} it follows that
\begin{equation}\label{corol.real.exis.6.proof.2}
0\leqslant
Z_{(\alpha,+\infty)}(Q)\leqslant1+Z_{(\alpha,+\infty)}(Q_{1}).
\end{equation}
The
inequalities~\eqref{corol.real.exis.6.proof.1}--\eqref{corol.real.exis.6.proof.2}
imply~\eqref{corol.real.axis.3.proof.1}, which is equivalent
to~\eqref{corol.real.axis.4.main.formula}, since $Z_\mathbb{R}(Q)$
and $Z_\mathbb{R}(Q_1)$ are even by Corollary~\ref{corol.3} and
Remark~\ref{remark.2.3}.
\end{proof}


\subsection{The number of real zeroes of the functions $Q$ and
$Q_1$}\label{section:general.theorems}

In this section, we establish bounds on the number of real zeroes
of $Q$ associated with a function $\varphi$ in~$\mathcal{L-P}^*$
in terms of the number of real zeroes of the function $Q_1$. We
consider separately three types of functions in $\mathcal{L-P}^*$
with finitely many real zeroes and also the functions with
infinitely many real zeroes. The main idea to establish the
mentioned bounds is to consider local estimates. That is, the real
zeroes $\alpha_j$ of a given function $\varphi\in\mathcal{L-P}^*$
and the real zeroes $\beta_j$ of its derivative $\varphi'$ split
the real axis into a few types of finite and half-infinite
intervals, so we can use results of
Sections~\ref{subsection:finite.intervals.2}
and~\ref{subsection:half-infinite.intervals} on those intervals.
In the case when $\varphi$ has no real zeroes, we use results of
Section~\ref{subsection:functions.at most.one.real.zero}.

So, if the function $\varphi$ has infinitely many positive and
negative zeroes, then its real zeroes and the real zeroes of its
derivative split the real axis into only the two type intervals
appeared in Corollary~\ref{biggest.technical.theorem} and
Theorem~\ref{corol.new}. If $\varphi$ has infinitely many negative
zeroes but finitely many positive ones, then we additionally
estimate the number of zeroes of $Q[\varphi]$ on half-infinite
intervals using Theorem~\ref{corol.4}.

Likewise, if the function $\varphi$ has finitely many real zeroes,
then its real zeroes (if any) and the real zeroes of its
derivative split the real axis into the three types of intervals
appeared in Corollary~\ref{biggest.technical.theorem} and
Theorems~\ref{corol.new} and~\ref{corol.4}.

\begin{example}
\begin{figure}[ht]
\centering \includegraphics{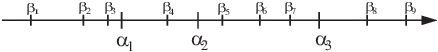} \caption{} \label{pic.3}
\end{figure}
Let the function $\varphi$ have the form $\varphi(z)=e^{\lambda
z}p(z)$, where $\lambda>0$ and $p$ is a real polynomial. Suppose
that the real zeroes of $\varphi$ and zeroes of its derivative
$\varphi'$ are located as on Figure~\ref{pic.3}. Then we can apply
Theorem~\ref{corol.new} to the intervals $(\beta_3,\beta_4)$,
$(\beta_4,\beta_5)$ and $(\beta_7,\beta_8)$ if $\varphi$ possesses
\textit{property~A}. Applying Theorem~\ref{corol.4} to the
intervals $(-\infty,\beta_1)$ and $(\beta_8,+\infty)$ and applying
Corollary~\ref{biggest.technical.theorem} to the
interval~$[\beta_5,\beta_7]$, we obtain

\begin{equation}\label{example.qqqq}
6\leqslant
Z_\mathbb{R}\left(Q\right)\leqslant6+Z_\mathbb{R}\left(Q_{1}\right).
\end{equation}
Here we take into account that
$Z_{\{\beta_4\}}\left(Q\right)=Z_{\{\beta_4\}}\left(Q\right)=0$.
Since  $E(\varphi')=7$ in this case, we have
$6=2\left[\frac{E(\varphi')}2\right]$. Besides, by
Theorem~\ref{Theorem.number.of.extra.zeroes.type.I}, we have
\begin{equation}\label{example.qqqq.2}
E(\varphi')=Z_{\mathbb{C}}(\varphi)-Z_{\mathbb{C}}(\varphi'),
\end{equation}
so from~\eqref{example.qqqq}--\eqref{example.qqqq.2} we obtain
\begin{equation}\label{example.qqqq.3}
Z_{\mathbb{C}}(\varphi)-Z_{\mathbb{C}}(\varphi')\leqslant
Z_\mathbb{R}\left(Q\right)\leqslant
Z_{\mathbb{C}}(\varphi)-Z_{\mathbb{C}}(\varphi')+Z_\mathbb{R}\left(Q_{1}\right).
\end{equation}
\end{example}

Thus, the main idea to estimate the number $Z_R(Q)$ is described
briefly. Now we give the full proof of the
inequalities~\eqref{example.qqqq.3} for all functions in
$\mathcal{L-P^*}$.

\vspace{3mm}

At first, we establish bounds on the number of real zeroes of $Q$
associated with a function~$\varphi$ in~$\mathcal{L-P}^*$ with
finitely many real zeroes. We express those bounds in terms of the
number of extra zeroes of $\varphi'$.

The following theorem concerns real polynomials multiplied by
exponentials.
\begin{theorem}\label{Th.main.theorem.for.finite.zeroes.function.1}
Let $p$ be a real polynomial and let $\varphi$ be the function
\begin{equation}\label{finite.zeroes.function.1}
\varphi(z)=e^{bz}p(z),
\end{equation}
where $b\neq0$ is a real number. If $\varphi$ possesses
property~A, then
\begin{equation}\label{finite.zeroes.function.1.condition}
2\left[\dfrac{E(\varphi')}2\right]\leqslant
Z_\mathbb{R}\left(Q\right)\leqslant2\left[\dfrac{E(\varphi')}2\right]+Z_\mathbb{R}\left(Q_{1}\right).
\end{equation}
Here $E(\varphi')$ is the number of extra zeroes of $\varphi'$.
\end{theorem}
\begin{proof} Without loss of generality, we assume that $b>0$ and
the leading coefficient of $p$ is also positive. Then
$\varphi(z)\to+\infty$ whenever $z\to+\infty$ and $\varphi(z)\to0$
whenever $z\to-\infty$.

\vspace{1mm}

We consider the following two cases: I. $p$ has no real zeroes,
and II. $p$ has at least one real zero.

\vspace{2mm}

\noindent Case I. Let $p(z)\neq0$ for $z\in\mathbb{R}$, then $\deg
p$ is even and $\varphi'(z)=e^{bz}[p'(z)+bp(z)]$. Thus, $\varphi'$
has no real zeroes or it has an even number of real zeroes, all of
which are extra zeroes of $\varphi'$. (This also follows from
Theorem~\ref{Theorem.number.of.extra.zeroes.type.I} with $n=0$.)

\vspace{1mm}

If $\varphi'(z)\neq0$ for $z\in\mathbb{R}$, then $E(\varphi')=0$,
and by Theorem~\ref{corol.real.axis.4}, we obtain the validity
of~\eqref{corol.real.axis.4.main.formula}, which is equivalent
to~\eqref{finite.zeroes.function.1.condition} in this case.

\vspace{1mm}

Let $\varphi'$ have an even number of real zeroes, say
$\beta_1\leqslant\beta_2\leqslant\ldots\leqslant\beta_{2r}$, where
$r>0$, then $E(\varphi')=2r$. Since $b>0$ and $\varphi(z)>0$ for
$z\in\mathbb{R}$ by assumption, $\varphi'(z)\to+\infty$ whenever
$z\to+\infty$ and ~$\varphi'(z)\to+\,0$ whenever $z\to-\infty$. In
particular, we have
\begin{equation}\label{Th.main.theorem.for.finite.zeroes.function.1.proof.0.1}
\varphi'(z)>0\quad\text{for}\quad
z\in(-\infty,\beta_1)\cup(\beta_{2r},+\infty).
\end{equation}
Also~$\varphi''(z)\to+\infty$ whenever $z\to+\infty$
and~$\varphi''(z)\to+\,0$ whenever $z\to-\infty$. By the
analyticity of~$\varphi'$, the following inequalities hold
\begin{equation}\label{Th.main.theorem.for.finite.zeroes.function.1.proof.1}
\varphi'(\beta_1-\varepsilon)\varphi''(\beta_1-\varepsilon)<0,
\end{equation}
\begin{equation}\label{Th.main.theorem.for.finite.zeroes.function.1.proof.2}
\varphi'(\beta_{2r}+\delta)\varphi''(\beta_{2r}+\delta)>0
\end{equation}
for all sufficiently small~$\varepsilon>0$ and $\delta>0$
(see~\eqref{corol.4.proof.2}). These inequalities
with~\eqref{Th.main.theorem.for.finite.zeroes.function.1.proof.0.1}
imply $\varphi''(\beta_1-\varepsilon)<0$ and
$\varphi''(\beta_{2r}+\delta)>0$ for all sufficiently small
positive $\varepsilon$ and $\delta$. Therefore, $\varphi''$ has an
odd number of zeroes in the interval $(-\infty,\beta_1)$ and an
even number of zeroes (or no zeroes) in the interval
$(\beta_{2r},+\infty)$. Moreover, $Q(\beta_1-\varepsilon)<0$
and~$Q(\beta_{2r}+\delta)>0$ for all sufficiently small~$\delta>0$
and~$\varepsilon>0$, since the sign of~$Q$ in a~small vicinity of
a real zero of~$\varphi'$ equals the sign of $\varphi\varphi''$ in
that vicinity (see~\eqref{main.function.2}) and since
$\varphi(z)>0$ for $z\in\mathbb{R}$ by assumption. As was
mentioned in Theorem~\ref{Theorem.finitude.number.of.zeroes.of.Q},
$Q(z)$ is negative for sufficiently large real $z$, consequently,
$Q$ has an even number of zeroes (possibly no zeroes) in the
interval~$(-\infty,\beta_1)$ and an odd number of zeroes
in~$(\beta_{2r},+\infty)$. Thus,
\begin{equation}\label{Th.main.theorem.for.finite.zeroes.function.1.proof.3}
Z_{(\beta_{2r},+\infty)}(Q)\geqslant1.
\end{equation}

From Corollary~\ref{biggest.technical.theorem} it follows that
\begin{equation}\label{Th.main.theorem.for.finite.zeroes.function.1.proof.4}
2r-1\leqslant Z_{[\beta_{1},\beta_{2r}]}(Q)\leqslant
2r-1+Z_{[\beta_{1},\beta_{2r}]}(Q_{1}).
\end{equation}

If $\varphi''$ has zeroes in~$(\beta_{2r},+\infty)$, say
$\gamma^{(+)}_1\leqslant\gamma^{(+)}_2\leqslant\ldots\leqslant\gamma^{(+)}_{2M}$,
$M>0$, then,
by~\eqref{Th.main.theorem.for.finite.zeroes.function.1.proof.2}
and by Case~II of Lemma~\ref{lemma.6.5}, we have
\begin{equation}\label{Th.main.theorem.for.finite.zeroes.function.1.proof.5}
0\leqslant Z_{[\gamma^{(+)}_{1},\gamma^{(+)}_{2M}]}(Q)\leqslant
-1+Z_{[\gamma^{(+)}_{1},\gamma^{(+)}_{2M}]}(Q_{1}).
\end{equation}
Lemma~\ref{lemma.5} and Remark~\ref{remark.2.4} yield
\begin{equation}\label{Th.main.theorem.for.finite.zeroes.function.1.proof.6}
Z_{(\gamma^{(+)}_{2M},+\infty)}(Q)\leqslant
1+Z_{(\gamma^{(+)}_{2M},+\infty)}(Q_{1}),
\end{equation}
\begin{equation}\label{Th.main.theorem.for.finite.zeroes.function.1.proof.7}
Z_{(\beta_{2r},\gamma^{(+)}_{1})}(Q)\leqslant
1+Z_{(\beta_{2r},\gamma^{(+)}_{1})}(Q_{1}).
\end{equation}
Then
from~\eqref{Th.main.theorem.for.finite.zeroes.function.1.proof.3}
and~\eqref{Th.main.theorem.for.finite.zeroes.function.1.proof.5}--\eqref{Th.main.theorem.for.finite.zeroes.function.1.proof.7}
it follows that
\begin{equation}\label{Th.main.theorem.for.finite.zeroes.function.1.proof.8}
1\leqslant Z_{(\beta_{2r},+\infty)}(Q)\leqslant
1+Z_{(\beta_{2r},+\infty)}(Q_{1}).
\end{equation}
These inequalities cannot be improved by parity considerations, since
$Z_{(\beta_{2r},+\infty)}(Q_{1})$ is even by Lemma~\ref{lemma.2}
applied to $Q_1$.

If $\varphi''(z)\neq0$ for $z\in(\beta_{2r},+\infty)$, then, by
Lemma~\ref{lemma.5} with Remark~\ref{remark.2.4} and
by~\eqref{Th.main.theorem.for.finite.zeroes.function.1.proof.3},
we obtain the
inequalities~\eqref{Th.main.theorem.for.finite.zeroes.function.1.proof.8}
again.

Since $\varphi''$ has an odd number of zeroes
in~$(-\infty,\beta_1)$, say
$\gamma^{(-)}_1\leqslant\gamma^{(-)}_2\leqslant\ldots\leqslant\gamma^{(-)}_{2N+1}$,
$N\geqslant0$, and $\varphi''(z)<0$ in
$(\gamma^{(-)}_{2N+1},\beta_1)$ as we showed above,
$\varphi''(z)>0$ in $(-\infty,\gamma^{(-)}_{1})$. Then we have
\begin{equation}\label{Th.main.theorem.for.finite.zeroes.function.1.proof.9}
0\leqslant Z_{(-\infty,\gamma^{(-)}_1)}(Q)\leqslant
Z_{(-\infty,\gamma^{(-)}_1)}(Q_{1}),
\end{equation}
\begin{equation}\label{Th.main.theorem.for.finite.zeroes.function.1.proof.10}
0\leqslant Z_{[\gamma^{(-)}_{1},\gamma^{(-)}_{2N+1}]}(Q)\leqslant
Z_{[\gamma^{(-)}_{1},\gamma^{(-)}_{2N+1}]}(Q_{1}),
\end{equation}
\begin{equation}\label{Th.main.theorem.for.finite.zeroes.function.1.proof.11}
0=Z_{(\gamma^{(-)}_{2N+1},\beta_{1})}(Q)\leqslant
Z_{(\gamma^{(-)}_{2N+1},\beta_{1})}(Q_{1}).
\end{equation}
Indeed, the
inequalities~\eqref{Th.main.theorem.for.finite.zeroes.function.1.proof.9}
follow from Lemma~\ref{lemma.5} with Remark~\ref{remark.2.4} and
the fact that $Q_1$ has an even number of zeroes in
$(-\infty,\gamma^{(-)}_{1})$ by Lemma~\ref{corol.1}
applied~to~$Q_1$. At the same time, the lower bound of the
inequalities~\eqref{Th.main.theorem.for.finite.zeroes.function.1.proof.9}
cannot be improved by parity considerations, since $Q$ has an even number of zeroes in the
interval $(-\infty,\gamma^{(-)}_1)$
by~\eqref{intermediate.10.lemma.6} and by the fact that $Q(z)<0$
for sufficiently large real~$z$. Further, the
inequalities~\eqref{Th.main.theorem.for.finite.zeroes.function.1.proof.10}
are exactly~\eqref{intermediate.11.lemma.6}. At last, since
$\varphi''(z)<0$ in $(\gamma^{(-)}_{2N+1},\beta_1)$ and
$\varphi(z)>0$ for $z\in\mathbb{R}$, $Q(z)$ has no zeroes in
$(\gamma^{(-)}_{2N+1},\beta_1)$ by~\eqref{main.function.2}. This
fact we use
in~\eqref{Th.main.theorem.for.finite.zeroes.function.1.proof.11}.

\vspace{2mm}

Now~\eqref{Th.main.theorem.for.finite.zeroes.function.1.proof.4},~\eqref{Th.main.theorem.for.finite.zeroes.function.1.proof.8}
and~\eqref{Th.main.theorem.for.finite.zeroes.function.1.proof.9}--\eqref{Th.main.theorem.for.finite.zeroes.function.1.proof.11}
imply
\begin{equation*}\label{Th.main.theorem.for.finite.zeroes.function.1.proof.12}
2r\leqslant Z_\mathbb{R}(Q)\leqslant 2r+Z_\mathbb{R}(Q_{1}).
\end{equation*}
These inequalities are equivalent
to~\eqref{finite.zeroes.function.1.condition}, since
$E(\varphi')=2r$.

\vspace{2mm}

\noindent Case II. Let $\varphi$ have at least one real zero.

\vspace{1mm}

Let $\alpha_S$ and $\alpha_L$ be the smallest and the largest
zeroes of $\varphi$. It is easy to see that, for all sufficiently
small $\varepsilon>0$,
\begin{equation}\label{Th.main.theorem.for.finite.zeroes.function.1.proof.14.5}
\varphi(\alpha_S-\varepsilon)\varphi'(\alpha_S-\varepsilon)<0.
\end{equation}
Also since $\varphi(z)\to0$ whenever $z\to-\infty$, we have, for
sufficiently large negative $C$,
\begin{equation}\label{Th.main.theorem.for.finite.zeroes.function.1.proof.14.5.1}
\varphi(C)\varphi'(C)>0.
\end{equation}
From~\eqref{Th.main.theorem.for.finite.zeroes.function.1.proof.14.5}--\eqref{Th.main.theorem.for.finite.zeroes.function.1.proof.14.5.1}
it follows that $\varphi'$ has an odd number of zeroes, counting
multiplicities, say $2r_{-}+1,r_{-}\geqslant0$, in
$(-\infty,\alpha_S)$, since $\varphi(z)\neq0$ in
$(-\infty,\alpha_S)$. Let $\beta^{-}_L$ be the largest zero of
$\varphi'$ in the interval $(-\infty,\alpha_S)$. Then by
Theorem~\ref{corol.4} and by Remark~\ref{remark.2.8}, we have
\begin{equation}\label{Th.main.theorem.for.finite.zeroes.function.1.proof.15}
2r_{-}\leqslant Z_{(-\infty,\beta^{-}_{L}]}(Q)\leqslant 2r_{-}+
Z_{(-\infty,\beta^{-}_{L}]}(Q_{1}),
\end{equation}
since $2r_{-}=2\left[\frac{2r_{-}+1}2\right]$.

Further, by assumption, $\varphi(z)$ is positive for
$z\in(\alpha_L,+\infty)$ and tends to $+\infty$ whenever~$z$
tends~to~$+\infty$. Therefore, $\varphi'(z)\to+\infty$ as
$z\to+\infty$ and, by~\eqref{corol.4.proof.1}, we have
$\varphi'(z)>0$ in a small right-sided neighbourhood of the point
$\alpha_L$. Consequently, $\varphi'$ has an even number of zeroes,
counting multiplicities, say $2r_{+}\geqslant0$, in the interval
$(\alpha_L,+\infty)$. If $\varphi'$ has at least one zero in
$(\alpha_L,+\infty)$ and $\beta^{+}_S$ is the smallest one, then
by Theorem~\ref{corol.4}, we have
\begin{equation}\label{Th.main.theorem.for.finite.zeroes.function.1.proof.14}
2r_{+}\leqslant Z_{[\beta^{+}_{S},+\infty)}(Q)\leqslant 2r_{+}+
Z_{[\beta^{+}_{S},+\infty)}(Q_{1})
\end{equation}
since $2r_{+}=2\left[\frac{2r_{+}}2\right]$.

\begin{itemize}
\item[II.1] Let $\varphi$ have a unique real zero $\alpha$. Then
$\alpha=\alpha_S=\alpha_L$.

If $\varphi'$ has no zeroes in the interval $(\alpha,+\infty)$ and
has exactly $2r_{-}+1$ zeroes, counting multiplicities, in
$(-\infty,\alpha)$, then from Theorem~\ref{corol.new} and
Remark~\ref{remark.2.6} it follows that
\begin{equation}\label{Th.main.theorem.for.finite.zeroes.function.1.proof.15.55}
0\leqslant Z_{(\beta^{-}_{L},+\infty)}(Q)\leqslant
Z_{(\beta^{-}_{L},+\infty)}(Q_{1}),
\end{equation}
where $\beta^{-}_{L}$ is the largest zero of $\varphi'$ in the
interval $(-\infty,\alpha)$. Now the
inequalities~\eqref{Th.main.theorem.for.finite.zeroes.function.1.proof.15}
and~\eqref{Th.main.theorem.for.finite.zeroes.function.1.proof.15.55}
imply~\eqref{finite.zeroes.function.1.condition}, since
$E(\varphi')=2r_{-}+1$ in this case.

Let $\varphi'$ have exactly $2r_{+}>0$ zeroes, counting
multiplicities, in the interval $(\alpha,+\infty)$ and let
$\beta^{+}_S$ be the smallest one. By Theorem~\ref{corol.new}, we
have
\begin{equation}\label{Th.main.theorem.for.finite.zeroes.function.1.proof.15.56}
0\leqslant Z_{(\beta^{-}_{L},\beta^{+}_S)}(Q)\leqslant
Z_{(\beta^{-}_{L},\beta^{+}_S)}(Q_{1}).
\end{equation}
This inequality, together
with~\eqref{Th.main.theorem.for.finite.zeroes.function.1.proof.15}
and~\eqref{Th.main.theorem.for.finite.zeroes.function.1.proof.14},
gives~\eqref{finite.zeroes.function.1.condition}, since
$E(\varphi')=2r_{+}+2r_{-}+1$ in this case.
\item[II.2]
Let $\varphi$ have exactly $l\geqslant2$ \textit{distinct} real
zeroes, say $\alpha_S=\alpha_1<\alpha_2<\ldots<\alpha_l=\alpha_L$.
By Rolle's theorem, $\varphi'$ has an odd number of zeroes, say
$2M_i+1$, $M_i\geqslant0$, counting multiplicities, in each of the
intervals~$(\alpha_i,\alpha_{i+1})$, $i=1,2,\ldots,l-1$. If we
denote by $\beta^{(1)}_S$ and $\beta^{(l-1)}_L$ the smallest zero
of $\varphi'$ in the interval $(\alpha_1,\alpha_{2})$ and the
largest zero of $\varphi'$ in the interval
$(\alpha_{l-1},\alpha_l)$, then, by
Corollary~\ref{biggest.technical.theorem} and by
Theorem~\ref{corol.new}, we have
\begin{equation}\label{Th.main.theorem.for.finite.zeroes.function.1.proof.13}
\displaystyle\sum_{i=1}^{l-1}2M_i\leqslant
Z_{[\beta^{(1)}_{S},\beta^{(l-1)}_{L}]}(Q)\leqslant
\sum_{i=1}^{l-1}2M_i+Z_{[\beta^{(1)}_{S},\beta^{(l-1)}_{L}]}(Q_{1}).
\end{equation}
In the interval $(\beta^{-}_L,\beta^{(1)}_{S})$, we use
Theorem~\ref{corol.new} to yield
\begin{equation}\label{Th.main.theorem.for.finite.zeroes.function.1.proof.20}
0\leqslant Z_{(\beta^{-}_L,\beta^{(1)}_{S})}(Q)\leqslant
Z_{(\beta^{-}_L,\beta^{(1)}_{S})}(Q_{1}),
\end{equation}
where $\beta^{-}_L$ is the largest zero of $\varphi'$ in the
interval $(-\infty,\alpha_1)$. The existence of this zero was
proved above.

If $\varphi'$ has no zeroes in the interval $(\alpha_l,+\infty)$,
then, by Theorem~\ref{corol.new} and by Remark~\ref{remark.2.6},
we have
\begin{equation}\label{Th.main.theorem.for.finite.zeroes.function.1.proof.21}
0\leqslant Z_{(\beta^{(l-1)}_{L},+\infty)}(Q)\leqslant
Z_{(\beta^{(l-1)}_{L},+\infty)}(Q_{1}).
\end{equation}
Thus, if we suppose that $\varphi'$ has $2r_{-}+1$ zeroes,
counting multiplicities, in $(-\infty,\alpha_1)$, then
$E(\varphi')=2r_{-}+1+\sum_{i=1}^{l-1}2M_i$, so the
inequalities~\eqref{Th.main.theorem.for.finite.zeroes.function.1.proof.15},
\eqref{Th.main.theorem.for.finite.zeroes.function.1.proof.20},
\eqref{Th.main.theorem.for.finite.zeroes.function.1.proof.13}
and~\eqref{Th.main.theorem.for.finite.zeroes.function.1.proof.21}
imply~\eqref{finite.zeroes.function.1.condition}.

If $\varphi'$ has at least one zero in the interval
$(\alpha_l,+\infty)$ and $\beta^{+}_S$ is the smallest one, then
by Theorem~\ref{corol.new}, we have
\begin{equation}\label{Th.main.theorem.for.finite.zeroes.function.1.proof.22}
0\leqslant Z_{(\beta^{(l-1)}_{L},\beta^{+}_S)}(Q)\leqslant
Z_{(\beta^{(l-1)}_{L},\beta^{+}_S)}(Q_{1}).
\end{equation}
If we suppose that $\varphi'$ has $2r_{+}>0$ zeroes, counting
multiplicities, in the interval $(\alpha_l,+\infty)$, and it has
$2r_{-}+1$ zeroes, counting multiplicities, in
$(-\infty,\alpha_1)$, then we obtain the following
$E(\varphi')=2r_{-}+2r_{+}+1+\sum_{i=1}^{l-1}2M_i$,
so~\eqref{finite.zeroes.function.1.condition} follows
from~\eqref{Th.main.theorem.for.finite.zeroes.function.1.proof.15},
\eqref{Th.main.theorem.for.finite.zeroes.function.1.proof.20},
\eqref{Th.main.theorem.for.finite.zeroes.function.1.proof.13},
\eqref{Th.main.theorem.for.finite.zeroes.function.1.proof.22}
and~\eqref{Th.main.theorem.for.finite.zeroes.function.1.proof.14}.
\end{itemize}
\end{proof}

\noindent Now we derive a bound on the number of real zeroes of
$Q$ associated with a real polynomial.
\begin{theorem}\label{Th.main.theorem.for.finite.zeroes.function.2}
Let $\varphi$ be a real polynomial possessing property~A.
\begin{itemize}
\item[\emph{I.}] If $\varphi(z)\neq0$ for $z\in\mathbb{R}$, then
\begin{equation}\label{finite.zeroes.function.2.condition.without.real.zeroes}
2\left[\dfrac{E(\varphi')}2\right]+2\leqslant
Z_\mathbb{R}\left(Q\right)\leqslant2\left[\dfrac{E(\varphi')}2\right]+2+Z_\mathbb{R}\left(Q_{1}\right),
\end{equation}
where $E(\varphi')$ is the number of extra zeroes of $\varphi'$.
\item[\emph{II.}] If $\varphi$ has at least one real zero, then the
inequalities~\eqref{finite.zeroes.function.1.condition} hold.
\end{itemize}
\end{theorem}
\begin{proof} Without loss of generality, we may assume that the
leading coefficient of $\varphi$ is positive. Then
$\varphi(z)\to+\infty$ whenever $z\to+\infty$ and
$\varphi(z)\to\pm\infty$ whenever $z\to-\infty$.

\vspace{2mm}

\noindent I. Let $\varphi(z)\neq0$ for $z\in\mathbb{R}$. Then
$\deg\varphi$ is even and $\deg\varphi'=\deg\varphi-1$.
Consequently, $\varphi'$ has an~odd number of zeroes, all of which
are extra zeroes of $\varphi'$ (see also
Theorem~\ref{Theorem.number.of.extra.zeroes.type.II} with $n=0$).

\vspace{1mm}

We denote the real zeroes of $\varphi'$ by
$\beta_1\leqslant\beta_2\leqslant\ldots\leqslant\beta_{2r+1}$, so
$E(\varphi')=2r+1$. By assumption, $\varphi(z)>0$ for
$z\in\mathbb{R}$, therefore,~$\varphi'(z)\to+\infty$ whenever
$z\to+\infty$ and $\varphi'(z)\to-\infty$ whenever $z\to-\infty$.
Consequently, $\varphi'(z)<0$ for $z\in(-\infty,\beta_1)$ and
$\varphi'(z)>0$ for $z\in(\beta_{2r+1},+\infty)$.
Also~$\varphi''(z)\to+\infty$ whenever $z\to\pm\infty$. As in the
proof of
Theorem~\ref{Th.main.theorem.for.finite.zeroes.function.1}, one
can show that
the~inequality~\eqref{Th.main.theorem.for.finite.zeroes.function.1.proof.1}
holds and implies $\varphi''(\beta_1-\varepsilon)>0$ for all
sufficiently small~$\varepsilon>0$. Analogously
to~\eqref{Th.main.theorem.for.finite.zeroes.function.1.proof.2},
we have
\begin{equation}\label{Th.main.theorem.for.finite.zeroes.function.2.proof.2}
\varphi'(\beta_{2r+1}+\delta)\varphi''(\beta_{2r+1}+\delta)>0
\end{equation}
for all sufficiently small~$\delta>0$, that is,
$\varphi''(\beta_{2r+1}+\delta)>0$. Therefore, $\varphi''$ has an
even number of zeroes in each of the intervals $(-\infty,\beta_1)$
and $(\beta_{2r+1},+\infty)$. Moreover, since $\varphi(z)>0$ for
$z\in\mathbb{R}$ by assumption, $Q(\beta_1-\varepsilon)>0$ and
$Q(\beta_{2r+1}+\delta)>0$ for all sufficiently small~$\delta>0$
and~$\varepsilon>0$ according to~\eqref{main.function.2}.
But~$Q(z)<0$ for all sufficiently large real $z$ (see the proof of
Theorem~\ref{Theorem.finitude.number.of.zeroes.of.Q}),
consequently, $Q$ has an odd number of zeroes in each of the
intervals~$(-\infty,\beta_1)$ and~$(\beta_{2r+1},+\infty)$. Thus,
\begin{equation}\label{Th.main.theorem.for.finite.zeroes.function.2.proof.2.5}
Z_{(-\infty,\beta_{1})}(Q)\geqslant1,
\end{equation}
\begin{equation}\label{Th.main.theorem.for.finite.zeroes.function.2.proof.3}
Z_{(\beta_{2r+1},+\infty)}(Q)\geqslant1.
\end{equation}
By Corollary~\ref{biggest.technical.theorem}, we have
\begin{equation}\label{Th.main.theorem.for.finite.zeroes.function.2.proof.4}
2r\leqslant Z_{[\beta_{1},\beta_{2r+1}]}(Q)\leqslant
2r+Z_{[\beta_{1},\beta_{2r+1}]}(Q_{1}).
\end{equation}

Since $\varphi''$ has an even number of zeroes in each of the
intervals $(-\infty,\beta_{1})$ and $(\beta_{2r+1},+\infty)$, in
the same way as in Case I of
Theorem~\ref{Th.main.theorem.for.finite.zeroes.function.1}
(see~\eqref{Th.main.theorem.for.finite.zeroes.function.1.proof.8}),
one can show that
\begin{equation}\label{Th.main.theorem.for.finite.zeroes.function.2.proof.10}
1\leqslant Z_{(-\infty,\beta_{1})}(Q)\leqslant1+
Z_{(-\infty,\beta_{1})}(Q_{1}),
\end{equation}
\begin{equation}\label{Th.main.theorem.for.finite.zeroes.function.2.proof.11}
1\leqslant Z_{(\beta_{2r+1},+\infty)}(Q)\leqslant1+
Z_{(\beta_{2r+1},+\infty)}(Q_{1}).
\end{equation}
Now~\eqref{Th.main.theorem.for.finite.zeroes.function.2.proof.4}--\eqref{Th.main.theorem.for.finite.zeroes.function.2.proof.11}
imply
\begin{equation*}\label{Th.main.theorem.for.finite.zeroes.function.2.proof.12}
2r+2\leqslant Z_\mathbb{R}(Q)\leqslant 2r+2+Z_\mathbb{R}(Q_{1}),
\end{equation*}
which is equivalent
to~\eqref{finite.zeroes.function.2.condition.without.real.zeroes},
since $E(\varphi')=2r+1$ and, therefore,
$2\left[\frac{E(\varphi')}2\right]=2r$.

\vspace{2mm}

\noindent II. Let $\varphi$ have at least one real zero and let
$\alpha_S$ and $\alpha_L$ be the smallest and the largest
zeroes~of~$\varphi$, respectively. It is easy to see that the
inequality~\eqref{Th.main.theorem.for.finite.zeroes.function.1.proof.14.5}
holds in this case. Moreover, since the function $\varphi(z)$
tends to $\pm\infty$ as $z\to-\infty$, we have, for sufficiently
large negative $C$, $\varphi(C)\varphi'(C)<0$. Consequently,
$\varphi'$ has an even number of zeroes (or has no zeroes) in
$(-\infty,\alpha_S)$, since $\varphi(z)\neq0$ for
$z\in(-\infty,\alpha_S)$. Analogously, $\varphi'$ has an even
number of zeroes (or has no zeroes) in $(\alpha_L,+\infty)$.
Further, we consider the following two cases.

\begin{itemize}
\item[II.1] Let $\varphi$ have a unique real zero $\alpha=\alpha_S=\alpha_L$.
If $\varphi'(z)\neq0$ for $z\in\mathbb{R}\backslash\{\alpha\}$,
then by Theorem~\ref{corol.real.exis.6}, the
inequality~\eqref{corol.real.axis.4.main.formula} holds. The
latter is equivalent
to~\eqref{finite.zeroes.function.1.condition}, since
$E(\varphi')=0$ in this~case.

Now suppose that $\varphi'$ has at least one zero in exactly one
of the intervals $(-\infty,\alpha)$ and $(\alpha,+\infty)$.
Without loss of generality, we may assume that $\varphi'(z)\neq0$
for $z\in(-\infty,\alpha)$ and $\varphi'$ has an even number of
zeroes, counting multiplicities, say $2r_{+}>0$, in
$(\alpha,+\infty)$. Let $\beta^{+}_S$ be the smallest zero of
$\varphi'$ in this interval. Then by the same argument as in
Case~II of
Theorem~\ref{Th.main.theorem.for.finite.zeroes.function.1}, one
can prove that the
inequalities~\eqref{Th.main.theorem.for.finite.zeroes.function.1.proof.14}
hold. Moreover, by Theorem~\ref{corol.new} and
Remark~\ref{remark.2.6}, we have
\begin{equation*}\label{Th.main.theorem.for.finite.zeroes.function.2.proof.14}
0\leqslant Z_{(-\infty,\beta^{+}_{S})}(Q)\leqslant
Z_{(-\infty,\beta^{+}_{S})}(Q_{1}).
\end{equation*}
Combined
with~\eqref{Th.main.theorem.for.finite.zeroes.function.1.proof.14},
these inequalities
give~\eqref{finite.zeroes.function.1.condition}, since
$E(\varphi')=2r_{+}$ in this case.

At last, let $\varphi$ have a unique real zero $\alpha$ and let
$\varphi'$ has at least one zero in each of the intervals
$(-\infty,\alpha)$ and $(\alpha,+\infty)$. Let $\varphi'$ have
exactly $2r_{-}>0$ zeroes, counting multiplicities, in the
interval $(-\infty,\alpha)$ and exactly $2r_{+}>0$ zeroes,
counting multiplicities, in $(\alpha,+\infty)$. Let $\beta^{-}_L$
be the largest zero of $\varphi'$ in $(-\infty,\alpha)$ and let
$\beta^{+}_S$ be the smallest zero of $\varphi'$ in
$(\alpha,+\infty)$. From Theorems~\ref{corol.4}
and~\ref{corol.new} it follows that the
inequalities~\eqref{Th.main.theorem.for.finite.zeroes.function.1.proof.15},~\eqref{Th.main.theorem.for.finite.zeroes.function.1.proof.14}
and~\eqref{Th.main.theorem.for.finite.zeroes.function.1.proof.15.56}
hold. These inequalities
imply~\eqref{finite.zeroes.function.1.condition}, since
$E(\varphi')=2r_{-}+2r_{+}$ in this case.

\item[II.2] Let now $\varphi'$ have $l\geqslant2$ \textit{distinct}
real zeroes, counting multiplicities, say
$\alpha_1<\ldots<\alpha_l$. Then by Rolle's theorem, $\varphi'$
has an odd number of zeroes, say $2M_i+1$, $M_i\geqslant0$,
counting multiplicities, in each of the
intervals~$(\alpha_i,\alpha_{i+1})$, $i=1,2,\ldots,l-1$. If
$\beta^{(1)}_S$ is the smallest zero of $\varphi'$ in the interval
$(\alpha_1,\alpha_{2})$ and $\beta^{(l-1)}_L$ is the largest zero
of $\varphi'$ in the interval $(\alpha_{l-1},\alpha_l)$, then, as
in the proof of Case~II of
Theorem~\ref{Th.main.theorem.for.finite.zeroes.function.1}, the
inequalities~\eqref{Th.main.theorem.for.finite.zeroes.function.1.proof.13}
hold by Corollary~\ref{biggest.technical.theorem} and
Theorem~\ref{corol.new}. Moreover, if $\varphi'$ has
$2r_{-}\geqslant0$ zeroes, counting multiplicities, in the
interval $(-\infty,\alpha_1)$ and $2r_{+}\geqslant0$ zeroes,
counting multiplicities, in the interval $(\alpha_l,+\infty)$,
then Theorems~\ref{corol.4} and~\ref{corol.new} with
Remark~\ref{remark.2.6} imply
\begin{equation}\label{Th.main.theorem.for.finite.zeroes.function.2.proof.155}
2r_{-}\leqslant Z_{(-\infty,\beta^{(1)}_S)}(Q)\leqslant 2r_{-}+
Z_{(-\infty,\beta^{(1)}_S)}(Q_{1})
\end{equation}
and
\begin{equation}\label{Th.main.theorem.for.finite.zeroes.function.2.proof.144}
2r_{+}\leqslant Z_{(\beta^{(l-1)}_L,+\infty)}(Q)\leqslant 2r_{+}+
Z_{(\beta^{(l-1)}_L,+\infty)}(Q_{1}),
\end{equation}
since $2r_{-}=2\left[\frac{2r_{-}}2\right]$ and
$2r_{+}=2\left[\frac{2r_{+}}2\right]$. Obviously, we have
$E(\varphi')=2r_{+}+2r_{-}+\sum_{i=1}^{l-1}2M_i$ in this case.
Therefore, the
inequalities~\eqref{finite.zeroes.function.1.condition} follow
from~\eqref{Th.main.theorem.for.finite.zeroes.function.1.proof.13},
\eqref{Th.main.theorem.for.finite.zeroes.function.2.proof.155}
and~\eqref{Th.main.theorem.for.finite.zeroes.function.2.proof.144}.
\end{itemize}
\vspace{-7mm}
\end{proof}

Our next theorem concerns the last subclass of entire functions
in~$\mathcal{L-P^*}$ with finitely many zeroes, namely,
polynomials multiplied by exponentials of the form
$e^{-az^{2}+bz}$.

\begin{theorem}\label{Th.main.theorem.for.finite.zeroes.function.3}
Let $p$ be a real polynomial and let $\varphi$ be the function
\begin{equation}\label{finite.zeroes.function.3}
\displaystyle\varphi(z)=e^{-az^{2}+bz}p(z),
\end{equation}
where $a>0$ and $b\in\mathbb{R}$. Let $\varphi$ possess
property~A.
\begin{itemize}
\item[\emph{I.}] If $p$ has no real zeroes, then the
inequalities~\eqref{finite.zeroes.function.1.condition} hold.
\item[\emph{II.}] If $p$ has at least one real zero, then
\begin{equation}\label{finite.zeroes.function.3.condition.with.real.zeroes}
2\left[\dfrac{E(\varphi')}2\right]-2\leqslant
Z_\mathbb{R}\left(Q\right)\leqslant2\left[\dfrac{E(\varphi')}2\right]-2+Z_\mathbb{R}\left(Q_{1}\right),
\end{equation}
where $E(\varphi')$ is the number of extra zeroes of $\varphi'$.
\end{itemize}
\end{theorem}
\begin{proof} Without loss of generality, we may assume that the
leading coefficient of $p$ is positive. Then $\varphi(z)\to+\,0$
whenever $z\to+\infty$ and $\varphi(z)\to0$ whenever
$z\to-\infty$.

\vspace{2mm}

\noindent I. Let $p(z)\neq0$ for $z\in\mathbb{R}$. Then $\deg p$
is even and $\varphi'(z)=e^{-az^{2}+bz}g(z)$, where $g$ is a real
polynomial of odd degree, which is equal to $\deg p+1$. Thus,
$\varphi'$ has an odd number of real zeroes, counting
multiplicities, say
$\beta_1\leqslant\beta_2\leqslant\ldots\leqslant\beta_{2r+1}$, all
of which are extra zeroes of $\varphi'$ (see also
Theorem~\ref{Theorem.number.of.extra.zeroes.type.III} with $n=0$).
Therefore,
the~inequalities~\eqref{Th.main.theorem.for.finite.zeroes.function.2.proof.4}
hold in this case.

Since $p(z)\neq0$ for $z\in\mathbb{R}$, $\varphi(z)>0$ on the real
axis and $\varphi(z)\to0$ whenever $z\to\pm\infty$ by assumption.
Then~$\varphi'(z)\to-\,0$ whenever $z\to+\infty$, and
$\varphi'(z)\to+\,0$ whenever $z\to-\infty$. In particular,
$\varphi'(z)>0$ for $z\in(-\infty,\beta_1)$ and $\varphi'(z)<0$
for $z\in(\beta_{2r+1},+\infty)$. Also~$\varphi''(z)\to0$ whenever
$z\to\pm\infty$ and $\varphi(z)>0$ when $z$ is real and $|z|$ is
sufficiently large. From the
inequality~\eqref{Th.main.theorem.for.finite.zeroes.function.1.proof.1}
it follows that $\varphi''(\beta_1-\varepsilon)<0$ for all
sufficiently small $\varepsilon>0$. By a similar reasoning, one
can show that $\varphi''(\beta_{2r+1}+\delta)<0$ for all
sufficiently small $\delta>0$. Consequently, $\varphi''$ has an
odd number of zeroes in each of the intervals $(-\infty,\beta_1)$
and $(\beta_{2r+1},+\infty)$. Moreover, since $\varphi(z)$ is
positive on the real line by assumption,
$Q(\beta_1-\varepsilon)<0$ and $Q(\beta_{2r+1}+\delta)<0$ for all
sufficiently small $\varepsilon>0$ and $\delta>0$
by~\eqref{main.function.2}. Consequently, $Q$ has an even number
of zeroes in each of the intervals $(-\infty,\beta_1)$ and
$(\beta_{2r+1},+\infty)$, because $Q(z)<0$ for all sufficiently
large real $z$ (see the proof of
Theorem~\ref{Theorem.finitude.number.of.zeroes.of.Q}). Thus, in
the same way as in the proof of Case~I of
Theorem~\ref{Th.main.theorem.for.finite.zeroes.function.1}
(see~\eqref{Th.main.theorem.for.finite.zeroes.function.1.proof.9}--\eqref{Th.main.theorem.for.finite.zeroes.function.1.proof.11}),
it is easy to show that
\begin{equation}\label{Th.main.theorem.for.finite.zeroes.function.3.proof.1}
0\leqslant Z_{(-\infty,\beta_1)}(Q)\leqslant
Z_{(-\infty,\beta_1)}(Q_{1}).
\end{equation}
And analogously,
\begin{equation}\label{Th.main.theorem.for.finite.zeroes.function.3.proof.2}
0\leqslant Z_{(\beta_{2r+1},+\infty)}(Q)\leqslant
Z_{(\beta_{2r+1},+\infty)}(Q_{1}).
\end{equation}
Since $E(\varphi')=2r+1$ in this case, the
inequalities~\eqref{Th.main.theorem.for.finite.zeroes.function.2.proof.4}
and~\eqref{Th.main.theorem.for.finite.zeroes.function.3.proof.1}--\eqref{Th.main.theorem.for.finite.zeroes.function.3.proof.2}
imply~\eqref{finite.zeroes.function.1.condition}.

\vspace{2mm}

\noindent II. Let $p$ have at least one real zero and let
$\alpha_S$ and $\alpha_L$ be the smallest and the~largest zeroes
of $\varphi$. It is easy to see that the
inequalities~\eqref{Th.main.theorem.for.finite.zeroes.function.1.proof.14.5}
and~\eqref{Th.main.theorem.for.finite.zeroes.function.1.proof.14.5.1}
hold in this case. From these inequalities it follows that
$\varphi'$ has an odd number of zeroes, counting multiplicities,
say $2r_{-}+1,r_{-}\geqslant0$, in $(-\infty,\alpha_S)$, since
$\varphi(z)\neq0$ in this interval. Analogously, one can show that
$\varphi'$ has an odd number of zeroes, counting multiplicities,
say $2r_{+}+1,r_{+}\geqslant0$, in the interval
$(\alpha_L,+\infty)$. Let $\beta^{-}_L$ be the~largest zero of
$\varphi'$ in the interval $(-\infty,\alpha_S)$ and let
$\beta^{+}_S$ be the smallest zero of $\varphi'$ in
$(\alpha_L,+\infty)$. Then by Theorem~\ref{corol.4}, we have
\begin{equation}\label{Th.main.theorem.for.finite.zeroes.function.3.proof.3}
2r_{-}\leqslant Z_{(-\infty,\beta^{-}_{L}]}(Q)\leqslant 2r_{-}+
Z_{(-\infty,\beta^{-}_{L}]}(Q_{1}),
\end{equation}
\begin{equation}\label{Th.main.theorem.for.finite.zeroes.function.3.proof.4}
2r_{+}\leqslant Z_{[\beta^{+}_{S},+\infty)}(Q)\leqslant 2r_{+}+
Z_{[\beta^{+}_{S},+\infty)}(Q_{1}),
\end{equation}
since $2r_{-}=2\left[\frac{2r_{-}+1}2\right]$ and
$2r_{+}=2\left[\frac{2r_{+}+1}2\right]$.

\vspace{1mm}

If $\varphi$ has a unique real zero $\alpha=\alpha_S=\alpha_L$,
then by Theorem~\ref{corol.new}, we have
\begin{equation}\label{Th.main.theorem.for.finite.zeroes.function.3.proof.7}
0\leqslant Z_{(\beta^{-}_{L},\beta^{+}_{S})}(Q)\leqslant
Z_{(\beta^{-}_{L},\beta^{+}_{S})}(Q_{1}).
\end{equation}
Summing the
inequalities~\eqref{Th.main.theorem.for.finite.zeroes.function.3.proof.3}--\eqref{Th.main.theorem.for.finite.zeroes.function.3.proof.7},
we
obtain~\eqref{finite.zeroes.function.3.condition.with.real.zeroes},
since $E(\varphi')=2r_{-}+1+2r_{+}+1$ in this case.

\vspace{1mm}

If $\varphi$ has exactly $l\geqslant2$ \textit{distinct} zeroes,
say $\alpha_1<\alpha_2<\ldots<\alpha_l$, then by Rolle's theorem,
$\varphi'$ has an odd number of zeroes, say $2M_i+1$,
$M_i\geqslant0$, counting multiplicities, in each of the
intervals~$(\alpha_i,\alpha_{i+1})$, $i=1,2,\ldots,l-1$. As above,
from Corollary~\ref{biggest.technical.theorem} and
Theorem~\ref{corol.new} it follows that
\begin{equation}\label{Th.main.theorem.for.finite.zeroes.function.3.proof.13}
\displaystyle\sum_{i=1}^{l-1}2M_i\leqslant
Z_{(\beta^{-}_{L},\beta^{+}_{S})}(Q)\leqslant
\sum_{i=1}^{l-1}2M_i+Z_{(\beta^{-}_{L},\beta^{+}_{S})}(Q_{1}).
\end{equation}
Now the
inequalities~\eqref{Th.main.theorem.for.finite.zeroes.function.3.proof.3}--\eqref{Th.main.theorem.for.finite.zeroes.function.3.proof.4}
and~\eqref{Th.main.theorem.for.finite.zeroes.function.3.proof.13}
imply
\begin{equation*}\label{Th.main.theorem.for.finite.zeroes.function.3.proof.5}
\displaystyle2r_{-}+2r_{+}+\sum_{i=1}^{l-1}2M_i\leqslant
Z_\mathbb{R}\left(Q\right)\leqslant2r_{-}+2r_{+}+\sum_{i=1}^{l-1}2M_i+Z_\mathbb{R}\left(Q_{1}\right),
\end{equation*}
which is equivalent
to~\eqref{finite.zeroes.function.3.condition.with.real.zeroes},
since $E(\varphi')=2r_{-}+1+2r_{+}+1+\sum_{i=1}^{l-1}2M_i$.
\end{proof}

Theorems~\ref{Th.main.theorem.for.finite.zeroes.function.1}--\ref{Th.main.theorem.for.finite.zeroes.function.3}
describe all functions in $\mathcal{L-P^*}$ with \textit{finitely}
many zeroes. The next theorem provides a bound on the number of
real zeroes of $Q$ associated with a function in $\mathcal{L-P^*}$
with \textit{infinitely} many real zeroes.
\begin{theorem}\label{Th.main.theorem.for.infinite.zeroes.functions}
Let $\varphi\in\mathcal{L-P^*}$ and suppose that~$\varphi$ has
infinitely many real zeroes. If $\varphi$ possesses property~A,
then the inequalities~\eqref{finite.zeroes.function.1.condition}
hold.
\end{theorem}
\begin{proof}
If $\varphi$ has infinitely many positive and negative zeroes,
then $\varphi'$ can have extra zeroes only between two consecutive
zeroes of $\varphi$. Therefore, $E(\varphi')$ is an even number
and the inequalities~\eqref{finite.zeroes.function.1.condition}
follow from Corollary~\ref{biggest.technical.theorem} and
Theorem~\ref{corol.new}.

Let $\varphi$ have infinitely many real zeroes but only finitely
many positive or negative zeroes. Without loss of generality, we
may assume that $\varphi$ has the largest real zero, say
$\alpha_L$. If $\beta<\alpha_L$ is a zero of~$\varphi'$ such that
$\varphi(z)\neq0$ and $\varphi'(z)\neq0$ for
$z\in(\beta,\alpha_L)$, then applying
Corollary~\ref{biggest.technical.theorem} and
Theorem~\ref{corol.new} to the interval~$(-\infty,\beta]$ and
Theorems~\ref{corol.new} and~\ref{corol.4} (see also
Remark~\ref{remark.2.6}) to the interval~$(\beta,+\infty)$, we
again obtain~\eqref{finite.zeroes.function.1.condition}.
\end{proof}

Now we are able to prove the following general theorem.

\begin{theorem}\label{technical.theorem.condition.A}
Let $\varphi\in\mathcal{L-P^*}$ and suppose that $\varphi$
has exactly $2m$ nonreal zeroes. If~$\varphi$ possesses
property~A, then
\begin{equation}\label{main.result.2}
2m-2m_{1}\leqslant
Z_\mathbb{R}\left(Q\right)\leqslant2m-2m_{1}+Z_\mathbb{R}\left(Q_{1}\right),
\end{equation}
where $2m_1=Z_\mathbb{C}(\varphi')$.
\end{theorem}
\begin{proof} The theorem follows from
Theorems~\ref{Theorem.number.of.extra.zeroes.type.I}--\ref{Theorem.number.of.extra.zeroes.type.IV}
combined with
Theorems~\ref{Th.main.theorem.for.finite.zeroes.function.1}--\ref{Th.main.theorem.for.infinite.zeroes.functions}.

In fact, consider the following four cases.
\begin{itemize}
\item[I.] Let $\varphi$ be a real polynomial. Then $\varphi'$ has
exactly $\deg\varphi-1$ zeroes, counting multiplicities.

If $\varphi(z)\neq0$ for $z\in\mathbb{R}$, then by Theorem~\ref{Theorem.number.of.extra.zeroes.type.II}
(for $n=0$), we have $E(\varphi')=2m-2m_1-1$ (see~\eqref{Number.of.extra.zeroes.of.function.type.II}),
so according to Theorem~\ref{Th.main.theorem.for.finite.zeroes.function.2}, the
inequalities~\eqref{finite.zeroes.function.2.condition.without.real.zeroes}
imply~\eqref{main.result.2}.

If $\varphi$ has at least one real zero, then Theorem~\ref{Theorem.number.of.extra.zeroes.type.II}
(for $n=0$) implies $E(\varphi')=2m-2m_1$, so the inequalities~\eqref{main.result.2} follow
from~\eqref{finite.zeroes.function.1.condition} by
Theorem~\ref{Th.main.theorem.for.finite.zeroes.function.2}.
\item[II.] Let $\varphi$ be a function of the
type~\eqref{finite.zeroes.function.1}.

If $\varphi(z)\neq0$ for $z\in\mathbb{R}$, then by Theorem~\ref{Theorem.number.of.extra.zeroes.type.I} (for $n=0$),
$E(\varphi')=2m-2m_1$, so
from~\eqref{finite.zeroes.function.1.condition} we
obtain~\eqref{main.result.2} by
Theorem~\ref{Th.main.theorem.for.finite.zeroes.function.1}.

If $\varphi$ has at least one real zero, Theorem~\ref{Theorem.number.of.extra.zeroes.type.I} (for $n=0$) gives
$E(\varphi')=2m-2m_1+1$, and the
inequalities~\eqref{main.result.2} again follow
from~\eqref{finite.zeroes.function.1.condition}  by
Theorem~\ref{Th.main.theorem.for.finite.zeroes.function.1}.
\item[III.] If $\varphi$ is a function of the
type~\eqref{finite.zeroes.function.3}, Theorem~\ref{Theorem.number.of.extra.zeroes.type.III} (for $n=0$) implies
$E(\varphi')=2m-2m_1+1$, and
from~\eqref{finite.zeroes.function.1.condition} it follows that the
inequalities~\eqref{main.result.2} hold
by~Theorem~\ref{Th.main.theorem.for.finite.zeroes.function.3}.

If $\varphi$ has at least one real zero and $\deg p=2m+r$, then according to
Theorem~\ref{Theorem.number.of.extra.zeroes.type.III} (for $n=0$), we have $E(\varphi')=2m-2m_1+2$.
Now the inequalities~\eqref{main.result.2} follow
from~\eqref{finite.zeroes.function.3.condition.with.real.zeroes} by
Theorem~\ref{Th.main.theorem.for.finite.zeroes.function.3}.
\item[IV.] If $\varphi$ has infinitely many real zeroes, then
by Theorem~\ref{Theorem.number.of.extra.zeroes.type.IV} (for $n=0$), we have
\begin{equation*}\label{Th.main.theorem.final.proof.1}
2m-2m_1\leqslant E(\varphi')\leqslant2m-2m_1+1.
\end{equation*}
From these inequalities
and from~\eqref{finite.zeroes.function.1.condition} we
obtain~\eqref{main.result.2} by
Theorem~\ref{Th.main.theorem.for.infinite.zeroes.functions}.
\end{itemize}
\end{proof}

\begin{note}\label{remark.Edwards}
Example~\ref{example.Edwards} constructed by
S.\,Edwards~\cite{Edwards} can be used to show that
Theorem~\ref{technical.theorem.condition.A} is not valid for every
function in~$\mathcal{L-P^*}$. Indeed, for the
polynomial~\eqref{poly.example.Edwards}, we have
$Z_\mathbb{C}(p)=50$, $Z_\mathbb{C}(p')=50$, $Z_\mathbb{R}(Q)=4$
and $Z_\mathbb{R}(Q_1)=2$. Thus, that polynomial does not enjoy
the inequalities~\eqref{main.result.2} of
Theorem~\ref{technical.theorem.condition.A}, so necessarily does
not possess \textit{property~A}.
\end{note}

\begin{note}\label{remark.all.multiple.zeroes}
If a function $\varphi\in\mathcal{L-P^*}$ has only multiple real
zeroes, then, in the proofs of
Theorems~\ref{Th.main.theorem.for.finite.zeroes.function.1},
\ref{Th.main.theorem.for.finite.zeroes.function.2}
and~\ref{Th.main.theorem.for.finite.zeroes.function.3}, we can use
Remark~\ref{remark.multiple.zeroes} instead of
Theorem~\ref{corol.new}. Thus, the
inequalities~\eqref{main.result.2} also hold for functions in
$\mathcal{L-P^*}$ with only multiple real zeroes.
\end{note}

In conclusion of this section, we note that, by
Theorem~\ref{technical.theorem.condition.A} and
Remark~\ref{remark.all.multiple.zeroes}, we know two subclasses of
the class $\mathcal{L-P^*}$ such that for functions from these
subclasses the inequalities~\eqref{main.result.2} hold.
G.\,Csordas posed the following problem.

\vspace{2mm}

\textbf{Open problem.} \textit{ Describe all functions in
$\mathcal{L-P^*}$ satisfying the
inequalities~\eqref{main.result.2}.}

\setcounter{equation}{0}

\section{Proof of the Hawaii conjecture.}\label{section.Hawaii}

Now we are in a position to prove the Hawaii conjecture for
functions in the class~$\mathcal{L-P^*}$ by combining
Theorems~\ref{Th.analog.of.PW}
and~\ref{technical.theorem.condition.A}. We also prove
Proposition~\ref{proposition.1} from Introduction not only for
real polynomials but for all functions in~$\mathcal{L-P^*}$.
\begin{theorem}\label{theorem.the.Hawaii.conjecture}
Let $\varphi\in\mathcal{L-P^*}$ and $Q$ be a meromorphic function
defined by~\eqref{main.function.2}. If $\varphi$ has exactly $2m$
nonreal zeroes, counting multiplicities, then

\begin{equation}\label{theorem.the.Hawaii.conjecture.main.inequality}
2m-2m_1\leqslant Z_\mathbb{R}\left(Q\right)\leqslant2m,
\end{equation}
where $2m_1$ is the number of nonreal zeroes of $\varphi'$.
\end{theorem}
\begin{proof}

Indeed, by Theorem~\ref{technical.theorem.condition.A}, the
inequality
\begin{equation}\label{theorem.the.Hawaii.conjecture.proof.0000}
2m-2m_1\leqslant Z_\mathbb{R}\left(Q\right)
\end{equation}
is valid for functions in $\mathcal{L-P^*}$ possessing
\textit{property~A}. However, in the proofs of
Theorems~\ref{finite.zeroes.function.1}--\ref{Th.main.theorem.for.infinite.zeroes.functions}
we used the fact that $\varphi$ possesses \textit{property~A} only
when we applied Theorem~\ref{corol.new}. Moreover, in the proof of
Theorem~\ref{corol.new}, we used \textit{property~A} only for
the~upper bound but not for the lower one. Therefore, the
inequality~\eqref{theorem.the.Hawaii.conjecture.proof.0000} is
valid for all functions in $\mathcal{L-P^*}$.

\vspace{1mm}

Now we prove the upper bound in the number $Z_{\mathbb{R}}(Q)$. If
$Z_\mathbb{R}(Q)=0$, then the logarithmic derivative
$\varphi'/\varphi$ is decreasing on the real axis except its
poles. Therefore, $2m=2m_1$ in this case, so the theorem is true.

\vspace{1mm}

Let $Z_\mathbb{R}(Q)\neq0$. Then according to
Theorem~\ref{Th.analog.of.PW}, there exists a real~$\sigma_0$ such
that the function $\psi_0(z)=e^{-\sigma_0z}\varphi(z)$ possesses
\textrm{property~A} and
$Z_\mathbb{C}(\psi'_0)<Z_\mathbb{C}(\psi_0)$.

\vspace{1mm}

If $Z_\mathbb{R}\left(Q_{1}[\psi_0]\right)\neq0$, then we can
apply Theorem~\ref{Th.analog.of.PW} to $\psi'_0$ to get a real
$\sigma_1$ such that $\psi_1(z)=e^{-\sigma_1z}\psi'_0(z)$
possesses \textit{property~A} and
$Z_\mathbb{C}(\psi'_1)<Z_\mathbb{C}(\psi_1)$. If
$Z_\mathbb{R}\left(Q_{1}[\psi_1]\right)\neq0$, then we can apply
Theorem~\ref{Th.analog.of.PW} to $\psi'_1$ and so on. Thus, we
obtain a sequence of the functions $\psi_0,\psi_1,\psi_2,\ldots$
with \textit{property~A} satisfying the inequalities
$Z_\mathbb{C}(\psi'_j)<Z_\mathbb{C}(\psi_j)$, where
$\psi_j(z)=e^{-\sigma_jz}\psi'_{j-1}(z)$, $i=1,2,\ldots$

Since $\varphi$ has finitely many nonreal zeroes by assumption,
the sequence of the functions $\psi_j$ is finite. That is, there
exists a~nonnegative integer~$l\,(\leqslant m-1)$ such
that\footnote{We notice that a necessary condition for the
equality~\eqref{theorem.the.Hawaii.conjecture.proof.1} to be hold
is $Z_\mathbb{C}(\psi_l)=Z_\mathbb{C}(\psi'_l)$.}
\begin{equation}\label{theorem.the.Hawaii.conjecture.proof.1}
Z_\mathbb{R}\left(Q_{1}[\psi_l]\right)=0,
\end{equation}
while $Z_\mathbb{R}\left(Q[\psi_l]\right)>0$.

\vspace{1mm}

By construction, all the functions $\psi_j$ possess
\textit{property~A}. Consequently, we can apply
Theorem~\ref{technical.theorem.condition.A} to each of them to
obtain
\begin{equation}\label{theorem.the.Hawaii.conjecture.proof.2}
2m^{(j)}-2m_1^{(j)}\leqslant
Z_\mathbb{R}\left(Q[\psi_j]\right)\leqslant
2m^{(j)}-2m_1^{(j)}+Z_\mathbb{R}\left(Q_{1}[\psi_j]\right),\quad
j=0,1,\ldots,l,
\end{equation}
where we use notation $2m^{(j)}=Z_\mathbb{C}(\psi_j)$,
$2m_1^{(j)}=Z_\mathbb{C}(\psi'_j)$.

From~\eqref{theorem.the.Hawaii.conjecture.proof.1}
and~\eqref{theorem.the.Hawaii.conjecture.proof.2} it follows that
\begin{equation}\label{main.condition.1.2}
Z_\mathbb{R}\left(Q[\psi_l]\right)=2m^{(l)}-2m^{(l)}_{1}.
\end{equation}
Further, the
inequalities~\eqref{theorem.the.Hawaii.conjecture.proof.2} imply
\begin{equation}\label{main.result.1.1}
Z_\mathbb{R}\left(Q[\psi_j]\right)-Z_\mathbb{R}\left(Q_{1}[\psi_j]\right)\leqslant
2m^{(j)}-2m_1^{(j)},\quad j=0,1,\ldots,l-1.
\end{equation}
By construction of the functions $\psi_j$ (see
Remark~\ref{remark.new}), we have
\begin{equation*}\label{theorem.the.Hawaii.conjecture.proof.3}
Z_\mathbb{R}\left(Q_{1}[\psi_j]\right)=Z_\mathbb{R}\left(Q[\psi_{j+1}]\right),\quad
2m_1^{(j)}=2m^{(j+1)},\quad j=0,1,\ldots,l-1.
\end{equation*}
Therefore,~\eqref{main.result.1.1} can be rewritten in the form
\begin{equation}\label{theorem.the.Hawaii.conjecture.proof.4}
Z_\mathbb{R}\left(Q[\psi_j]\right)-Z_\mathbb{R}\left(Q[\psi_{j+1}]\right)\leqslant
2m^{(j)}-2m^{(j+1)},\quad j=0,1,\ldots,l-1.
\end{equation}
Summing the
inequalities~\eqref{theorem.the.Hawaii.conjecture.proof.4} for
$j=0,1,\ldots,l-1$, we obtain
\begin{equation*}\label{theorem.the.Hawaii.conjecture.proof.5}
\begin{array}{c}
Z_\mathbb{R}\left(Q[\psi_0]\right)-Z_\mathbb{R}\left(Q[\psi_l]\right)
=Z_\mathbb{R}\left(Q[\psi_0]\right)-Z_\mathbb{R}\left(Q[\psi_1]\right)+
Z_\mathbb{R}\left(Q[\psi_1]\right)-Z_\mathbb{R}\left(Q[\psi_2]\right)+\\
 \\
+Z_\mathbb{R}\left(Q[\psi_2]\right)-Z_\mathbb{R}\left(Q[\psi_3]\right)+\ldots+Z_\mathbb{R}\left(Q[\psi_{l-1}]\right)-Z_\mathbb{R}\left(Q[\psi_l]\right)\leqslant\\
 \\
\leqslant(2m^{(0)}-2m^{(1)})+(2m^{(1)}-2m^{(2)})+\ldots+
(2m^{(l-1)}-2m^{(l)})=2m^{(0)}-2m^{(l)}.
\end{array}
\end{equation*}
This inequality and the equality~\eqref{main.condition.1.2} yield
\begin{equation}\label{theorem.the.Hawaii.conjecture.proof.6}
Z_\mathbb{R}\left(Q[\psi_0]\right)\leqslant2m^{(0)}-2m^{(l)}+Z_\mathbb{R}\left(Q[\psi_l]\right)=
2m^{(0)}-2m^{(l)}_{1}\leqslant2m^{(0)}.
\end{equation}
But by construction of $\psi_0$, we have $Q=Q[\psi_0]$ (see
Remark~\ref{remark.new}) and $2m=2m^{(0)}$. Therefore, the
inequality~\eqref{theorem.the.Hawaii.conjecture.proof.6} is
exactly
\begin{equation*}\label{theorem.the.Hawaii.conjecture.proof.7}
Z_\mathbb{R}\left(Q\right)\leqslant2m,
\end{equation*}
Together with~\eqref{theorem.the.Hawaii.conjecture.proof.0000},
this inequality
implies~\eqref{theorem.the.Hawaii.conjecture.main.inequality}, as
required.
\end{proof}

\section{Bounds on the number of real critical points of the logarithmic derivatives of functions in $U_{2n}^*$ with
$n\geqslant1$}\label{section:main.results.for.U_2n*}

In this chapter, we extend
Theorem~\ref{technical.theorem.condition.A} to the classes
$U_{2n}^*$ with $n\geqslant1$.

As we discussed in Section~\ref{section:log.deriv}, the function
$Q[\varphi](z)$ associated with a function $\varphi$ in
$U_0^*=\mathcal{L-P^*}$ is negative for all sufficiently
large~$z$. But functions $Q[\varphi]$ associated with functions
$\varphi$ in the classes $U_{2n}^*$ with $n\geqslant1$ can be
positive for all sufficiently large~$z$, as we will see below.
Therefore, before we establish an extended version of
Theorem~\ref{technical.theorem.condition.A}, we should prove
corresponding inequalities on the number of real zeroes of
$Q[\varphi]$ on half-intervals free of zeroes of $\varphi$ when
$Q[\varphi]$ is positive at infinity.

First, we revise Theorem~\ref{corol.4} with respective modification
and we represent it in terms that are more convenient for us in
the present section.
\begin{theorem}\label{Theorem.half.infinite.Q.negative}
Let $\varphi\in U_{2n}^*$ with $n\geqslant1$, and let $\varphi$
have the largest zero~$\alpha_L$, and let the function
$Q[\varphi](z)$ be negative for all sufficiently large
positive~$z$. If $\varphi'$ has exactly $r\geqslant1$ zeroes in
the interval~$(\alpha_L,+\infty)$, counting multiplicities, and
$\beta_{S}$ is the minimal one, then
\begin{equation}\label{Theorem.half.infinite.Q.negative.formula}
2\left[\dfrac{r}2\right]\leqslant
Z_{[\beta_{S},+\infty)}(Q)\leqslant
r+Z_{[\beta_{S},+\infty)}(Q_{1}).
\end{equation}
\end{theorem}
The modification we made in this theorem concerns the behaviour of
the function $Q_1(z)$ for large real~$z$. Namely, in the
assumption of Theorem~\ref{Theorem.half.infinite.Q.negative}, the
function $Q_1(z)$ may potentially be either positive or negative
for all large positive~$z$, therefore, the numbers
$Z_{[\beta_{S},+\infty)}(Q)$ and $Z_{[\beta_{S},+\infty)}(Q_1)$
may be of different parities, and we have to add the unity at the
upper bound in the case of odd $r$.

By Theorem~\ref{Theorem.Q.for.U_2n*.at.plus.infinity}, if a
function $\varphi\in U_{2n}^*$ has the largest zero $\alpha_L$
and $Q[\varphi](z)>0$ for all sufficiently large positive~$z$,
then in this case, $\varphi'$~has an even number of zeroes in the
interval $(\alpha_L,+\infty)$. We use this fact in the following
theorem.

\begin{theorem}\label{Theorem.half.infinite.Q.positive}
Let $\varphi\in U_{2n}^*$ with $n\geqslant1$ and let $\varphi$
have the largest zero~$\alpha_L$ and let the function
$Q[\varphi](z)$ be positive for all sufficiently large
positive~$z$. If $\varphi'$ has exactly $r\geqslant2$ zeroes in
the interval~$(\alpha_L,+\infty)$, counting multiplicities, and
$\beta_{S}$ is the minimal one, then
\begin{equation}\label{Theorem.half.infinite.Q.positive.formula}
r-1\leqslant Z_{[\beta_{S},+\infty)}(Q)\leqslant
r+Z_{[\beta_{S},+\infty)}(Q_{1}).
\end{equation}
\end{theorem}
To prove this theorem, one can follow verbatim the proof of
Theorem~\ref{corol.4}.  The only differences we must take into
account are the behaviour of the function $Q(z)$ as $z\to+\infty$
and the fact that, generally speaking, the numbers
$Z_{[\beta_{S},+\infty)}(Q)$ and $Z_{[\beta_{S},+\infty)}(Q_1)$
may be of different parities.

\begin{note}\label{remark.5.1}
Theorems~\ref{Theorem.half.infinite.Q.negative}
and~\ref{Theorem.half.infinite.Q.positive} remain valid with
respective modification in the case when $\varphi$ has
the~smallest~zero~$a_S$ (see Remark~\ref{remark.2.8} and
Theorem~\ref{Theorem.Q.for.U_2n*.at.minus.infinity}).
\end{note}

\vspace{3mm}

Consider the function $\varphi$ of the following form:
\begin{equation}\label{finite.zeroes.function.type.I.positive.a}
\varphi(z)=e^{az^{2n+1}+q(z)}p(z),\quad
a\in\mathbb{R},\,\,a\neq0,\,\,n\in\mathbb{N},
\end{equation}
where $q$ is a real polynomial of degree at most $2n$ and $p$ is a
real polynomial. The~derivative of the logarithmic derivative of
the function $\varphi$ has the form
\begin{equation*}\label{Q.for.functions.of.type.I}
Q[\varphi](z)=Q[p](z)+2na(2n+1)z^{2n-1}+q''(z).
\end{equation*}
It is clear that
\begin{equation*}\label{Q.for.functions.of.type.I.on.infinity}
\sgn a\cdot Q[\varphi](z)\to\pm\infty\quad\text{as}\quad
z\to\pm\infty.
\end{equation*}

Analogously to Lemma~\ref{lemma.2} and~\ref{corol.1}, one can
prove the following simple facts.

\begin{lemma}\label{lemma.5.1}
Let $\varphi$ be of the
form~\eqref{finite.zeroes.function.type.I.positive.a} with $a>0$
$(a<0)$. If $\alpha_L$ is  the largest zero of $\varphi$, then~$Q$
has an~odd~\emph{(}even\emph{)} number of real zeroes in
$(\alpha_L,+\infty)$, counting multiplicities.
\end{lemma}
\begin{lemma}\label{lemma.5.2}
Let $\varphi$ be of the
form~\eqref{finite.zeroes.function.type.I.positive.a} with $a>0$
$(a<0)$. If $\alpha_S$ is  the smallest zero~of~$\varphi$,
then~$Q$ has an~even \emph{(}odd\emph{)} number of real zeroes in
$(-\infty,\alpha_S)$, counting multiplicities.
\end{lemma}

\begin{lemma}\label{lemma.55.3}
Let $\varphi$ be of the
form~\eqref{finite.zeroes.function.type.I.positive.a} with $a>0$ and suppose that $\varphi$ has the
largest zero $\alpha_L$ and $\varphi'$ has exactly $r\geqslant1$
extra zeroes, counting multiplicities, in the
interval~$(\alpha_L,+\infty)$. If $\beta_L$ is the largest zero of
$\varphi'$ in~$(\alpha_L,+\infty)$, then $Q$ has an even
\emph{(}odd\emph{)} number of real zeroes in $(\beta_L,+\infty)$
whenever~$r$~is an~even~\emph{(}odd\emph{)} number.
\end{lemma}
\begin{lemma}\label{lemma.55.4}
Let $\varphi$ be of the
form~\eqref{finite.zeroes.function.type.I.positive.a} with $a<0$ and suppose that $\varphi$ has the
largest zero $\alpha_L$ and $\varphi'$ has exactly $r\geqslant1$
extra zeroes, counting multiplicities, in the
interval~$(\alpha_L,+\infty)$. If $\beta_L$ is the largest zero of
$\varphi'$ in~$(\alpha_L,+\infty)$, then $Q$ has an odd
\emph{(}even\emph{)} number of real zeroes in $(\beta_L,+\infty)$
whenever~$r$~is an~even~\emph{(}odd\emph{)} number.
\end{lemma}
\noindent These lemmata can be proved in the same way as
Lemmata~\ref{lemma.2} and~\ref{corol.1} were.

From Lemmata~\ref{lemma.5.1} and~\ref{lemma.5.2} and
Proposition~\ref{prop.Q.at.alpha} it is easy to determine the
parity of the number of real zeroes of the function $Q$ associated
with functions of the
form~\eqref{finite.zeroes.function.type.I.positive.a}.
\begin{theorem}\label{corol.3.type.I}
If $\varphi$ is of the
form~\eqref{finite.zeroes.function.type.I.positive.a}, then the
function $Q$ associated with $\varphi$ has an odd number of real
zeroes, counting multiplicity.
\end{theorem}

The next lemma and theorems are exact analogues of
Theorems~\ref{corol.real.axis.3},~\ref{corol.real.axis.4}
and~\ref{corol.real.exis.6} and can be established in the same way
with respective modification where we must take into account that
$\sgn a\ Q(z)\to+\infty$ as $z\to+\infty$ for functions of the
form~\eqref{finite.zeroes.function.type.I.positive.a}.

\begin{lemma}\label{corol.real.axis.3.type.I}
Let $\varphi$ be of the
form~\eqref{finite.zeroes.function.type.I.positive.a}. If
$\varphi'(z)\neq0$, $\varphi''(z)\neq0$ for $z\in\mathbb{R}$, then
\begin{equation}\label{corol.real.axis.4.main.formula.type.I}
1\leqslant Z_\mathbb{R}(Q)\leqslant Z_\mathbb{R}(Q_{1}).
\end{equation}
\end{lemma}

\begin{theorem}\label{theorem.real.axis.4.type.I}
Let $\varphi$ be of the
form~\eqref{finite.zeroes.function.type.I.positive.a} and let
$\varphi(z)\neq0$, $\varphi'(z)\neq0$ for $z\in\mathbb{R}$. Then
the inequalities~\eqref{corol.real.axis.4.main.formula.type.I}
hold for $Q[\varphi]$.
\end{theorem}
\begin{theorem}\label{theorem.real.exis.6.type.I}
Let $\varphi$ be of the
form~\eqref{finite.zeroes.function.type.I.positive.a} and have a
unique real zero $\alpha$ such that $\varphi'(z)\neq0$ for
$z\in\mathbb{R}\backslash\{\alpha\}$. If~$\varphi$ possesses
property A at $\alpha$, then the
inequalities~\eqref{corol.real.axis.4.main.formula.type.I} hold.
\end{theorem}

Thus, for functions of the
form~\eqref{finite.zeroes.function.type.I.positive.a} one can
prove the following general theorem about bounds on the number of
real zeroes of the function $Q$.

\begin{theorem}\label{Th.main.theorem.for.finite.zeroes.function.1.general}
Let $\varphi$ be of the
form~\eqref{finite.zeroes.function.type.I.positive.a} and suppose
that $\varphi$ has exactly $2m$ nonreal zeroes. If~$\varphi$
possesses property~A, then
\begin{equation}\label{main.ineualtyg.general.type.I}
\max\{1,2m-2m_{1}+2n-1\}\leqslant
Z_\mathbb{R}\left(Q\right)\leqslant2m-2m_{1}+2n+Z_\mathbb{R}\left(Q_{1}\right),
\end{equation}
where $2m_1=Z_\mathbb{C}(\varphi')$.
\end{theorem}

We omit the proof of this theorem, since it is analogous to the
proof of
Theorem~\ref{Th.main.theorem.for.finite.zeroes.function.1}. We
only have to take into account the possibly different behaviour of the
function $Q(z)$ for sufficiently large positive~$z$ and then use
Theorem~\ref{Theorem.number.of.extra.zeroes.type.I}.

%
%

\vspace{6mm}

Let now the function $\varphi$ be of the form
\begin{equation}\label{finite.zeroes.function.type.II.for.general}
\varphi(z)=e^{az^{2n}+q(z)}p(z),\quad a>0,\,\,\, n\in\mathbb{N},
\end{equation}
where $q$ is a real polynomial of degree at most $2n-1$ and $p$ is
a real polynomial.

The function $Q$ associated  with $\varphi$ is as follows
\begin{equation*}\label{Q.for.functions.of.type.II}
Q[\varphi](z)=Q[p](z)+2na(2n-1)z^{2n-2}+q''(z).
\end{equation*}
It is easy to see that
\begin{equation*}\label{Q.for.functions.of.type.II.on.infinity}
Q[\varphi](z)\to+\infty\quad\text{as}\quad z\to\pm\infty.
\end{equation*}

This immediately implies the following simple results.

\begin{lemma}\label{lemma.5.1.type.II}
Let $\varphi$ be of the
form~\eqref{finite.zeroes.function.type.II.for.general}. If
$\alpha_L$ is  the largest zero of $\varphi$, then~$Q$ has an~odd
number of real zeroes in $(\alpha_L,+\infty)$, counting
multiplicities.
\end{lemma}
\begin{lemma}\label{lemma.5.2.type.II}
Let $\varphi$ be of the
form~\eqref{finite.zeroes.function.type.II.for.general}. If
$\alpha_S$ is  the smallest zero of $\varphi$, then~$Q$ has an~odd
number of real zeroes in $(-\infty,\alpha_S)$, counting
multiplicities.
\end{lemma}
\begin{theorem}\label{corol.3.type.II}
If $\varphi$ is of the
form~\eqref{finite.zeroes.function.type.II.for.general}, then the
function $Q$ associated with $\varphi$ has an~even number of real
zeroes, counting multiplicity.
\end{theorem}

Using these lemmata and theorem, one can establish the following
theorem.
\begin{theorem}\label{Th.main.theorem.for.finite.zeroes.function.2.general}
Let $\varphi$ be of the
form~\eqref{finite.zeroes.function.type.II.for.general} and
suppose that $\varphi$ has exactly $2m$ nonreal zeroes.
If~$\varphi$ possesses property~A, then
\begin{equation}\label{main.ineualtyg.general.type.II}
2m-2m_{1}+2n-2\leqslant
Z_\mathbb{R}\left(Q\right)\leqslant2m-2m_{1}+2n+Z_\mathbb{R}\left(Q_{1}\right),
\end{equation}
where $2m_1=Z_\mathbb{C}(\varphi')$.
\end{theorem}
As above we omit the proof of this theorem, since it is analogous
to the proof of
Theorem~\ref{Th.main.theorem.for.finite.zeroes.function.2}. We
only take into account that $Q(z)>0$ for all sufficiently large
real~$z$ and then use
Theorem~\ref{Theorem.number.of.extra.zeroes.type.II}.

%
%

\vspace{6mm}

If the function $\varphi$ has the form
\begin{equation}\label{finite.zeroes.function.type.III.for.general}
\displaystyle\varphi(z)=e^{-az^{2n+2}+q(z)}p(z),\quad
a>0,\,\,\,n\in\mathbb{N},
\end{equation}
where $q$ is a real polynomial of degree at most $2n+1$ and $p$ is
a real polynomial, then the function $Q[\varphi]$ is as follows
\begin{equation*}\label{Q.for.functions.of.type.III}
Q[\varphi](z)=Q[p](z)-a(2n+1)(2n+2)z^{2n}+q''(z).
\end{equation*}
It is clear that
\begin{equation*}\label{Q.for.functions.of.type.III.on.infinity}
Q[\varphi](z)\to-\infty\quad\text{as}\quad z\to\pm\infty.
\end{equation*}
Therefore,  $Q(z)<0$ for all sufficiently large real~$z$ and all
results of
Sections~\ref{subsection:half-infinite.intervals}--\ref{section:general.theorems}
are true in this case. Consequently, for functions of the
form~\eqref{finite.zeroes.function.type.III.for.general}
Theorem~\ref{Th.main.theorem.for.finite.zeroes.function.3} is
valid and can be proved in the same way. From
Theorems~\ref{Th.main.theorem.for.finite.zeroes.function.3}
and~\ref{Theorem.number.of.extra.zeroes.type.III} we have

\begin{theorem}\label{Th.main.theorem.for.finite.zeroes.function.3.general}
Let $\varphi$ be of the
form~\eqref{finite.zeroes.function.type.III.for.general} and
suppose that $\varphi$ has exactly $2m$ nonreal zeroes.
If~$\varphi$ possesses property~A, then
\begin{equation}\label{main.ineualtyg.general.type.III}
2m-2m_{1}+2n\leqslant
Z_\mathbb{R}\left(Q\right)\leqslant2m-2m_{1}+2n+Z_\mathbb{R}\left(Q_{1}\right),
\end{equation}
where $2m_1=Z_\mathbb{C}(\varphi')$.
\end{theorem}

Next, if a function $\varphi\in U_{2n}^*$ with $n\geqslant1$ has
an infinite number of positive and negative zeroes, then we use
results of Section~\ref{section:finite.intervals} to prove the
following theorem.
\begin{theorem}\label{Th.main.theorem.for.finite.zeroes.function.4.general}
Let $\varphi\in U_{2n}^*$ with $n\geqslant1$ have infinitely many
positive and negative zeroes. If~$\varphi$ possesses property~A,
then
\begin{equation}\label{main.ineualtyg.general.type.IV}
2m-2m_{1}+2n\leqslant
Z_\mathbb{R}\left(Q\right)\leqslant2m-2m_{1}+2n+Z_\mathbb{R}\left(Q_{1}\right),
\end{equation}
where $2m_1=Z_\mathbb{C}(\varphi')$.
\end{theorem}
This theorem follows from
Theorem~\ref{Theorem.number.of.extra.zeroes.type.IV} and the
inequality~\eqref{finite.zeroes.function.1.condition} which can be
proved in the same way as in
Theorem~\ref{Th.main.theorem.for.infinite.zeroes.functions}.

\vspace{4mm}

At last, consider the functions in $U_{2n}^*$ with $n\geqslant1$
that have infinitely many zeroes but only finitely many positive
or negative zeroes.
\begin{theorem}\label{Th.main.theorem.for.finite.zeroes.function.5.general}
Let $\varphi\in U_{2n}^*$ with $n\geqslant1$ have infinitely many
zeroes but only finitely many positive or negative zeroes and
suppose that $\varphi$ has exactly $2m$ nonreal zeroes.
If~$\varphi$ possesses property~A, then
\begin{equation}\label{main.ineualtyg.general.type.V}
2m-2m_{1}+2n-1\leqslant
Z_\mathbb{R}\left(Q\right)\leqslant2m-2m_{1}+2n+2+Z_\mathbb{R}\left(Q_{1}\right),
\end{equation}
\end{theorem}
\begin{proof}
Without loss of generality, we may assume that $\varphi$ has the
largest real zero, say $\alpha_L$. If $\beta<\alpha_L$ is a zero
of $\varphi'$ such that $\varphi(z)\neq0$ and $\varphi'(z)\neq0$
for $z\in(\beta,\alpha_L)$, then applying
Corollary~\ref{biggest.technical.theorem} and
Theorem~\ref{corol.new} to the interval $(-\infty,\beta]$, we have
\begin{equation}\label{Th.main.theorem.for.finite.zeroes.function.5.general.proof.1}
2r\leqslant Z_{(-\infty,\beta]}\left(Q\right)\leqslant
2r+Z_{(-\infty,\beta]}\left(Q_{1}\right),
\end{equation}
where $2r\geqslant0$ is the number of extra zeroes of $\varphi'$
in the interval $(-\infty,\beta]$.

If $\varphi'$ has at least one zero in the interval
$(\alpha_L,+\infty)$, then
by~\eqref{Th.main.theorem.for.finite.zeroes.function.5.general.proof.1}
and by Theorems~\ref{corol.new}
and~\ref{Theorem.half.infinite.Q.negative}--\ref{Theorem.half.infinite.Q.positive},
we obtain
\begin{equation}\label{Th.main.theorem.for.finite.zeroes.function.5.general.proof.2}
E(\varphi')-1\leqslant
Z_\mathbb{R}\left(Q\right)\leqslant E(\varphi')+Z_\mathbb{R}\left(Q_{1}\right).
\end{equation}

Let now $\varphi'$ has no zeroes in the interval
$(\alpha_L,+\infty)$, then by~\eqref{main.work.formula.new.new}
(see Remark~\ref{remark.2.6}) and
by~\eqref{Th.main.theorem.for.finite.zeroes.function.5.general.proof.1}
we obtain
\begin{equation}\label{Th.main.theorem.for.finite.zeroes.function.5.general.proof.2.dub}
E(\varphi')-1\leqslant
Z_\mathbb{R}\left(Q\right)\leqslant E(\varphi')+1+Z_\mathbb{R}\left(Q_{1}\right).
\end{equation}
The
inequalities~\eqref{Th.main.theorem.for.finite.zeroes.function.5.general.proof.2}--\/\eqref{Th.main.theorem.for.finite.zeroes.function.5.general.proof.2.dub}
with Theorem~\ref{Theorem.number.of.extra.zeroes.type.IV}
imply~\eqref{main.ineualtyg.general.type.V}, as required.
\end{proof}

Now combining
Theorems~\ref{Th.main.theorem.for.finite.zeroes.function.1.general}
and~\ref{Th.main.theorem.for.finite.zeroes.function.2.general}--\ref{Th.main.theorem.for.finite.zeroes.function.5.general}
we obtain the general result for functions in $U_{2n}^*$ with
$n\geqslant1$.

\begin{theorem}\label{Theorem.genaral,bound.for.U_2n*}
Let $\varphi$ be in $U_{2n}^*$ with $n\geqslant1$ and suppose that
$\varphi$ has exactly $2m$ nonreal zeroes. If~$\varphi$ possesses
property~A, then
\begin{equation}\label{main.result.generalest}
2m-2m_{1}+2n-2\leqslant
Z_\mathbb{R}\left(Q\right)\leqslant2m-2m_{1}+2n+2+Z_\mathbb{R}\left(Q_{1}\right),
\end{equation}
where $2m_1=Z_\mathbb{C}(\varphi')$.
\end{theorem}

As in Section~\ref{section:general.theorems}, we make the
following remarks.

\begin{note}\label{remark.all.multiple.zeroes.for.U_2n*}
If a function $\varphi\in U_{2n}^*$ with $n\geqslant1$ has only
multiple real zeroes, then in the proof of
Theorem~\ref{main.result.generalest}, we can use
Remark~\ref{remark.multiple.zeroes} instead of
Theorem~\ref{corol.new}. Thus, the
inequalities~\eqref{main.result.generalest} also hold for
functions in $U_{2n}^*$, $n\geqslant1$, with only multiple real
zeroes.
\end{note}

As in the proofs of
Theorems~\ref{Th.main.theorem.for.finite.zeroes.function.1},~\ref{Th.main.theorem.for.finite.zeroes.function.2},
\ref{Th.main.theorem.for.finite.zeroes.function.3}
and~\ref{Th.main.theorem.for.infinite.zeroes.functions}, in the
proof of Theorem~\ref{main.result.generalest}, we do not use the
fact that $\varphi$ possesses \textit{property~A} getting the
lower bound. Therefore in Theorem~~\ref{main.result.generalest},
the lower bound of the~number $Z_\mathbb{R}(Q)$ does not depend on
\textit{property~A} and holds for every function in $U_{2n}^*$
with $n\geqslant1$, and we have our final theorem.

\begin{theorem}\label{Theorem.genaral,bound.for.U_2n*.2}
Let $\varphi$ be in $U_{2n}^*$ with $n\geqslant1$ and suppose that
$\varphi$ has exactly $2m$ nonreal zeroes. Then
\begin{equation*}\label{main.result.generalest.lower}
Z_\mathbb{R}\left(Q\right)\geqslant2m-2m_{1}+2n-2,
\end{equation*}
where $2m_1=Z_\mathbb{C}(\varphi')$.
\end{theorem}

\section*{Acknowledgments}
The author thanks Thomas Craven, George Csordas, Alexandre
Eremenko, Olga Holtz, and Wayne Smith for helpful discussions and
suggestions that promoted simplifications of some proofs and the
referee for helpful critique.

\end{document}